\documentclass[11pt,a4paper]{article}  % default square logo 
\usepackage{amsmath}
\usepackage{amssymb}
\usepackage{amsthm}
\usepackage{bm}
\usepackage{bbm}
\usepackage[usenames,dvipsnames]{color}
\usepackage{enumerate}
\usepackage{fullpage}
\usepackage{hyperref}
\usepackage{thmtools}
\usepackage{thm-restate}
\newcommand{\R}{\mathbb R}
\newcommand{\N}{\mathbb{N}}
\newcommand{\Tcal}{\mathcal T}
\newcommand{\E}[1]{\mathbb{E}\left[#1\right]}
\newcommand{\Realpart}[1]{\textnormal{Re}\left(#1\right)}
\renewcommand{\Re}{\textnormal{Re}}
\newcommand{\Prob}[1]{\mathbb{P}\left(#1\right)}
\newcommand{\Var}[1]{\mathrm{Var}\left(#1\right)}
\newcommand{\Expo}[1]{\textnormal{Exp}\left(#1\right)}
\newcommand{\Ind}{\mathbbm{1}}
\newcommand{\IndEvent}[1]{\mathbbm{1}\left(#1\right)}
\newcommand{\fl}[1]{\lfloor #1 \rfloor}
\fi \iftrue 

\fi
\newcommand{\given}[1]{\ensuremath{\ #1|\ }}
\newcommand{\convdist}{\ensuremath{\overset{d}{\to}}}
\newcommand{\convprob}{\ensuremath{\overset{p}{\to}}}
\newcommand{\eqdist}{\ensuremath{\overset{d}{=}}}

\newtheorem{theorem}{Theorem}

\newtheorem{proposition_restate}[theorem]{Proposition}
\theoremstyle{definition}
\newtheorem{remark}[theorem]{Remark}

\numberwithin{theorem}{section}
\numberwithin{corollary}{section}
\numberwithin{proposition}{section}
\numberwithin{proposition_restate}{section}

\newtheorem{corollary_restate}[theorem]{Corollary}
\numberwithin{corollary_restate}{section}

\newtheorem{lemma_restate}[theorem]{Lemma}
\numberwithin{lemma_restate}{section}

\numberwithin{lemma}{section}
\numberwithin{remark}{section}

\title{The stable trees revisited}

\author{Christina Goldschmidt\footnote{Department of Statistics and Lady Margaret Hall, University of Oxford;\\ \texttt{christina.goldschmidt@stats.ox.ac.uk}} \ and Liam Hill\footnote{Mathematical Institute, University of Oxford; \texttt{liam.hill@magd.ox.ac.uk}}}

\begin{document}
\maketitle
\begin{abstract}
We introduce a new, relatively simple, line-breaking construction of the $\alpha$-stable tree which realises its random finite-dimensional distributions. This is a direct analogue of Aldous' line-breaking construction of the Brownian continuum random tree, which is based on an inhomogeneous Poisson process. Here, we replace the deterministic rate function from the Brownian setting by a random rate process, given by a certain measure-changed $(\alpha-1)$-stable subordinator. Rather than attaching uniformly, the line-segments now connect to locations chosen with probability proportional to the sizes of the jumps of the rate process. 

We also give a new proof of an invariance principle originally due to Duquesne, which states that the family tree of a Bienaym\'e branching process with critical offspring distribution in the domain of attraction of an $\alpha$-stable law (for $\alpha \in (1,2))$, conditioned to have $n$ vertices, converges on rescaling distances appropriately to the $\alpha$-stable tree. Our proof makes use of a discrete line-breaking construction of the branching process tree, which we show converges to our continuous line-breaking construction.
\end{abstract}

\section{Introduction}\label{sec: intro}

\subsection{The $\alpha$-stable trees}

The $\alpha$-stable trees are a family of random $\R$-trees which constitute the scaling limits of the size-conditioned family trees of Bienaym\'e branching processes whose offspring distribution is critical and lies in the domain of attraction of an $\alpha$-stable law, for $\alpha \in (1,2]$. The $2$-stable tree is a (constant multiple of) the Brownian continuum random tree (see Aldous~\cite{aldous1,aldous2,aldous3}); the rest of the family was introduced by Duquesne and Le Gall~\cite{DuquesneLeGall,DuquesneLeGallProbFract}, building on earlier work of Le Gall and Le Jan~\cite{LeGallLeJan}. The focus of this paper is the case $\alpha \in (1,2)$.

Fix $\alpha \in (1,2)$ and let $L = (L_t)_{t \ge 0}$ be a spectrally positive $\alpha$-stable L\'evy process, whose distribution is specified by its Laplace transform
    \[
    \E{\exp(-\lambda L_t)} = \exp(t \lambda^{\alpha}), \quad \lambda \ge 0, t \ge 0.
    \]
The \emph{$\alpha$-stable tree} $(\mathcal{T}_{\alpha}, d_{\alpha}, \mu_{\alpha})$ is a random variable taking values in the space of isometry classes of compact metric measure spaces endowed with the Gromov--Hausdorff--Prokhorov topology. It is usually defined in a somewhat involved fashion using a normalised excursion $\mathbbm{e} = (\mathbbm{e}_t)_{0 \le t \le 1}$ of $L$. Indeed one may derive a continuous process called the \emph{height process} via
    \[
    h_t =  \lim_{\epsilon \to 0} \frac{1}{\epsilon} \int_0^t \IndEvent{\mathbbm{e}_s < \inf_{r \in [s,t]} \mathbbm{e}_r + \epsilon} ds, \quad 0 \le t \le 1,
    \]
where the limit exists in probability. The height process then acts as a contour function for an $\R$-tree. More precisely, we define a pseudo-metric $d^{\circ}_{\alpha}$ on $[0,1]$ via
    \[
    d^{\circ}_{\alpha}(x,y) = h_x + h_y - 2 \inf_{x \wedge y  \le t \le x \vee y} h_t
    \]
and an equivalence relation by declaring $x \sim y$ if $d^{\circ}_{\alpha}(x,y) = 0$. Finally, take $\Tcal_{\alpha}$ to be the quotient space $[0,1]/{\sim}$, which is endowed with the metric $d_{\alpha}$ induced on $\Tcal_{\alpha}$ by $d_{\alpha}^{\circ}$. Take $\mu_{\alpha}$ to be the push-forward of the Lebesgue measure on $[0,1]$ onto $\Tcal_{\alpha}$. (See \cite{LeGall,DuquesneLeGall,duquesnealphalimit} for more details.)

Write $\xi$ for a random variable taking values in $\N_0$. Suppose that $\E{\xi} = 1$ and that
    \[
    \Prob{\xi > k} \sim k^{-\alpha} l(k), \quad \text{as $k \to \infty$,}
    \]
where $\ell:\R^+ \to \R^+$ is a slowly varying function. Then $\xi$ is in the domain of attraction of the $\alpha$-stable law represented by the random variable $L_1$. Let $\xi_1, \xi_2, \ldots$ be i.i.d.\ copies of $\xi$. Then there exists a non-negative increasing sequence $(a_n)_{n \ge 1}$ which is regularly varying of index $1/\alpha$ such that
    \[
    \left(\frac{1}{a_n} \sum_{i=1}^{\fl{nt}}(\xi_i-1), t \ge 0\right) \convdist c (L_t, t \ge 0)
    \]
in the Skorokhod sense for some $c > 0$. By replacing $a_n$ with $ca_n$ for all $n$, we may assume without loss of generality that $c = 1$, and we will work with this specific case throughout. Note that the L\'evy measure of the L\'evy process $L$ is given by $C_{\alpha}x^{-\alpha-1}\,dx, x > 0$, where $C_\alpha = \frac{1}{\Gamma(-\alpha)} = \frac{\alpha(\alpha - 1)}{\Gamma(2-\alpha)}$. 

Let $T^n$ be the rooted ordered tree which represents the genealogy of a branching process with offspring distributed as $\xi$, conditioned to have size $n$. We think of this as a metric space by endowing its vertices $V(T^n)$ with the graph distance, $d^n$, and additionally we let $\mu^n$ be the uniform probability measure on its vertices.  Theorem 3.1 of Duquesne~\cite{duquesnealphalimit} proves the convergence of the height and contour processes of $T^n$ to their continuous analogues (see also Kortchemski~\cite{igorduquesneproof} for a streamlined proof). The following result is a corollary.

\begin{theorem} \label{thm:scalinglimit}
We have
    \[
    \left(V(T^n), \left(\frac{n}{a_n}\right)^{-1} d^n, \mu^n\right) \convdist (\mathcal{T}_{\alpha}, d_{\alpha}, \mu_{\alpha})
    \]
as $n \to \infty$, in the Gromov--Hausdorff--Prokhorov sense.
\end{theorem}

Haas and Miermont~\cite{haasmiermbt} provide an alternative approach to the proof of \autoref{thm:scalinglimit} in the case where $\Prob{\xi = k} \sim c k^{-\alpha-1}$ for some constant $c > 0$ (see their Theorem 8).

Our aims in this paper are twofold. Firstly, we introduce a novel representation of the $\alpha$-stable tree using a line-breaking construction. Secondly, we use the representation in order to provide a new proof of \autoref{thm:scalinglimit}, which does not make use of the height process.

\subsection{Line-breaking construction}
We begin by introducing what we mean by a line-breaking construction. Suppose that $(y_k)_{k \ge 1}$ and $(z_k)_{k \ge 1}$ in $\R^+$ are sequences such that
    \begin{enumerate}
        \item $(y_k)$ is strictly increasing with $y_k \uparrow \infty$;
        \item $z_k < y_k$ for all $k$.
    \end{enumerate}
We call $(y_k)$ the \emph{cut points} and $(z_k)$ the \emph{attachment points}. Informally, we construct a tree by cutting the half-line $\R^+$ into segments $(y_k, y_{k+1}]$ and inductively gluing the open end of $(y_k, y_{k+1}]$ at position $z_k$. Formally, we define a consistent sequence of metrics on (subsets of) $\R^+$: the metric $d_k^{\circ}$ is defined inductively on $[0, y_k]$ (with $d_0^{\circ}$ the trivial metric on $\{0\}$) by
    \[ d_k^{\circ}(u, v) 
    = \begin{cases}
    d^{\circ}_{k-1}(u, v) & u, v \leq y_{k-1}\\
    d^{\circ}_{k-1}(u, z_{k-1}) + |v - y_{k-1}| & u \leq y_{k-1} < v\\
    |v - u| & u, v > y_{k-1}
    \end{cases}
    \]
for $u < v \leq y_k$. Write $(\mathcal{T}(k), d_k)$ for the space $([0,y_k],d_k^{\circ})$. There is a unique metric $d^{\circ}$ on $\R^+$ that agrees with $d^{\circ}_k$ for all $k \ge 1$; let $(\Tcal,d)$ be the completion of the space $(\R^+, d^{\circ})$.\\

By considering the subspaces $[0, t] \subset \mathcal T$ as $t$ varies over $\mathbb R^+$, we can think of this metric space as growing in time: we continuously grow one branch at a time, and at certain discrete time points $(y_k)$ we switch over to a new branch started from an existing point $(z_k)$.

Now consider a random increasing c\`adl\`ag process $\tau = (\tau_t)_{t \ge 0}$ such that $\tau_t < \infty$ almost surely for each $t \ge 0$ and $\tau_t \to \infty$ as $t \to \infty$. 
%For technical reasons, \textcolor{red}{(be precise about what they are)} we will assume that $\E{\tau_t} < \infty$ for all $t \ge 0$ but $\lim_{t \to \infty} \E{\tau_t} = \infty$.
We will construct cut points and attachment points using $\tau$. To begin, conditionally given $\tau$, let $(Y_k)_{k \ge 1}$ be the jump times of an inhomogeneous Poisson process on $\R^+$ with intensity $\tau_t \,dt$, arranged in increasing order; these will be the cut points for our tree. To construct the attachment points, we first sample i.i.d.\ uniform random variables $U_1, U_2, \dots$ on $[0, 1]$, independent of everything else so far. Now we may set
    \[ 
    Z_k = \inf\{ t > 0 : \tau_t > U_k \tau_{Y_k-}\}. 
    \]
(This picks out the time of a jump of $\tau$ in $(0,Y_k)$ chosen with probability proportional to its size. Note that, almost surely, $Y_k$ is not a jump-time of the process $\tau$, and so $\tau_{Y_k} = \tau_{Y_k-}$ a.s.) If the cut and attachment points are generated in this way, we say that $(\mathcal{T},d)$ is \emph{generated by a line-breaking construction with intensity process $\tau$}. 
%In the tree-growing picture we can think of $\lambda_t$ as a total mass process for the tree: it yields a (possibly infinite) measure $d\sigma$ on $\R^+$, which restricts to a finite measure on $\mathcal T_t$ of total mass $\sigma_t$. Then $Z_n$ is a random point in $\mathcal T_{Y_n}$ sampled according to (the normalisation of) the mass measure.  
The case $\tau_t = t$, where $Z_k$ is uniformly distributed on $[0, Y_k]$, is precisely Aldous' line-breaking construction of the Brownian CRT \cite{aldous1}.

We now describe the intensity process $\tau$ that we we will use. First, let $(\sigma_t)_{t \ge 0}$ be an increasing L\'evy process with $\E{\exp(-\lambda \sigma_t)} = \exp(-t\alpha \lambda^{\alpha-1})$. Note that this is (a multiple of) an $(\alpha - 1)$-stable subordinator. Let $p: \R \to \R^+$ be the probability density function of the random variable $L_1$. Then \autoref{cor: extendrecipemg} below states that the process $(M_t)_{t \ge 0}$ defined by
    \[
    M_t = \exp\left(\int_0^t \sigma_s ds\right) \frac{p(-\sigma_t)}{p(0)}, \quad t \ge 0,
    \]
is a mean-1 martingale in the natural filtration of the subordinator. We use it as a Radon--Nikodym derivative in order to define a new process $(\widetilde{\sigma}_t)_{0 \le s \le t}$ via change of measure: for any bounded measurable test-function $F: \mathbb{D}([0,t], \R^+) \to \R$, let
    \begin{equation} \label{eqn:tildesigma}
    \E{F((\widetilde{\sigma}_s)_{s \le t})} = \E{M_t F((\sigma_s)_{s \le t})}.
    \end{equation}
By a standard application of Kolmogorov's extension theorem, we deduce the existence of a process $\widetilde \sigma = (\widetilde{\sigma}_t)_{t \ge 0}$. Moreover, $\widetilde{\sigma}$ is an increasing time-homogeneous Markov process (see \autoref{prop:Markovprop} below).

Now consider the $\R$-tree $(\mathcal{T},d)$ arising from the line-breaking construction with intensity process $\widetilde{\sigma}$. In \autoref{sec: lbalphastab}, we give an independent proof of the compactness of $(\Tcal,d)$. We will also want to endow $(\Tcal,d)$ with its ``uniform'' measure. For each $k \ge 1$, define a measure $\mu_k := \frac{1}{k} \sum_{i=1}^k \delta_{Y_i}$ on $(\Tcal,d)$. In \autoref{thm: weakcvglinebreakmeasure}, we show that $\mu_k$ converges as $k \to \infty$ to a limiting measure $\mu$ on $(\Tcal,d)$. 

In \autoref{sec: discrete}, we give a description of a discrete line-breaking-type construction of a conditioned Bienaym\'e tree $T^n$. In \autoref{sec: cvgdisc}, we then prove a version of \autoref{thm:scalinglimit} with $(\Tcal,d,\mu)$ as the limiting space. We may thus identify this limit as the $\alpha$-stable tree:

%Conditionally on $(\widetilde{\sigma}_t)_{t \ge 0}$, let $(N_t)_{t \ge 0}$ be an inhomogeneous Poisson process of rate $\widetilde{\sigma}_t$ at time $t \ge 0$. Let $0  \le C_1 \le C_2 \le \dots$ be the times of its jumps. Independently, let $U_1, U_2, \ldots$ be i.i.d.\ $\mathrm{U}[0,1]$ random variables. For $k \ge 1$, let $A_k = \inf\{t \ge 0: \widetilde{\sigma}_t > U_k \widetilde{\sigma}_{C_k-}\}$. This picks out the time of a jump of $\widetilde{\sigma}$ in $(0,C_k)$ chosen with probability proportional to its size. Now cut $\R^+$ into line-segments $[0,C_1]$, $(C_1, C_2]$, $(C_2, C_3]$, \ldots. For $k \ge 1$, glue the line-segment $(C_k, C_{k+1}]$ to the point $A_k \in [0, C_k]$. Let $\mathcal{T}(k)$ be the $\R$-tree created by gluing together the first $k$ line-segments, and note that we may think of the sequence $(\mathcal{T}(k))_{k \ge 1}$ as nested. Define $\mathcal{T}$ to be the completion of $\cup_{k \ge 1} \mathcal{T}(k)$, with intrinsic distance $d$.

\begin{theorem}
We have
    \[
    (\mathcal{T},d,\mu) \eqdist (\mathcal{T}_{\alpha}, d_{\alpha},\mu_{\alpha}).
    \]
\end{theorem}

Along the way, we show that $(\mathcal{T}(k))_{k \ge 1}$ realises the \emph{random finite-dimensional distributions} of the $\alpha$-stable tree. That is, for each $k \ge 1$, $\mathcal{T}(k)$ has the same distribution as the subtree of $\Tcal_{\alpha}$ spanned by the root and $k$ vertices chosen independently according to the measure $\mu_{\alpha}$. (As shown by Aldous~\cite{aldous3}, the random finite-dimensional distributions characterise the law of any continuum random tree.)

As we show in \autoref{sec: lbalphastab}, an analogue of the change of measure (\ref{eqn:tildesigma}) works for a much broader class of L\'evy process/subordinator pairs. In view of this, we conjecture that our construction should generalise to other L\'evy trees, at least under certain regularity conditions. 

\subsection{Discussion}

Line-breaking constructions of random $\R$-trees date right back to Aldous' first paper on the Brownian continuum random tree~\cite{aldous1}. That construction was generalised by Aldous and Pitman~\cite{aldouspitmanicrt} to give the family of \emph{inhomogeneous continuum random trees} (ICRTs), which were further studied in \cite{camarripitmanicrt,aldouspitmanicrt2,aldousicrtexplore,abrcompactness,abrfennec}. This family is parametrised by a sequence $\theta = (\theta_i)_{i \geq 0}$ such that $\theta_1 \geq \theta_2 \geq \dots \ge 0$, $\sum_{i \geq 0}\theta_i^2 = 1$ and either $\theta_0 \neq 0$ or $\sum_{i = 1}\theta_i = \infty$.  For $i \ge 1$, let $E_i \sim \mathrm{Exp}(\theta_i)$ and then set
    \[
    \tau_t = \theta_0^2 t + \sum_{i=1}^{\infty} \theta_i \IndEvent{E_i \le t}, \quad t \ge 0.
    \]
Then use $(\tau_t)_{t \ge 0}$ as the intensity process in a line-breaking construction. (This succinct formulation, which is due to Blanc-Renaudie~\cite{abrcompactness}, is equivalent to the  original formulation of \cite{aldouspitmanicrt}.) The ICRTs were shown by Camarri and Pitman~\cite{camarripitmanicrt} to be the scaling limits of a family of discrete models called $p$-trees. We will find it helpful in the sequel to relax the condition $\sum_{i \ge 0} \theta_i^2 = 1$ to the condition $\|\theta\|_2 < \infty$, the effect being simply a scaling of all edge-lengths in the tree by $\|\theta\|_2$. Let us write $(\Tcal^{\theta},d^{\theta},\mu^{\theta})$ for the ICRT with parameter $\theta$. Blanc-Renaudie~\cite{abrcompactness} proves that $(\Tcal^{\theta},d^{\theta})$ is a.s.\ compact if and only if
    \[ 
    \int_{R}^\infty \frac{1}{t\E{\tau_t}}\,dt < \infty \text{ for some } R > 0.
    \]

Other line-breaking constructions with general lengths and uniform attachment mechanism have been studied in \cite{AminiDevroyeGriffithsOlver,CurienHaas,Haas}.

A first, more complicated, line-breaking construction for the $\alpha$-stable trees was given by the first author with B\'en\'edicte Haas in \cite{lbcstable} in 2015. This construction is not based on an inhomogeneous Poisson point process, but rather on an increasing Markov chain, $(M_k)_{k \ge 1}$. In order to describe it, we need first to introduce the two-parameter family of \emph{generalised Mittag-Leffler distributions}, $\mathrm{ML}(\theta,\beta)$, for $\beta \in (0,1)$ and $\theta > -\beta$. If $X \sim \mathrm{ML}(\theta,\beta)$ then $X$ takes values in $\R^+$ and has distribution determined by its moments, 
    \[
    \E{X^n} = \frac{\Gamma(\theta)\Gamma(\theta/\beta+n)}{\Gamma(\theta/\beta) \Gamma(\theta + n\beta)}, \quad n \ge 1.
    \]
The Markov chain $(M_k)_{k \ge 1}$ has marginal distributions $M_k \sim \mathrm{ML}(1-1/\alpha,k-1/\alpha)$ and is most easily described via its \emph{backward} transitions, which are such that, for $k \ge 1$, $M_k = M_{k+1} \cdot \beta_k$, where $\beta_k$ is independent of $M_{k+1}$ and 
    \[
    \beta_k \sim \mathrm{Beta}\left(\frac{(k+1)\alpha -2}{\alpha-1}, \frac{1}{\alpha-1} \right).
    \]
These two properties turn out to completely characterise its distribution.

Now let $B_2, B_3, \ldots$ be i.i.d.\ $\mathrm{Beta}(1,(2-\alpha)/(\alpha-1))$ random variables, independent of $(M_k)_{k \ge 1}$. The line-breaking construction generates an increasing sequence of trees $(\widetilde{\Tcal}(k))_{k \ge 1}$ as we now describe.

Initialise by setting $L_1 = M_1$, and letting $\widetilde{\Tcal}(1)$ be the tree consisting of a line-segment of length $L_1$ (rooted at one end). 

Then, for $k \ge 1$, given $\widetilde{\Tcal}(k)$, at step $k+1$, we attach a line-segment of length $(M_{k+1} - M_k) \cdot B_{k+1}$ onto $\widetilde{\Tcal}(k)$ at a point to be determined as follows:
\begin{itemize}
    \item We choose to attach at a pre-existing branch-point of $\widetilde{T}(k)$ with probability $(M_k-L_k)/M_k$. Conditionally on being in this case, we pick such a branch-point of degree $d \ge 3$ with probability proportional to $d-1-\alpha$.
    \item Otherwise, with probability $L_k/M_k$, we attach at a point chosen  according to the normalised Lebesgue measure on $\widetilde{\Tcal}(k)$.
\end{itemize}
(A second version of the construction was also given in \cite{lbcstable} but since it does not apparently reveal any extra information about the stable tree, we do not describe it here. See also Rembart and Winkel~\cite{RembartWinkel} for an alternative perspective.)

In \cite{lbcstable}, it is proved that $(\widetilde{\Tcal}(k))_{k \ge 1}$ realises the random finite-dimensional distributions of the stable tree. Therefore, the construction makes manifest the distributions of various distances in the tree (see Proposition~1.6 of \cite{lbcstable} for more details). We are able to recover these facts using our new construction, although it requires some work, for example:

\begin{restatable}{proposition_restate}{firstcuttimelaw}\label{prop: firstcuttimelaw}
    The first cut time $Y_1$ of the line-breaking construction with intensity $\widetilde \sigma$ is such that $\alpha Y_1 \sim \mathrm{ML}\left(1-1/\alpha, 1 - 1/\alpha\right)$.
\end{restatable}
The proof is long but fairly straightforward, so we defer it to the appendix. 

The use of line-breaking constructions has recently acquired new impetus, beginning with work of Addario-Berry, Blanc-Renaudie, Donderwinkel, Maazoun and Martin~\cite{surveypaper}, who gave a survey of analogous \emph{discrete} constructions for uniform random trees with a given degree sequence. We describe a particular instance of this in detail in \autoref{sec: discrete} below, and with randomised degrees it will play a key role in our proof of \autoref{thm:scalinglimit}. (It has also recently been used to study various aspects of random trees with given degrees in \cite{abrfixeddegrees,ABDrandomtrees,snakes,ABDKcriticaltrees}.)

It is possible to consider mixtures of ICRTs by randomising the parameter $\theta$: that is, we first sample $\theta$ according to some distribution and then, conditionally given $\theta$, we sample an ICRT with parameter $\theta$. (In this context, we do not wish to constrain $\|\theta\|_2$ to take the value 1.) Aldous, Pitman and Miermont~\cite{aldousicrtexplore} conjectured that the L\'evy trees of Duquesne and Le Gall~\cite{DuquesneLeGall} (which include the stable trees) are mixtures of ICRTs. This has recently been shown by Blanc-Renaudie in \cite{abrfixeddegrees} for L\'evy trees which are Gromov--Prokhorov limits of Bienaym\'e trees.  The corresponding statement in the case of the stable trees was proved by Wang~\cite{minmin} in the following more refined form, which identifies $\theta$ as the sequence of jumps of an $\alpha$-stable excursion.

\begin{theorem}[Wang~\cite{minmin}]
Fix $\alpha \in (1,2)$ and let $(\Theta_i, i \ge 1)$ be the ordered sequence of jumps of the normalised $\alpha$-stable excursion $\mathbbm{e}$. Then 
    \[
    (\Tcal_{\alpha},d_{\alpha},\mu_{\alpha}) \eqdist (\Tcal^{\Theta},d^{\Theta},\mu^{\Theta}).
    \]
\end{theorem}

As we show below in \autoref{prop: sigmatildeidentified}, if $(\Delta_i, i \ge 1)$ is the ordered sequence of jumps of $\widetilde{\sigma}$ then
    \[
    (\Delta_i, i \ge 1) \eqdist (\Theta_i, i \ge 1)
    \]
and, moreover,
    \[
    (\widetilde{\sigma}_t)_{t \ge 0} \eqdist \left(\sum_{i \ge 1} \Theta_i \IndEvent{E_i \le t}\right)_{t \ge 0},
    \]
where, conditionally on $\Theta$, $E_i \sim \mathrm{Exp}(\Theta_i)$ independently for all $i \ge 1$. So our approach and Wang's give two different perspectives on the same construction, although this does not seem obvious a priori.

\section{A line-breaking construction of the $\alpha$-stable tree} \label{sec: lbalphastab}

In this section, we first establish the validity of a more general version of the change of measure (\ref{eqn:tildesigma}). We then prove various properties of $\widetilde{\sigma}$ and use them to show that the line-breaking construction with intensity process $\widetilde \sigma$ yields a compact $\R$-tree. Finally, we study the mass measure on the tree.

\subsection{A measure-changed subordinator}

\begin{restatable}{theorem}{recipemg}\label{thm: recipemg}
    Let $a, b \geq 0$, and let $\nu$ be a measure on $\R^+$ such that $\int_{\R^+}(x \wedge x^2)\,\nu(dx) < \infty$. Let $L = (L_t)_{t \ge 0}$ be a spectrally positive L\'evy process on $\R$ with characteristic exponent
        \[ 
        \Psi_L(\lambda) = ai\lambda  + \frac 12 b^2\lambda^2 - \int_0^\infty (e^{i\lambda x} - 1 - i\lambda x)\,\nu(dx),\]
    so that $\E{e^{i\lambda L_t}} = e^{-t \Psi_L(\lambda)}$. Let $(\sigma_t)_{t \ge 0}$ be a (killed) subordinator, not necessarily defined on the same probability space, with characteristic exponent
        \[ \Psi_\sigma(\lambda) = a - b^2i\lambda - \int_0^\infty (e^{i\lambda x} - 1)\,x\nu(dx) = -i\Psi_L'(\lambda).\]
    Assume that
    \begin{enumerate}[(i)]
        \item $\int_0^\infty (e^{ux} - 1 - ux)\,\nu(dx) < \infty$ holds for all $u \in \R$;
        \item $\Realpart{\Psi_L(\lambda)} \gg |\lambda|$ as $|\lambda| \to \infty$.
    \end{enumerate}
    Then $L_1$ has a density $p$, and the process $M^{\nu, a, b}_t := \exp(\int_0^t \sigma_s\,ds)p(-\sigma_t)\IndEvent{\sigma_t < \infty}$ is a non-negative martingale (in the natural filtration associated to $\sigma$).
\end{restatable}

The proof of \autoref{thm: recipemg} is quite technical, and as such we defer it to the Appendix. For now, we give an outline of some of the key ideas.

\begin{proof}[Proof sketch]
    We will only illustrate the case $a = b = 0$ as it reduces notation (but exactly the same argument still works in the general case --- we just have more terms to keep track of). Condition (ii) implies that $e^{-\Psi_L(\lambda)}$ is in $L^1$, so by Fourier inversion $L_1$ has a density, namely
        \[ p(x) = \frac{1}{2\pi}\int_\R e^{-i\lambda x}e^{-\Psi_L(\lambda)}\,d\lambda. \]

    In what follows we write $M_t = M_t^{\nu, a, b}$ for brevity. Let $(\mathcal F_t)$ be the filtration associated to $\sigma$. We will reduce the martingale property to an equivalent statement about (unconditional) expectations. Let $t,s \geq 0$ and define $\sigma_r' = \sigma_{s+r}-\sigma_s$ for $r \geq 0$: then the process $\sigma'$ has the same law as $\sigma$ and is independent of $\mathcal F_s$, and we may compute
    \begin{align*}
        \E{M_{t+s} \,|\, \mathcal F_s} &= M_s \E{\exp\left(\int_s^{t+s}\sigma_r\,dr\right)\frac{p(-\sigma_{t+s})}{p(-\sigma_s)} \,\Bigg|\, \mathcal F_s}\\
        &= M_s \E{\exp\left(\int_0^t (\sigma_s + \sigma_r')\,dr\right)\frac{p(-\sigma_s - \sigma_t')}{p(-\sigma_s)} \,\Bigg|\, \mathcal F_s}.
    \end{align*}
    So it is equivalent to show that
        \begin{equation} \label{eqn:martingalekey} 
        \E{\exp\left(\int_0^t (c + \sigma_r)\,dr\right)p(-c-\sigma_t)} = p(-c), \quad \text{ for all } c \geq 0.
        \end{equation}
    The Fourier inversion formula shows that
        \[ p(-c) = \frac{1}{2\pi} \int_\R \exp\left(G(\lambda)\right)\,d\lambda, \]
    where $G(\lambda) = ic\lambda - \Psi_L(\lambda)$. By condition (i), the characteristic exponent $\Psi_L$ has an analytic continuation to all of $\mathbb C$ (obtained simply by replacing $\lambda$ with a complex number in its Lévy--Khintchine decomposition), and so $G$ has an analytic continuation too. By applying Fourier inversion to the density in the expectation, and then applying Fubini's theorem, we obtain
        \[ \E{\exp\left(\int_0^t (c + \sigma_r)\,dr\right)p(-c-\sigma_t)} = \frac{1}{2\pi}\int_\R \E{\exp\left(i\lambda \sigma_t + \int_0^t \sigma_r\,dr\right)}e^{c(t+i\lambda)-\Psi_L(\lambda)}\,d\lambda.\]
    We can now handle the expectation using Campbell's formula: indeed, $\sigma$ is a pure jump subordinator, so we can find a Poisson process on $\R^+ \times \R^+$, of intensity $ds \otimes x\nu(dx)$, for which the set of atoms $\Pi \subset \R^+\times \R^+$ satisfies
        \[ \sigma_t = \sum_{\substack{(s, x) \in \Pi\\s\leq t}}x, \]
    and so we have
        \[ \E{\exp\left(i\lambda \sigma_t + \int_0^t \sigma_r\,dr\right)} = \E{\exp\left(\sum_{\substack{(s, x) \in \Pi\\ s\leq t}}(i\lambda x + (t-s)x)\right)}. \]
    Then Campbell's formula yields
    \begin{align*}
        \E{\exp\left(\sum_{\substack{(s, x) \in \Pi\\ s\leq t}}(i\lambda x + (t-s)x)\right)} &= \exp\left(\int_0^\infty \int_0^t (e^{i\lambda x + (t-s)x}-1)\,ds\,x\nu(dx)\right)\\
        &= \exp\left(\int_0^\infty \left(e^{i\lambda x}\left(\frac{e^{tx}-1}{x}\right)-t\right)x\nu(dx)\right)\\
        &= \exp\left(\int_0^\infty \left(e^{(t+i\lambda)x} - e^{i\lambda x} - tx\right)\,\nu(dx)\right),
    \end{align*}
    which we recognise as $\exp\big(\Psi_L(\lambda) - \Psi_L(\lambda - it)\big)$. Hence,
    \begin{align*}
        & \E{\exp\left(\int_0^t (c + \sigma_r)\,dr\right)p(-c-\sigma_t)} \\
        & \qquad = \frac{1}{2\pi}\int_\R \E{\exp\left(i\lambda \sigma_t + \int_0^t \sigma_r\,dr\right)}e^{c(t+i\lambda)-\Psi_L(\lambda)}\,d\lambda\\
        & \qquad = \frac{1}{2\pi}\int_\R \! \exp\Big(\Psi_L(\lambda) - \Psi_L(\lambda - it) + c(t+i\lambda) - \Psi_L(\lambda)\Big)\,d\lambda\\
        & \qquad = \frac{1}{2\pi}\int_\R \exp\Big(ic(\lambda - it) - \Psi_L(\lambda - it)\Big)\,d\lambda\\
        & \qquad = \frac{1}{2\pi} \int_\R \exp(G(\lambda - it))\,d\lambda.
    \end{align*}
    Thus it remains only to show that
        \[ \int_\R \exp(G(\lambda - it))\,d\lambda = \int_\R \exp(G(\lambda))\,d\lambda, \]
    and we can do this by a contour integral (using conditions (i) and (ii) to ensure sufficient decay).
\end{proof}

In our case of interest we have $a = b = 0$ and $\nu(dx) = Cx^{-\alpha-1}\,dx$. It is easy to compute $\Realpart{\Psi_L(\lambda)} = \Theta(|\lambda|^\alpha)$ and so the second condition of \autoref{thm: recipemg} holds, but the first does not; thus we cannot apply our theorem directly --- we instead deduce it by taking limits.

\begin{corollary_restate}\label{cor: extendrecipemg}
    The conclusion of \autoref{thm: recipemg} also holds in the case $a = b = 0$, $\nu(dx) = C_\alpha x^{-\alpha-1}\,dx$.
\end{corollary_restate}
\begin{proof}
    For each $\varepsilon \in (0, 1]$ define a measure $\nu_\varepsilon$ on $(0, \infty)$ by $\nu_\varepsilon(dx) = C_{\alpha}x^{-\alpha-1}e^{-\varepsilon x^2}\,dx$. We apply \autoref{thm: recipemg} in the case $a = b = 0, \nu = \nu_\varepsilon$: let $L^\varepsilon$ and $\sigma^\varepsilon$ denote the relevant Lévy process and subordinator.
    
    First we must check that the two conditions hold. Condition (i) holds straightforwardly by splitting the integral into $(0, 1)$ and $[1, \infty)$ and using standard estimates. For Condition (ii), we may compute
    \begin{align*}
        \Realpart{\Psi_{L^\varepsilon}} &= \Realpart{-\int_0^\infty (e^{i\lambda x} - 1 - i\lambda x)\,\nu_\varepsilon(dx)}\\
        &= \int_0^\infty (1-\cos(\lambda x))\nu_\varepsilon(\,dx)\\
        &= C_{\alpha}\int_0^\infty (1-\cos(|\lambda|x))x^{-\alpha-1}e^{-\varepsilon x^2}\,dx.
    \end{align*}
    Now a substitution $x = |\lambda|u$ yields that this is equal to
    \begin{align*}
         C_{\alpha}|\lambda|^\alpha\int_0^\infty (1-\cos(u))u^{-\alpha-1}e^{-\varepsilon u^2 / |\lambda|^2}\,du.
    \end{align*}
    By monotone convergence, the integral in the above expression converges to a finite quantity $I_\alpha = \int_0^\infty (1-\cos(u))u^{-\alpha-1}\,du$ as $|\lambda| \to \infty$, and so $\Realpart{\Psi_{L^\varepsilon}(\lambda)} \sim C_\alpha I_\alpha |\lambda|^\alpha$ as $|\lambda| \to \infty$. In particular, $\Realpart{\Psi_{L^\varepsilon}} \gg |\lambda|$.

    Hence, letting $p_\varepsilon$ be the density of $L_1^\varepsilon$, we see that $M^\varepsilon_t = \exp(\int_0^t \sigma_s^\varepsilon\,ds)p_\varepsilon(-\sigma_t^\varepsilon)$ is a martingale with mean $p_\varepsilon(0)$. As seen in the sketch proof of \autoref{thm: recipemg}, it is easier to work with the condition
        \begin{equation*} \E{\exp\left(\int_0^t (c+\sigma_s^\varepsilon)\,ds\right)p_\varepsilon(-c-\sigma_t^\varepsilon)} = p_\varepsilon(-c) \quad \text{ for all } c \geq 0, 
        \end{equation*}
    and deduce the analogous statement for $\sigma$ by taking limits. In what follows, we consider $c \geq 0$ fixed. To handle the expectation on the left-hand side, it is helpful to consider a coupling of $\sigma^\varepsilon$ and $\sigma$. Since $\nu_\varepsilon \ll \nu$ with Radon--Nikodym derivative $e^{-\varepsilon x^2}$ taking values in $[0, 1]$, we can achieve this by thinning: consider a Poisson point process $\Pi \subset \R^+ \times \R^+ \times [0, 1]$ with intensity measure $\nu(dx) \otimes dt \otimes du$ and atoms denoted by $(x, t, u)$. Then we set
        \[ \sigma_t = \sum_{\substack{(x, s, u) \in \Pi:\\ s \leq t}} x \quad \text{ and } \quad \sigma_t^\varepsilon = \sum_{\substack{(x, s, u) \in \Pi :\\s \leq t\\u \leq e^{-\varepsilon x^2}}}x.\]
    This coupling arranges that $\sigma_t^\varepsilon \uparrow \sigma_t$ for all $t$, almost surely. Thus we have
        \[ \exp\left(\int_0^t(c+\sigma_s^\varepsilon)\,ds \right) \to \exp\left(\int_0^t (c+\sigma_s)\,ds\right) \quad \text{ a.s.}\]

    Now notice that $p_\varepsilon \to p$ uniformly: indeed, for any $x \in \R$, we have
        \begin{align*}
        |p_\varepsilon(x) - p(x)| &= \left|\frac{1}{2\pi}\int_\R e^{-i\lambda x}\left(\exp(-\Psi_{L^\varepsilon}(\lambda))-\exp(-\Psi_L(\lambda))\right)\,d\lambda\right|\\
        &\leq \frac{1}{2\pi}\int_\R \left|e^{-\Psi_{L^\varepsilon}(\lambda)}-e^{-\Psi_L(\lambda)}\right|\,d\lambda,
        \end{align*}
    a bound which does not depend on $x$. The integrand tends to 0 as $\varepsilon \downarrow 0$ and, by our calculations above, is bounded by $2e^{-\delta|\lambda|^\alpha\IndEvent{|\lambda| \geq 1}}$ for some $\delta > 0$, so by dominated convergence this integral tends to 0. Now we may bound
    \begin{align*}
        \left|p_\varepsilon(-c-\sigma_t^\varepsilon) - p(-c-\sigma_t)\right| \leq \|p_\varepsilon - p\|_\infty + |p(-c-\sigma_t) - p(-c-\sigma_t^\varepsilon)|,
    \end{align*}
    where we know that the first term tends to 0 as $\varepsilon \to 0$. Note that $-c - \sigma_t^\varepsilon \to -c - \sigma_t$ a.s., and $p$ is continuous at $-c-\sigma_t$, so the second term tends to zero a.s. Hence $p_\varepsilon(-c-\sigma_t^\varepsilon) \to p(-c-\sigma_t)$ a.s. Combining with the convergence of the exponential terms, we see
        \[ \exp\left(\int_0^t (c+\sigma_s^\varepsilon)\,ds\right)p_\varepsilon(-c-\sigma_t^\varepsilon) \to \exp\left(\int_0^t(c+\sigma_s)\,ds\right)p(-c-\sigma_t) \quad \text{ a.s.}\]
    The proof will be complete if we can deduce that the corresponding expectations converge.

    To do this, we will give a uniform bound on all of the expressions above and apply the bounded convergence theorem. To this end we introduce a coupling between $L^\varepsilon$ and $L$: let $\mu^\varepsilon = \nu - \nu^\varepsilon$, and observe that $\mu^\varepsilon$ is a (positive) measure with $\mu^\varepsilon(\R^+) < \infty$ and $\int_0^\infty x\mu^\varepsilon(dx) < \infty$. We can construct a Lévy process $R^\varepsilon$ independent of $L^\varepsilon$ with characteristic exponent
        \[ \Psi_{R^\varepsilon}(\lambda) = i\lambda \int_0^\infty x\mu^\varepsilon(dx) - \int_0^\infty (e^{i\lambda x} - 1)\mu^\varepsilon(dx), \]
    and with the property that $L^\varepsilon + R^\varepsilon \overset{d}{=} L$. Letting $\eta$ be the law of $R^\varepsilon_1$, we have
        \[ p(x) = \int_\R p_\varepsilon(x-y)\,\eta(dy), \quad \text{ for all } x \in \R. \]
    However, $R^\varepsilon$ is the sum of a deterministic drift and a pure jump process, and the intensity measure of the jump process is $\mu^\varepsilon$, which is finite. Hence, $\eta$ has an atom at $-a_\varepsilon = -\int_0^\infty x\mu_\varepsilon(dx)$ of mass at least $e^{-\mu^\varepsilon(\R^+)}$. It follows that
        \[ p(x) \geq p_\varepsilon(x+a_\varepsilon)e^{-\mu^\varepsilon(\R^+)}. \]
    Using that $\mu^\varepsilon(\R^+) \leq \mu^1(\R^+)$ for $\varepsilon \leq 1$, we see there is a constant $M = e^{\mu^1(\R^+)}$ such that $p_\varepsilon(-x) \leq Mp(-x-a_\varepsilon)$ for all $\varepsilon$ and all $x$.

    Now by Theorem 14.35 of Sato~\cite{satobook}, there exist constants $C_1, C_2 > 0$ and $\beta = \frac{2 - \alpha}{2\alpha - 2}$ such that
        \begin{equation}\label{eqn:satoasymptotic} p(-x) \sim C_1 x^{\beta} \exp\left(-C_2x^{\frac{\alpha}{\alpha-1}}\right) \quad \text{ as } x \to \infty, \end{equation}
    and so in particular we can find constants $\widetilde C_1, \widetilde C_2 > 0$ such that
        \[ p(-x) \leq \widetilde C_1 \exp\left(-\widetilde C_2 x^{\frac{\alpha}{\alpha-1}}\right) \quad \text{ for all } x > 0.\]
    Hence, we almost surely have the bound
    \begin{align*}
        \exp\left(\int_0^t (c+\sigma_s^\varepsilon)\,ds\right)p_\varepsilon(-c-\sigma_t^\varepsilon) &\leq \exp\left(t(c+\sigma_t^\varepsilon)\right)p_\varepsilon(-c-\sigma_t^\varepsilon)\\
        &\leq M\exp\left(t(c+\sigma_t^\varepsilon)\right)p(-c-\sigma_t^\varepsilon-a_\varepsilon)\\
        &\leq \sup_{x \geq 0}\left(M e^{tx}p(-x-a_\varepsilon)\right)\\
        &\leq \sup_{x \geq 0}\left(M \widetilde C_1 \exp\left(tx - \widetilde C_2(x+a_\varepsilon)^{\frac{\alpha}{\alpha-1}}\right)\right)\\
        &\leq \sup_{x \geq 0}\left(M \widetilde C_1 \exp\left(tx - \widetilde C_2x^{\frac{\alpha}{\alpha-1}}\right)\right),
    \end{align*}
    which is a finite constant, so we are done by bounded convergence.
\end{proof}

Recall the definition of the measure-changed subordinator $\widetilde{\sigma}$ from (\ref{eqn:tildesigma}).

\begin{proposition_restate} \label{prop:Markovprop}
    $(\widetilde \sigma_t)_{t \ge 0}$ is a time-homogeneous increasing Markov process with law such that for any suitable test functions $F$, $G$ and any $t_1, t_2 > 0$ we have
        \begin{multline*} \E{F(\widetilde \sigma_s, 0 \leq s \leq t_1)G(\widetilde \sigma_{t_1 + r}, 0 \leq r \leq t_2)}
        \\ = \E{F(\widetilde{\sigma}_s, 0 \le s \le t_1)  \exp \left(\int_0^{t_2} (\widetilde{\sigma}_{t_1} + \sigma'_u) du \right) \frac{p(-\widetilde{\sigma}_{t_1} - \sigma'_{t_2})}{p(-\widetilde{\sigma}_{t_1})} G(\widetilde{\sigma}_{t_1} + \sigma'_{r}, 0 \le r \le t_2)}, \end{multline*}
    where $\sigma'$ is a copy of the subordinator $\sigma$ independent of $\widetilde \sigma$.
\end{proposition_restate}

\begin{proof}
    The increasing property is inherited from that of the subordinator $\sigma$.
    
    Now fix $t_1, t_2 \ge 0$ and suppose that $F$ and $G$ are suitable test-functions. Let $\sigma'$ be an independent copy of the subordinator $\sigma$, which is also independent of $\widetilde{\sigma}$. Then, using the independence of the increments of $\sigma$, we have that 
    \begin{align*}
        & \E{F(\widetilde{\sigma}_s, 0 \le s \le t_1) G(\widetilde{\sigma}_{t_1+r}, 0 \le r \le t_2)} \\
        & = \E{\exp \left( \int_0^{t_1+t_2} \sigma_u du\right) \frac{p(-\sigma_{t_1+t_2})}{p(0)} F(\sigma_s, 0 \le s \le t_1) G(\sigma_{t_1+r}, 0 \le r \le t_2)} \\
        & = \mathbb{E} \left[ \exp \left( \int_0^{t_1} \sigma_u du \right) \frac{p(-\sigma_{t_1})}{p(0)} F(\sigma_s, 0 \le s \le t_1) \right. \\
        & \qquad \quad \cdot \left.\exp \left(\int_0^{t_2} (\sigma_{t_1} + \sigma'_u) du \right) \frac{p(-\sigma_{t_1} - \sigma'_{t_2})}{p(-\sigma_{t_1})} G(\sigma_{t_1} + \sigma'_{r}, 0 \le r \le t_2) \right] \\
        & = \E{F(\widetilde{\sigma}_s, 0 \le s \le t_1)  \exp \left(\int_0^{t_2} (\widetilde{\sigma}_{t_1} + \sigma'_u) du \right) \frac{p(-\widetilde{\sigma}_{t_1} - \sigma'_{t_2})}{p(-\widetilde{\sigma}_{t_1})} G(\widetilde{\sigma}_{t_1} + \sigma'_{r}, 0 \le r \le t_2)}.
    \end{align*}
    It follows that $(\widetilde{\sigma}_{t_1+r}, 0 \le r \le t_2)$ depends on $(\widetilde{\sigma}_s, 0 \le s \le t_1)$ only through the value of $\widetilde{\sigma}_{t_1}$, and that its distribution does not depend on $t_1$. The result follows.
\end{proof}

We observe that the Radon--Nikodym derivative is not uniformly integrable: indeed, one can show that $\lim_{t \to \infty} M_t = 0$ a.s. Thus the absolute continuity only holds on compact time intervals and, in particular, almost sure late-time properties of $\sigma$ do not carry over to $\widetilde \sigma$. In fact, we will see in the next section that the asymptotics of $\sigma_t$ and $\widetilde \sigma_t$ as $t \to \infty$ are very different.

\subsection{Properties of $\widetilde \sigma_t$}
We aim to show that the line-breaking construction with intensity $\widetilde \sigma_t$ yields a compact $\R$-tree: to do this it will be important to study the moments. We can compute the Laplace transform:

\begin{proposition_restate}\label{prop: sigmalaplace}
    Define $G(z) = \int_0^\infty (e^{zx}-1-zx)C_\alpha x^{-\alpha-1}\,dx$ for $\Re(z) \leq 0$. Then the Laplace transform of $\widetilde \sigma_t$ is given by
        \[ \E{e^{-\lambda \widetilde \sigma_t}} = \frac{1}{2\pi p(0)}\int_\R \exp\left(G(iu-\lambda)-G(iu-t-\lambda)+G(iu-t)\right)\,du, \quad \Re(\lambda) \geq 0.\]
\end{proposition_restate}
The proof of this identity is obtained by an adaptation of the proof of \autoref{thm: recipemg} given in the Appendix. The following are useful consequences:

\begin{corollary_restate}\label{cor: sigmean}
    As $t \to \infty$ we have $\E{\widetilde \sigma_t} \sim \alpha t^{\alpha-1}$. 
\end{corollary_restate}
\begin{proof}
    We differentiate the Laplace transform identity in \autoref{prop: sigmalaplace} at $\lambda = 0$. By considering the same sequence of approximations used in the proof of \autoref{cor: extendrecipemg} we see that one can freely exchange derivatives with the Fourier integral, thus obtaining
        \[ \E{-\widetilde \sigma_t} = \frac{1}{2\pi p(0)}\int_\R \big(G'(iu-t)-G'(iu)\big)\exp\big(G(iu)\big)\,du.\]
    Recall that $G(iu) = -\Psi_L(u)$. By differentiating this (again valid by considering the same approximations) we see
        \[ G'(iu) - G'(iu-t) = \int_0^\infty e^{iux}(1-e^{-tx})C_\alpha x^{-\alpha}\,dx \]
    and so
    \begin{align*}
        \E{\widetilde \sigma_t} &= \frac{1}{2\pi p(0)}\int_\R \int_0^\infty e^{iux}(1-e^{-tx})C_\alpha x^{-\alpha}\exp(G(iu))\,dx\,du\\
            &= \int_0^\infty (1-e^{-tx})C_\alpha x^{-\alpha}\left(\frac{1}{2\pi p(0)}\int_\R e^{iux}e^{-\Psi_L(u)}\,du\right)\,dx\\
            &= \int_0^\infty (1-e^{-tx})C_\alpha x^{-\alpha}\frac{p(-x)}{p(0)}\,dx.
    \end{align*}
    Now a substitution of $y = tx$ yields
    \[ \E{\widetilde \sigma_t} = t^{\alpha-1}\int_0^\infty (1-e^{-y})C_\alpha y^{-\alpha}\frac{p(-y/t)}{p(0)}\,dy, \]
    and by dominated convergence the integral converges to $C_\alpha \int_0^\infty (1-e^{-y}) y^{-\alpha}\,dy$. Integration by parts reveals that this is equal to $C_\alpha \frac{\Gamma(2-\alpha)}{\alpha - 1} = \alpha$.
\end{proof}

\begin{corollary_restate}\label{cor: sigvar}
    $\Var{\widetilde \sigma_t} = O(t^{\alpha-1})$ as $t \to \infty$.
\end{corollary_restate}
\begin{proof}
    Let us write $H = H(\lambda, t, u) = G(iu - \lambda) - G(iu - t - \lambda) + G(iu - t)$. We observe that $H$ satisfies the PDE
        \[ \frac{\partial H}{\partial u} + i\frac{\partial H}{\partial \lambda} = iG'(iu - t),\]
    so in particular we have
        \[ \frac{\partial^2 H}{\partial \lambda^2} = i\frac{\partial}{\partial u}\left(\frac{\partial H}{\partial \lambda}\right).\]
    Using this, we can calculate
    \begin{align*}
        2\pi p(0) \E{\widetilde \sigma_t^2 e^{-\lambda \widetilde \sigma_t}} &= \int_\R \frac{\partial ^2}{\partial \lambda^2}\left(\exp(H(\lambda, t, u))\right)\,du\\
        &= \int_\R \left(\frac{\partial^2 H}{\partial \lambda^2} + \left(\frac{\partial H}{\partial \lambda}\right)^2\right)\exp(H)\,du\\
        &= -i\int_\R \frac{\partial H}{\partial \lambda}\frac{\partial}{\partial u}\left(\exp(H)\right)\,du + \int_\R \left(\frac{\partial H}{\partial \lambda}\right)^2\exp(H)\,du\\
        &= \int_\R \frac{\partial H}{\partial \lambda}\left(-i\frac{\partial H}{\partial u} + \frac{\partial H}{\partial \lambda}\right)\exp(H)\,du\\
        &= \int_\R \frac{\partial H}{\partial \lambda}G'(iu - t)\exp(H)\,du,
    \end{align*}
    using integration by parts to obtain the first term in the third line. Expanding and evaluating at $\lambda = 0$, we deduce that
        \[ \E{\widetilde \sigma_t^2} = \frac{1}{2\pi p(0)}\int_\R G'(iu-t)(G'(iu-t)-G'(iu))\exp(G(iu))\,du.\]
    Noting that $G(iu) = -\Psi_L(u)$, we further compute that
    \begin{align*}
        \E{\widetilde \sigma_t^2} &= \frac{1}{2\pi p(0)}\int_\R G'(iu-t)(G'(iu-t)-G'(iu))e^{-\Psi_L(u)}\,du\\
        &= \frac{C_\alpha^2}{2\pi p(0)}\int_\R e^{-\Psi_L(u)}\int_0^\infty (e^{(iu-t)x}-1)x^{-\alpha}\,dx\int_0^\infty e^{iuy}(e^{-ty}-1)y^{-\alpha}\,dy\,du\\
        &= \frac{C_\alpha^2}{2\pi p(0)}\int_0^\infty \int_0^\infty x^{-\alpha}y^{-\alpha}(e^{-ty}-1)\left(\int_\R e^{-\Psi_L(u)}\left(e^{iu(x+y)}e^{-tx}-e^{iuy}\right)\,du\right)\,dx\,dy\\
        &= \frac{C_\alpha^2}{p(0)}\int_0^\infty \int_0^\infty x^{-\alpha}y^{-\alpha}(e^{-ty}-1)\left(e^{-tx}p(-x-y)-p(-y)\right)\,dx\,dy.
    \end{align*}
    Now recall from \autoref{cor: sigmean} that $\E{\widetilde \sigma_t} = \frac{C_\alpha}{p(0)}\int_0^\infty (1-e^{-tx})x^{-\alpha}p(-x)\,dx$. Squaring this and writing it as a double integral, we obtain the following formula for the variance:
    \begin{align*}
        \Var{\widetilde \sigma_t} = \frac{C_\alpha^2}{p(0)^2}\int_0^\infty \int_0^\infty x^{-\alpha}y^{-\alpha}(1-e^{-ty})\Bigg(&p(-y)p(0) - e^{-tx}p(-x-y)p(0)\\
        & \quad - (1-e^{-tx})p(-x)p(-y)\Bigg)\,dx\,dy.
    \end{align*}

    The term in large brackets can be decomposed into $J_1 + J_2 + J_3$, where
    \begin{align*}
        J_1 &= p(-y)(p(0) - p(-x)),\\
        J_2 &= e^{-tx}p(0)(p(-y)-p(-x-y)),\\
        J_3 &= -e^{-tx}p(-y)(p(0)-p(-x)).
    \end{align*}
    But all $J_i$ can be bounded by $x||p||_\infty||p'||_\infty \wedge 2||p||_\infty^2$. As $p$ and $p'$ are bounded (a result that can be seen by the Fourier inversion formula) we can thus find constants $c_1, c_2 > 0$ such that $|J_i| \leq (c_1x) \wedge c_2$ for $i = 1, 2, 3$. Then
    \begin{align*}
        \Var{\widetilde \sigma_t} &\leq \frac{3C_\alpha^2}{p(0)^2}\int_0^\infty \int_0^\infty x^{-\alpha}y^{-\alpha}(1-e^{-ty})((c_1x) \wedge c_2)\,dx\,dy\\
        &= \frac{3C_\alpha^2}{p(0)^2}\int_0^\infty x^{-\alpha}((c_1x) \wedge c_2)\,dx \times \int_0^\infty y^{-\alpha}(1-e^{-ty})\,dy.
    \end{align*}
    The $x$-integral is finite and independent of $t$, while the $y$-integral is of order $t^{\alpha-1}$ (as seen in \autoref{cor: sigmean}). Thus we have a bound of order $t^{\alpha-1}$. 
\end{proof}

The process $\widetilde{\sigma}$ increases only by jumps. For $t \ge 0$, write $\Delta \widetilde{\sigma}_t = \widetilde{\sigma}_t - \widetilde{\sigma}_{t-}$. The \emph{quadratic variation process} $(Q_t)_{t \ge 0}$ of $\widetilde{\sigma}$ is defined to be
\[
Q_t = \sum_{0 \le s \le t} (\Delta \widetilde{\sigma}_t)^2,
\]
where the sum is over the (countably many) jumps of $\widetilde{\sigma}$ occurring before time $t$. As $Q_t$ is increasing in $t$, we let $Q_{\infty} = \lim_{t \to \infty} Q_t \in \R^+ \cup \{\infty\}.$

\begin{proposition_restate} \label{prop: quadraticvarfinite}
We have
\[
\E{Q_{\infty}} \leq C_\alpha \int_0^{\infty} x^{1-\alpha} \frac{p(-x)}{p(0)} dx < \infty.
\]
\end{proposition_restate}

We expect that in fact equality holds for the above, but we only prove an upper bound as it will suffice for our purposes. We will make use of the following lemma:

\begin{lemma_restate} \label{lem: quadvarlemma}
    For any $t, x > 0$ we have
        \[ \E{\frac{p(-\widetilde \sigma_t - x)}{p(-\widetilde \sigma_t)}} = \exp(-tx)\frac{p(-x)}{p(0)}. 
        \]
\end{lemma_restate}
\begin{proof}
    Using (\ref{eqn:martingalekey}) we can calculate
    \begin{align*}
        p(-x) &= \E{\exp\left(\int_0^t \left(x + \sigma_s\right)\,ds\right)p(-\sigma_t - x)}\\
        &= \exp(tx)\E{\exp\left(\int_0^t \sigma_s\,ds\right)p(-\sigma_t - x)}\\
        &= \exp(tx)\E{\exp\left(\int_0^t \sigma_s\,ds\right) \frac{p(-\sigma_t)}{p(0)} \frac{p(0)p(-\sigma_t - x)}{p(-\sigma_t)}}\\
        &= \exp(tx)\,p(0)\E{\frac{p(-\widetilde \sigma_t - x)}{p(-\widetilde \sigma_t)}}. \qedhere
    \end{align*} 
\end{proof}

\begin{proof}[Proof of \autoref{prop: quadraticvarfinite}]
For each $h \in (0, 1)$, let $I_h : [0, \infty) \to \R$ be the unique continuous function such that $I_h(\widetilde \sigma_t) = \frac{1}{h}\E{(\widetilde \sigma_{t+h} - \widetilde \sigma_t)^2 \mid \widetilde \sigma_t}$ for all $t \geq 0$. (Note that this exists by \autoref{prop:Markovprop}.) Also set
    \[ I(y) = C_\alpha\int_0^\infty x^{2-\alpha}\frac{p(-y-x)}{p(-y)}\,dx \quad \text{ for all } y \geq 0.\]
We begin by showing that
\begin{equation} \label{eqn:generatorconv}
\lim_{h \downarrow 0} I_h(\widetilde \sigma_t) = C_\alpha \int_0^{\infty} x^{2-\alpha} \frac{p(-\widetilde{\sigma}_t - x)}{p(-\widetilde{\sigma}_t)} dx = I(\widetilde \sigma_t) \quad \text{a.s.}
\end{equation}
To this end, note that by \autoref{prop:Markovprop}, we have
\begin{equation} \label{eqn:needsbounding}
\E{(\widetilde{\sigma}_{t+h} - \widetilde{\sigma_t})^2 \mid \widetilde{\sigma}_t} = \E{\sigma_h^2 \exp\left(h \widetilde{\sigma}_t + \int_0^h \sigma_s ds\right) \frac{p(-\widetilde{\sigma}_t - \sigma_h)}{p(-\widetilde{\sigma}_t)} \given{\Bigg} \widetilde{\sigma}_t},
\end{equation}
where $\sigma$ is an $(\alpha-1)$-stable subordinator, started from 0, and independent of $\widetilde{\sigma}_t$. 

For $r \ge 0$, define the functions $f_r(x,y) = x^2 \exp(rx) \frac{p(-y-x)}{p(-y)}$. Then
\[
\lim_{r \downarrow 0} f_r(x,y) = f_0(x,y).
\]
Moreover, all of these functions are bounded in both $x \ge 0$ and $y \ge 0$, uniformly for $0 \le r \le 1$.

Fix $\epsilon > 0$. Now note that the right-hand side of (\ref{eqn:needsbounding}) is bounded below by
\[
\exp(h \widetilde{\sigma}_t)\E{f_0(\sigma_h,\widetilde{\sigma}_t) \mid \widetilde{\sigma}_t}
\]
and above by
\[
\exp(h \widetilde{\sigma}_t) \E{f_{\epsilon}(\sigma_h,\widetilde{\sigma}_t) \mid \widetilde{\sigma}_t},
\]
whenever $h \le \epsilon$.

For $r \ge 0$, the (random) function $x \mapsto f_r(x,\widetilde{\sigma}_t)$ is twice continuously differentiable and such that $f_r(\cdot, \widetilde \sigma_t)$ along with its first and second partial derivatives in $x$ are vanishing as $x \to \infty$.  By Theorem 31.5 of \cite{satobook}, $f_r(\cdot,\widetilde{\sigma}_t)$ is in the domain of the generator of $\sigma$, which is the operator $\mathcal{L}$ given by
\[
\mathcal{L}f(y) = C_\alpha \int_0^{\infty} \left(f(y+x) - f(x)\right) x^{-\alpha} dx.
\]
Hence,
\[
\lim_{h \downarrow 0} \frac{1}{h} \E{f_r(\sigma_h,\widetilde{\sigma}_t) \mid \widetilde{\sigma}_t} = C_\alpha \int_0^{\infty} x^{2-\alpha} \exp(rx) \frac{p(-\widetilde{\sigma}_t - x)}{p(-\widetilde{\sigma}_t)} dx.
\]
It then follows that
\[
\liminf_{h \downarrow 0} \frac{1}{h} \E{(\widetilde{\sigma}_{t+h} - \widetilde{\sigma_t})^2 \mid \widetilde{\sigma}_t} \ge C_\alpha \int_0^{\infty} x^{2-\alpha} \frac{p(-\widetilde{\sigma}_t - x)}{p(-\widetilde{\sigma}_t)} dx
\]
and
\begin{align*}
\limsup_{h \downarrow 0} \frac{1}{h} \E{(\widetilde{\sigma}_{t+h} - \widetilde{\sigma_t})^2 | \widetilde{\sigma}_t} & \le C_\alpha \inf_{\epsilon > 0} \int_0^{\infty} x^{2-\alpha} \exp(\epsilon x) \frac{p(-\widetilde{\sigma}_t - x)}{p(-\widetilde{\sigma}_t)} dx \\
& = C_\alpha \int_0^{\infty} x^{2-\alpha} \frac{p(-\widetilde{\sigma}_t - x)}{p(-\widetilde{\sigma}_t)} dx,
\end{align*}
by bounded convergence. The claimed result (\ref{eqn:generatorconv}) follows.

Now note that, for each $h > 0$,
\[
Q_t = \sum_{0 \le s \le t}(\Delta \widetilde{\sigma}_s)^2 \le \int_0^t \frac{1}{h}(\widetilde{\sigma}_{s+h}-\widetilde{\sigma}_s)^2 ds
\]
and so by Fubini's theorem and the tower law,
\[
\E{Q_t} \le \int_0^t \E{\frac{1}{h} \E{(\widetilde{\sigma}_{s+h}-\widetilde{\sigma}_s)^2 \mid \widetilde{\sigma}_s}}\,ds = \int_0^t \E{ I_h(\widetilde \sigma_s)}\,ds.
\]
We would like to apply dominated convergence to the double integral $\int_0^t \E{\cdot}\,ds$; we will justify this step in a moment. For now, assuming that this application is valid, we obtain
    \[ \E{Q_t} \leq \lim_{h \downarrow 0} \int_0^t \E{I_h(\widetilde \sigma_s)}\,ds = \int_0^t \E{I(\widetilde \sigma_s)}\,ds. \]
Thus by monotone convergence we get $\E{Q_\infty} \leq \int_0^\infty \E{I(\widetilde \sigma_s)}\,ds$, and so it suffices to show that this integral is finite. This is now a straightforward calculation using \autoref{lem: quadvarlemma}:
\begin{align*}
    \int_0^\infty \E{I(\widetilde \sigma_s)}\,ds &= \int_0^\infty \int_0^\infty x^{2-\alpha}\E{\frac{p(-\widetilde \sigma_s - x)}{p(-\widetilde \sigma_s)}}\,ds\,dx\\
    &= \int_0^\infty \int_0^\infty x^{2-\alpha}\exp(-sx)\frac{p(-x)}{p(0)}\,ds\,dx\\
    &= \int_0^\infty x^{1-\alpha}\frac{p(-x)}{p(0)}\,dx < \infty.
\end{align*}

We now return to justifying the application of dominated convergence. We show that there exists some constant $M > 0$ such that, for all $s \in \R$ and all $h \in (0, 1)$, we have $I_h(\widetilde \sigma_s) \leq M\exp(\widetilde \sigma_s)$. Note that $\int_0^t \E{\exp(\widetilde \sigma_s)}\,ds < \infty$: in fact we can calculate it explicitly by analogous methods to those used in \autoref{prop: sigmalaplace}.

Using the asymptotic estimate from (\ref{eqn:satoasymptotic}) and the fact that $p$ is strictly positive we see that there exists a constant $R > 0$ such that 
\[
p(-x) \leq R(x^\beta \vee 1)\exp(-C_2x^{\frac{\alpha}{\alpha-1}}), \quad p(-x) \geq \frac{1}{R}(x^\beta \vee 1)\exp(-C_2x^{\frac{\alpha}{\alpha-1}}) \quad \text{for all $x \geq 0$},
\]
where $\beta = (2-\alpha)/(2\alpha-2)$.

Hence,
\begin{equation} \label{eq: densityratio}
 \frac{p(-x-y)}{p(-y)} \leq R^2 \quad \text{for all $x, y \geq 0$}. 
\end{equation}

For any $s \geq 0$ and any $h \in (0, 1)$, we may use \autoref{prop:Markovprop} to obtain
\begin{align*}
    I_h(\widetilde \sigma_s) &= \frac{1}{h}\E{(\widetilde \sigma_{s+h} - \widetilde \sigma_s)^2 \mid \widetilde \sigma_s}\\
    &= \frac{1}{h}\E{(\sigma'_h)^2\exp\left(\int_0^h \left(\widetilde \sigma_s + \sigma'_u\right)\,du\right)\frac{p(-\widetilde \sigma_s - \sigma'_h)}{p(-\widetilde \sigma_s)} \given{\Bigg} \widetilde \sigma_s}\\
    &\leq \frac{1}{h}\exp(h\widetilde \sigma_s)\E{(\sigma'_h)^2\exp\left(h\sigma'_h\right)\frac{p(-\widetilde \sigma_s - \sigma'_h)}{p(-\widetilde \sigma_s)} \given{\bigg} \widetilde \sigma_s}\\
    &\leq \exp(\widetilde \sigma_s) \times \frac{1}{h}\E{(\sigma'_h)^2\exp(\sigma'_h)\frac{p(-\widetilde \sigma_s - \sigma'_h)}{p(-\widetilde \sigma_s)} \given{\bigg} \widetilde \sigma_s}.
\end{align*}
It now suffices to find a constant $M$ such that
    \[ \frac{1}{h}\E{(\sigma'_h)^2\exp(\sigma'_h)\frac{p(-x - \sigma'_h)}{p(-x)}} \leq M\]
for all $x \geq 0$ and all $h \in (0, 1)$. We split the left-hand side into the sum
    \[ \frac{1}{h}\E{(\sigma'_h)^2\exp(\sigma'_h)\frac{p(-x - \sigma'_h)}{p(-x)}\IndEvent{\sigma'_h \leq R}} + \frac{1}{h}\E{(\sigma'_h)^2\exp(\sigma'_h)\frac{p(-x - \sigma'_h)}{p(-x)}\IndEvent{\sigma'_h > R}} \]
and find constants $M_1, M_2$ bounding each term.

Using (\ref{eq: densityratio}), we can crudely bound the first term by
\begin{align*}
    \frac{1}{h}\E{(\sigma'_h)^2\exp(\sigma'_h)\frac{p(-x - \sigma'_h)}{p(-x)}\IndEvent{\sigma'_h \leq R}} &\leq R^2 \times \frac{1}{h}\E{(\sigma'_h)^2\exp(\sigma'_h)\IndEvent{\sigma'_h \leq R}}\\
    &\leq R^2 e^R \times \frac{1}{h}\E{(\sigma'_h \wedge R)^2}.
\end{align*}
This bound does not depend on $x$, is finite for each fixed $h \in (0, 1)$ and is continuous in $h$ on this interval. Moreover, again using the generator, as $h \downarrow 0$ the bound converges to $R^2 e^R\int_0^\infty (y \wedge R)^2 C_\alpha y^{-\alpha}\,dy$, which is also finite. Thus there is some choice of $M_1$ such that our bound is at most $M_1$ for all $h \in (0, 1)$.

For the second term, we can use the asymptotic estimates for $p$ again, this time to obtain the existence of a constant $\widetilde R$ such that
    \[ y^2\exp(y)\frac{p(-x-y)}{p(-x)} \leq \widetilde R \quad \text{ for all } x, y \geq 0.\]
Using this, our second term satisfies the bounds
\begin{align*}
    \frac{1}{h}\E{(\sigma'_h)^2\exp(\sigma'_h)\frac{p(-x - \sigma'_h)}{p(-x)}\IndEvent{\sigma'_h > R}} &\leq \widetilde R \times \frac{1}{h}\Prob{\sigma'_h \geq R}\\
    &= \widetilde R \times \frac{1}{h}\Prob{\sigma'_1 \geq Rh^{-\frac{1}{\alpha-1}}}.
\end{align*}
This bound also depends only on $h$, is finite for each fixed $h$ and is continuous for $h \in (0, 1)$. As $h \to 0$, it converges to a constant (one can calculate this constant by integrating the density in (14.37) of \cite{satobook}), and so there exists some $M_2$ such that the bound is at most $M_2$ for all $h \in (0, 1)$.
\end{proof}

\subsection{Compactness of the $\R$-tree}
With $\widetilde \sigma$ constructed we may now consider a random $\mathbb R$-tree $\mathcal T$ generated by a line-breaking construction with intensity process $\widetilde \sigma$. We already know that $\mathcal T$ is a metric space; we aim to show that it is also compact. Recall that formally, $\mathcal T = (\mathbb R^+, d^\circ)$: it will be useful to introduce the notation $\mathcal T_x = ([0, x], d^\circ)$ for the subtree at time $x$.

Note that by \autoref{cor: sigmean} we have $\int_R^\infty \frac{1}{t\E{\widetilde \sigma_t}}dt < \infty$ for all large enough $R$. This condition matches that in the compactness conjecture from \cite{abrcompactness}, and our proof will follow a similar structure to the proof given there, although we will also require the variance estimate and a few more auxiliary lemmas.

The first lemma is a variation of Lemma 4.7 in \cite{abrcompactness}.
\begin{lemma_restate}\label{lem: abr4.7extension}
    Let $x \leq z \in \R^+$. Then conditionally given $\widetilde \sigma$, the law of $d(z, \mathcal T_x)$ is stochastically dominated by an exponential random variable of rate $\widetilde \sigma_x^2 / \widetilde \sigma_{z-}$. 
\end{lemma_restate}

\begin{proof}
    Recall the notation for cut points $Y_i$ and attachment points $Z_i$. We produce a sequence of points $(z_i, y_i)_{i \geq 0} \in [0, \infty)^2$ as follows. First, $z_0 = z$. Having constructed $z_i$, we define an index $k_i = \max\{k : Y_k < z_i\}$, then set $y_i = Y_{k_i}$ and $z_{i+1} = Z_{k_i}$. That is, $y_i$ is the start of the branch containing $z_i$, and $z_{i+1}$ is the attachment point for this branch (i.e.\ $y_i^+ = z_{i+1}$). The path $z_0 \to z_1 \to \dots$ gives the geodesic from $z$ back to the root, and terminating upon hitting $\mathcal T_x$ gives the geodesic from $z$ to this subtree.

    In particular, defining $T$ to be the least integer such that $z_{T+1} \leq x$, the distance $d(z, \mathcal T_x)$ is given by
        \[ d(z, \mathcal T_x) = \sum_{i=0}^T (z_i - \max\{y_i, x\}) \]
    (note that the maximum is always $y_i$ unless $i = T$, and when $i = T$ it equals $x$ only if the geodesic from $z$ to $\mathcal T_x$ ends at $x$). Our goal now is to show that the terms in the sum can be stochastically dominated by i.i.d.\ exponentials with parameter $\widetilde \sigma_x$, and that $T$ can be dominated by a geometric variable with parameter $\widetilde \sigma_x / \widetilde \sigma_{z-}$ which is independent of the dominating exponentials.

    First of all, note that $(z_i, y_i)$ forms a Markov chain: indeed given values for $(z_i, y_i)$ (and $\widetilde \sigma$), the law of $(z_{i+1}, y_{i+1})$ can be described as follows. First $z_{i+1}$ is sampled according to the normalisation of $d\widetilde \sigma|_{[0, y_i)}$. Then $y_{i+1}$ is the rightmost atom to the left of $z_{i+1}$ in the Poisson process of cuts. Running this procedure as $i$ increases only reveals atoms of the Poisson process from right to left, and so $y_{i+1}$ is conditionally independent of $(z_0, y_0), \dots, (z_{i-1}, y_{i-1})$ given $(z_i, y_i)$.

    We prove the stochastic domination via a coupling: let $S \subset (-\infty, z]$ be (the atoms of) a Poisson process of constant intensity $\widetilde \sigma_x$, coupled with the Poisson process of cuts such that $S \cap (-\infty, x]$ is independent of the cuts, and $S \cap (x, z]$ is a subset of $\{Y_i : i \geq 1\} \cap (x, z]$. For $i \geq 0$, let $w_i$ be the largest element of $S$ smaller than $z_i$. Let $E_i$ be a sequence of i.i.d.\ exponentials of rate $\widetilde \sigma_x$, independent of everything else, and set
        \[ e_i = \begin{cases}z_i - w_i, & i \leq T\\E_i, & i > T.\end{cases}\]
    Observe that $e_i$ are all i.i.d.\ exponentials of rate $\widetilde \sigma_x$, are independent of $T$, and almost surely we have $z_i - \max\{y_i, x\} \leq e_i$ for all $i \leq T$. Thus we have
        \[ d(z, \mathcal T_x) \leq \sum_{i = 0}^T e_i, \quad \text{ a.s.}\]
    Hence it remains only to show $T$ is stochastically dominated by a (zero-indexed) geometric random variable with parameter $p = \frac{\widetilde \sigma_x}{\widetilde \sigma_{z^-}}$. For each $n \geq 0$, let $\mathcal F_n = \sigma((z_i, y_i) : i \leq n, \widetilde \sigma)$ and observe that $T+1$ is an $(\mathcal F_n)$-stopping time. Moreover, on the event $\{ T+1 > n\} \in \mathcal F_n$, we have
        \[ \Prob{T+1 \leq n+1 \given\big \mathcal F_n} = \Prob{z_{n+1} \leq x \given\big \mathcal F_n} = \frac{\widetilde \sigma_x}{\widetilde \sigma_{y_n-}} \geq \frac{\widetilde \sigma_x}{\widetilde \sigma_{z-}}.\]
    On the complement of this event, the conditional probability is 1, and so we have
        \[ \Prob{T+1 \leq n+1 \given\big \mathcal F_n} \geq \frac{\widetilde \sigma_x}{\widetilde \sigma_{z-}} \IndEvent{T + 1 > n} + \IndEvent{T + 1 \leq n}.\]
    More conveniently we can write
        \[ \Prob{T + 1 > n + 1 \given\big \mathcal F_n} \leq \left(1 - \frac{\widetilde \sigma_x}{\widetilde \sigma_{z-}}\right)\IndEvent{T+1 > n}, \]
    and then the tower law yields
         \[ \Prob{T \ge n \given\big \widetilde \sigma} \leq \left(1 -\frac{\widetilde \sigma_x}{\widetilde \sigma_{z-}}\right)^n\Prob{T \geq 0 \given\big \widetilde \sigma} = \left(1 -\frac{\widetilde \sigma_x}{\widetilde \sigma_{z-}}\right)^n, \quad n \ge 0, \]
    completing the proof.
\end{proof}

The strategy to prove compactness is as follows: we take a sequence $\chi_n \to \infty$ of times and consider the trees $\mathcal T_{\chi_n}$. These are all compact subsets of the final tree $\mathcal T$, and we show that they form a Cauchy sequence with respect to the Hausdorff distance $d_{\mathrm{H}}$ in $\mathcal T$. The Hausdorff distance yields a complete metric on compact subsets, and thus $\mathcal T_{\chi_n} \to \mathcal T_\infty$ for some compact $T_\infty \subset \mathcal T$. We then argue that $\mathcal T_\infty = \mathcal T$, and thus $\mathcal T$ is itself a compact space.

To this end, define $\chi_n = \inf\{t \geq 0: \E{\widetilde \sigma_t} > 2^n\}$. The asymptotics for the mean imply that $\chi_n \sim c \, 2^{n/(\alpha-1)}$ for some $c > 0$. To prove $\mathcal T_{\chi_n}$ is a.s.\ Cauchy we show the following:

\begin{lemma_restate}\label{lem: cpctcauchyestimate}
    Almost surely there exists a (random) $n_0 \geq 1$ such that for all $n \geq n_0$ we have
        \[ d_{\mathrm{H}}(\mathcal T_{\chi_{n-1}}, \mathcal T_{\chi_n}) \leq \frac{18 \log(\chi_n)}{2^n} + \chi_n^{-1}. \]
\end{lemma_restate}
\begin{proof}
    A straightforward second-moment argument and Borel--Cantelli reveal that $\widetilde \sigma_{\chi_n} \sim 2^n$ a.s. Thus we can find a (random) $n_1 \geq 1$ such that
        \[ \widetilde \sigma_{\chi_n} \in \left[0.9 \times 2^n, 1.1 \times 2^n\right] \quad \text{ for all } n \geq n_1 - 1.\]
    In what follows we work with $n \geq n_1$ only. For any $t \in [0, \chi_n]$, let $E_n(t)$ be the event that $d(t, \mathcal T_{\chi_{n-1}}) \geq \frac{18\log(\chi_n)}{2^n}$. By Fubini's theorem and \autoref{lem: abr4.7extension}, we have
    \begin{align*}
        \E{\int_0^{\chi_n}\! \IndEvent{E_n(t)} dt \,\,\Bigg|\, \widetilde \sigma} &= \int_0^{\chi_n}\Prob{E_n(t) \,\big|\, \widetilde \sigma}\,dt\\
        &\leq \int_0^{\chi_n}\Prob{\Expo{\frac{\widetilde \sigma_{\chi_{n-1}}^2}{\widetilde \sigma_{\chi_n-}}} \geq \frac{18\log(\chi_n)}{2^n} \,\Bigg|\, \widetilde \sigma}\,dt\\
        &= \chi_n \exp\left(-\frac{18\log(\chi_n)\widetilde \sigma_{\chi_{n-1}}^2}{2^n \widetilde \sigma_{\chi_n-}}\right).
    \end{align*}
    Using $\widetilde \sigma_{\chi_{n-1}} \geq 0.9 \times 2^{n-1}$ and $\widetilde \sigma_{\chi_n} \leq 1.1 \times 2^n$ and some crude numerical estimates, we see the above quantity is at most $\chi_n^{-2}$. Then by Markov's inequality, we have
        \[ \Prob{\int_0^{\chi_n}\IndEvent{E_n(t)}\,dt > \chi_n^{-1} \,\Bigg|\, \widetilde \sigma} \leq \chi_n^{-1}\]
    for all $n \geq n_1$. Note that $\sum_{n \geq n_1}\chi_n^{-1} < \infty$ a.s.\ and so, by Borel--Cantelli, we almost surely have
        \[ \int_0^{\chi_n} \IndEvent{E_n(t)}dt \leq \chi_n^{-1} \quad \text{ for all } n \geq n_2, \]
    where $n_2$ is some (random) finite bound. We claim that the estimate on the Hausdorff distance holds for all such $n$. Indeed, if it does not hold, then there is some point $x \in \mathcal T_{\chi_n}$ with $d(x, \mathcal T_{\chi_{n-1}}) > \frac{18\log(\chi_n)}{2^n} + \chi_n^{-1}$. But then there is a section of the geodesic path from $x$ to $\mathcal T_{\chi_{n-1}}$ of length $d(x, \mathcal T_{\chi_{n-1}}) - \frac{18\log(\chi_n)}{2^n} > \chi_n^{-1}$ on which $\IndEvent{E_n(t)} = 1$, contradicting the integral bound.
\end{proof}

We can now put everything together:
\begin{theorem}\label{thm: alphaiscompact}
    The $\R$-tree $(\Tcal,d)$ built from the line-breaking construction with intensity $\widetilde \sigma$ is almost surely compact.
\end{theorem}
\begin{proof}
    By \autoref{lem: cpctcauchyestimate} and the observation that $\sum_{n=1}^\infty\left(\frac{18\log(\chi_n)}{2^n} + \chi_n^{-1}\right) < \infty$, we see that $(\mathcal T_{\chi_n})$ is almost surely Cauchy with respect to the Hausdorff distance. By completeness, we deduce that a.s.\ there exists a compact $\mathcal T_\infty \subset \mathcal T$ such that $\mathcal T_{\chi_n} \to \mathcal T_\infty$ a.s.\ in the Hausdorff metric.

    We show that (on the event that $\mathcal T_\infty$ exists) $\mathcal T_\infty = \mathcal T$. Suppose not; then there exists some point $x \in \mathcal T \setminus \mathcal T_\infty$. As $\mathcal T_\infty$ is compact, it is closed in $\mathcal T$ and so $x$ has positive distance from $\mathcal T_\infty$. Let $r = d(x, \mathcal T_\infty) > 0$. It follows that $d(x, \mathcal T_{\chi_n}) \geq r$ for all $n$. However, observe that
        \[ \bigcup_{n \geq 1} \mathcal T_{\chi_n} = \bigcup_{t \geq 0} \mathcal T_t \]
    and, by construction of $\mathcal T$, the latter is a dense subset of $\mathcal T$. Thus we may find a sequence $x_k \in \bigcup_{n \geq 1} \mathcal T_{\chi_n}$ such that $d(x, x_k) \to 0$. But for each $k$ there is some $n = n(k)$ such that $x_k \in \mathcal T_{\chi_n}$, and then
        \[ d(x, x_k) \geq d(x, \mathcal T_{\chi_{n(k)}}) \geq r \]
    which is a contradiction.
\end{proof}

\subsection{The mass measure}

We now seek to equip this compact metric space with a probability measure. For each $n$, recall that $\mu_n = \frac 1n \sum_{i = 1}^n \delta_{Y_i}$ is the uniform measure on the leaves of $\mathcal T(n) = \mathcal T_{Y_n}$. We will construct a weak limit of $(\mu_n)$ on $\mathcal T$.

For each $n$, let $\pi_n : \mathcal T \to \mathcal T(n)$ be the projection map onto $\mathcal T(n)$. We are interested in subsets of $\mathcal T$ which are pre-images under these projection maps.

\begin{lemma_restate}\label{lem: meascvgprojections}
    Let $n \geq 1$ and let $\widetilde S \subset \mathcal T(n)$ be measurable. Set $S = \pi_n^{-1}(\widetilde S)$. Then the sequence $(\mu_i(S))_{i \geq n}$ converges almost surely to a limit in $[0, 1]$.
\end{lemma_restate}

To prove this lemma, we require the following result about time-dependent P\'olya urns (part of which is Theorem 1 of Pemantle~\cite{Pemantle}). We defer the straightforward proof to the Appendix.

\begin{restatable}{proposition_restate}{polyaurncvg}\label{prop: polyaurncvg}
    Let $A_0, M_0, (m_n)_{n \geq 1}$ be (possibly random) positive real numbers with $A_0 < M_0$. Write $M_n = M_0 + \sum_{i=1}^n m_i$ for brevity. Conditionally given all of these, let $(A_i)_{i \geq 1}$ be a Markov process such that
\begin{align*}
    \Prob{A_i = A_{i-1} + m_i \given\Big (A_j)_{0 \leq j < i}, (M_j)_{j \geq 0}} &= \frac{A_{i-1}}{M_{i-1}}, \quad \text{ and}\\
    \Prob{A_i = A_{i-1} \given\Big (A_j)_{0 \leq j < i}, (M_j)_{j \geq 0}} &= \frac{M_{i-1} - A_{i-1}}{M_{i-1}}.
\end{align*}
    Then we can find a random variable $X_\infty$, supported on $[0, 1]$, such that the two sequences
        \[ \left(\frac{A_n}{M_n}\right)_{n \geq 1} \quad \text{ and } \quad \left(\frac{\#\{i \leq n : A_i \neq A_{i-1}\}}{n}\right)_{n \geq 1} \]
    converge almost surely to $X_\infty$.
\end{restatable}

\begin{proof}[Proof of \autoref{lem: meascvgprojections}]
    Take $n = n_0$ and pick $\widetilde S \subset \mathcal T(n_0).$ Let $S = \pi_{n_0}^{-1}(\widetilde S)$ be as in the statement. Let $\nu$ be the unique measure on $\mathbb R^+$ such that $\nu([0, x]) = \widetilde \sigma_x$ for all $x \geq 0$, and note that we can view this as a measure on $\mathcal T$ such that $\nu(\mathcal T(n)) = \widetilde \sigma_{Y_n} < \infty$ for each $n$. Now set
        \[ A_i = \nu(S \cap \mathcal T(i)) \quad \text{ and } \quad M_i = \nu(\mathcal T(i)) = \widetilde \sigma_{Y_{i}}\]
    for $i \geq n_0$, also defining $m_i = M_i - M_{i-1}$ for $i > n_0$. We note that when going from $A_{i-1}$ to $A_i$, the value either stays the same (if $Z_{i-1} \not \in S$) or increases by $m_i$ (if $Z_{i-1} \in S$), and conditionally given everything so far, the probability of the latter is $\frac{A_{i-1}}{M_{i-1}}$. Thus $(A_i)_{i \geq n_0}$ and $(M_i)_{i \geq n_0}$ form a time-dependent P\'olya urn as in \autoref{prop: polyaurncvg} (up to shifting the indices). It follows that the sequence
        \[ \left(\frac{\#\{i \leq n : A_i \neq A_{i-1}\}}{n}\right)_{n \geq 1}\]
    almost surely converges to a limit $X_\infty$ in $[0, 1]$. For each fixed $n > n_0$, the numerator is counting the number of indices $i$ with $n_0 \leq i \leq n$ such that $Z_{i-1} \in S$. For this range of indices we have $Z_{i-1} \in S$ iff $Y_i \in S$, and so we can estimate
        \[ \left|\frac{\#\{i \leq n : A_i \neq A_{i-1}\}}{n} - \mu_n(S)\right| \leq \frac{n_0}{n} \to 0.\]
    We deduce that $\mu_n(S) \to X_\infty$ almost surely.
\end{proof}

\begin{corollary_restate} \label{thm: measfunc}
    Almost surely the following holds: for any function $g : \mathcal T \to \mathbb R$ which is bounded and continuous, the sequence $(\mu_n(g))_{n \geq 1}$ converges to a limit in $\mathbb R$, where $\mu_n(g) := \int_{\mathcal T}g(x)\,d\mu_n(x).$
\end{corollary_restate}
\begin{proof}
    For each $n$, the subspace $\mathcal T_{Y_n}$ is compact, hence totally bounded, so we may partition it into finitely many subsets each with diameter at most $1/n$. Write $I^n_1, I^n_2, \dots, I^n_N$ for this partition, where $N = N(n)$ is the number of pieces. Without loss of generality we may take these partitions to be nested, in the following sense: for any $n < k$, the induced partition of $\mathcal T(k)$ from that of $\mathcal T(n)$ with pieces $(\pi_n^{-1}(I^n_i) \cap \mathcal T(k))_{i \leq N(n)}$ is coarser than the partition $(I^k_i)_{i \leq N(k)}$. That is, each piece $I^k_i$ is a subset of some $\pi_n^{-1}(I^n_j) \cap \mathcal T(k)$. Now pick arbitrary elements $x^n_i \in I^n_i$ for all $n$ and $i$.

    Write $J^n_i = \pi_n^{-1}(I^n_i)$, and note that for each fixed $n$, the sets $(J^n_i)_{i \leq N(n)}$ partition $\mathcal T$. By \autoref{thm: alphaiscompact} and \autoref{lem: meascvgprojections}, the event
        \[ \mathcal G = \{ \mathcal T \text{ is compact} \} \cap \{ (\mu_k(J^n_i))_{k \geq 1} \text{ converges for all pairs } (n, i) \text{ with } i \leq N(n)\}\]
    has probability 1. We show that the conclusion of the theorem holds on this event. Since $\mathcal T$ is compact, the Hausdorff distance $d_{\mathrm{H}}(\mathcal T, \mathcal T(n))$ between $\mathcal T$ and $\mathcal T(n)$ is finite and tends to zero as $n \to \infty$. It follows that we can bound
        \[ \delta_n := \max_{1 \leq i \leq N(n)}\text{diam}(J^n_i) \leq \frac{1}{n} + 2d_{\mathrm{H}}(\mathcal T, \mathcal T(n)) \to 0.\]

    Write $\mu(J^n_i) = \lim_{k \to \infty} \mu_k(J^n_i)$. The first step is to show that, for any bounded continuous function $g$, the limit
        \[ \lim_{n \to \infty} \sum_{i=1}^{N(n)} \mu(J^n_i)g(x^n_i) =: \mu(g)\]
    exists. We will do this by showing the sequence is Cauchy. Fix any $\varepsilon > 0$. By compactness of $\mathcal T$, $g$ is uniformly continuous, so there exists $\delta > 0$ such that $d(u, v) \leq \delta$ implies $|g(u) - g(v)| \leq \varepsilon$. Now there exists $n_0 \in \mathbb N$ such that $\delta_n < \delta$ for all $n \geq n_0$. For any $m > n > n_0$, we partition $[N(m)]$ into $N(n)$ pieces given by
        \[ F_i = \{ j \in [N(m)]: I^m_j \subset \pi_n^{-1}(I^n_i) \cap \mathcal T(m) \}, \quad\quad 1 \leq j \leq N(n). \]
    Now we compute
    \begin{align*}
        \left|\sum_{i=1}^{N(n)} \mu(J^n_i)g(x^n_i) - \sum_{j=1}^{N(m)}\mu(J^m_j)g(x^m_j)\right| &\leq \sum_{i=1}^{N(n)}\left|\mu(J^n_i)g(x^n_i) - \sum_{j \in F_i}\mu(J^m_j)g(x^m_j)\right|\\
        &= \sum_{i=1}^{N(n)}\left|\sum_{j \in F_i}\mu(J^m_j)\left(g(x^n_i) - g(x^m_j)\right)\right|\\
        &\leq \sum_{i=1}^{N(n)}\sum_{j \in F_i} \mu(J^m_j)\left|g(x^n_i) - g(x^m_j)\right|\\
        &\leq \sum_{i=1}^{N(n)}\sum_{j \in F_i} \mu(J^m_j) \cdot \varepsilon\\
        &= \varepsilon,
    \end{align*}
    where in the penultimate line we have used that both $x^n_i$ and $x^m_j$ lie within $J^n_i$, which has diameter at most $\delta_n < \delta$.

    We now prove that for any bounded continuous $g$, we have $\mu_n(g) \to \mu(g)$, where $\mu(g)$ is the limiting value obtained above. For any $m > n$ we may estimate
    \begin{multline} \label{eq: mu_nmu}
        \left|\mu_m(g) - \mu(g)\right|  \\
        \quad \leq \left|\mu_m(g) - \sum_{i=1}^{N(n)}\mu_m(J^n_i)g(x^n_i)\right| + \sum_{i=1}^{N(n)}\left|\mu_m(J^n_i) - \mu(J^n_i)\right||g(x^n_i)|
        + \left|\sum_{i=1}^{N(n)}\mu(J^n_i)g(x^n_i) - \mu(g)\right|.
    \end{multline}
    By writing $\mu_m(g) = \sum_{i=1}^{N(n)}\mu_m(g\Ind_{J^n_i})$, we can crudely bound the first term by $\sup\{|g(x) - g(x^n_i)| : x \in J^n_i, i \leq N(n)\}$. Now given $\varepsilon > 0$, we proceed as follows. By uniform continuity of $g$ and the diameter bound on $J^n_i$, we can find $n_1$ such that, whenever $n \geq n_1$, we have $\sup\{|g(x) - g(x^n_i)| : x \in J^n_i, i \leq N(n)\} \leq \varepsilon / 3$. By the convergence to $\mu(g)$ shown above, we can also find $n_2$ such that, whenever $n \geq n_2$, the third term in (\ref{eq: mu_nmu}) is at most $\varepsilon / 3$. Now take $n = n_1 \vee n_2$. For this choice of $n$, the second term in (\ref{eq: mu_nmu}) tends to 0 as $m \to \infty$, so there is some $m_0$ such that the second term is at most $\varepsilon / 3$ whenever $m \geq m_0$. Thus, for all $m \geq m_0$ we have $|\mu_m(g) - \mu(g)| \leq \varepsilon$, which establishes the required convergence.
\end{proof}

\begin{theorem}\label{thm: weakcvglinebreakmeasure}
   The sequence of measures $(\mu_n)_{n \ge 1}$ possesses a weak limit $\mu$ which is a probability measure on $\mathcal{T}$.
\end{theorem}

\begin{proof}
    Note that the convergence in \autoref{thm: measfunc} occurs for all bounded continuous functions simultaneously --- a much stronger statement than only having almost sure convergence for each individual bounded continuous function. On this probability 1 event, we can define a map $\mu$ from bounded continuous functions on $\mathcal T$ to the reals given by
        \[ \mu(g) := \lim_{n \to \infty} \mu_n(g). \]
    This is a positive linear functional and $\mathcal T$ is both locally compact and Hausdorff, so by the Riesz--Markov representation theorem it is realised by a Borel measure on $\mathcal T$. In a slight abuse of notation we will also call this measure $\mu$. It is easy to see that this measure has mass 1, i.e.\ is a probability measure.
\end{proof}

Note that it now follows straightforwardly that
\begin{equation} \label{eq:asGHPconv}
(\mathcal T(k), d|_{\mathcal{T}(k)}, \mu_k) \to (\mathcal{T}, d, \mu)
\end{equation}
almost surely, as $k \to \infty$, for the Gromov--Hausdorff--Prokhorov distance.

\section{Line-breaking constructions for discrete trees}\label{sec: discrete}
We begin by introducing the Ulam--Harris notation for rooted ordered trees. Set
    \[ \mathcal U = \bigcup_{n = 0}^\infty \mathbb N^n, \]
with the convention that $\mathbb N^0 = \{ \varnothing \}$. We write elements of $\mathbb N^n$ as $u = u_1\dots u_n$, for any $n \geq 1$, and we call $n$ the length of $u$, for which we write $n = |u|$. This gives the structure of a rooted ordered tree: indeed, $\varnothing$ is the root and any $u = u_1\dots u_n \in \mathbb N^n$ is a child of $u_1\dots u_{n-1}$.

A \emph{(locally finite) rooted ordered tree} is a subset $T \subset \mathcal U$ such that
    \begin{enumerate}[(i)]
        \item $\varnothing \in T$;
        \item For any $u \in T$ (including $u = \varnothing$) there is an integer $c(u) \geq 0$ such that for all $i \in \mathbb N$,
            \[ ui \in T \iff i \leq c(u). \]
    \end{enumerate}

 Let $\mu$ be a probability measure on $\mathbb N_0 = \{0, 1, 2, \dots \}$. A \emph{Bienaym\'e tree} with offspring distribution $\mu$ is a random rooted ordered tree $T \subset \mathcal U$ such that all integers $c(u), u \in \mathcal U$ are i.i.d.\ with distribution $\mu$. (Note that in this definition we have specified $c(u)$ even when $u \not \in T$ --- this makes the definition easier as we do not need to condition on the event that $u$ is a vertex in the tree.) For $\xi$ a random variable taking values in $\mathbb N_0$, we will often speak of Bienaym\'e trees with offspring distribution $\xi$: this is shorthand for a Bienaym\'e tree with offspring distribution $(p_k)$, where $p_k = \Prob{\xi = k}$.

Let $\mathbb T_n$ denote the set of rooted labelled trees with $n$ vertices. The Bienaymé trees (conditioned on size) will be our primary model for discrete random trees. It will be convenient for our purposes to think of them as random elements of $\mathbb{T}_n$ rather than $\mathcal U$; we will explain how to do this in \autoref{sec: relation} below. 

\subsection{Reverse Pr\"ufer codes}
We now present a bijection between $\mathbb T_n$ and $[n]^{n-1}$, the space of sequences of length $n-1$ taking values in $\{1, 2, \dots, n\}$. This is a variant of the classical Pr\"ufer code, and will be useful for constructing random discrete trees. This variant has been studied in \cite{seoshin} and is closely related to the bijection presented in \cite{surveypaper}.

Given a rooted labelled tree $(t, \rho)$, we produce a codeword by `revealing' the tree one edge at a time. We first reveal the path from the root to the lowest labelled (non-root) vertex, then the path from this branch to the lowest labelled vertex not already included, and so on. Formally, set $C_0 = 1$ and $L_1 = \min\{ i \geq 1 : i \neq \rho\}$. Say the path from $\rho$ to $L_1$ consists of $C_1 \geq 2$ vertices, namely $\rho = w_1, w_2, \dots, w_{C_1-1}$ and $L_1$ appearing in this order. Now inductively for $k \geq 2$, we take
    \[ \mathcal U_{k-1} = \bigcup_{j=1}^{k-1}\{w_{C_{j-1}}, \dots, w_{C_j-1}, L_j\}, \quad L_k = \min( [n] \setminus \mathcal U_{k-1}) \]
and consider the path from $\mathcal U_{k-1}$ to $L_k$: suppose its vertices are $w_{C_{k-1}}, w_{C_{k-1}+1}, \dots, w_{C_{k}-1}$ and $L_k$ in this order (so $w_{C_{k-1}} \in \mathcal U_{k-1}$). We repeat this until $\mathcal U_k = [n]$, i.e.\ we have exhausted the entire tree. Note that $\#\mathcal U_k = C_k$ for each $k \ge 1$, and so if we stop after exploring $B$ branches, then $C_B = n$. We take as codeword the sequence $(w_i)_{1 \leq i \leq C_B - 1}$, which has length $n-1$.

We may immediately observe that the multiplicity of a label in the codeword is equal to the out-degree of the corresponding vertex. In particular, the leaves never appear in the codeword.

We can also reverse this construction --- given a codeword $(w_i)_{1 \leq i \leq n-1}$, we show how to produce a corresponding tree. We know $w_1$ is the root. We calculate the values of $L_j$ and $C_j$ first: we know $C_0 = 1$ and $L_1 = \min\{ i \geq 1 : i \neq w_1\}$. New branches start whenever the codeword repeats a label or a previously revealed leaf, and so for $k \geq 1$,
    \[ C_k = \min \left\{ i \geq C_{k-1} + 1: w_i \in \{w_1, \dots, w_{i-1}\} \cup \{L_1, \dots, L_k\}\right\}\]
and
    \[ L_{k+1} = \min \left\{ i \geq 1 : i \not \in \bigcup_{j=1}^k \{ w_{C_{j-1}}, \dots, w_{C_j -1}, L_j \}\right\}\]
Now given these values, we can form the branches $(w_{C_{k-1}}, \dots, w_{C_k-1}, L_k)$ and glue these together to build a tree.

Write $\psi_n : [n]^{n-1} \to \mathbb T_n$ for this mapping. Note that for any vector $\bm d \in \mathbb N_0^n$ with $\sum_{i=1}^n d_i = n-1$, we can consider the set of codewords
    \[ W_{n, \bm d} = \{ (w_1, \dots, w_{n-1}) \in [n]^{n-1} : \#\{j : w_j = i\} = d_i \text{ for all } i \in [n]\}\]
and note that $\psi_n|_{W_{n, \bm d}}$ bijects to labelled rooted trees with out-degree sequence $\bm d$ (i.e.\ such that the vertex with label $i$ has out-degree $d_i$).

We observe that revealing entries in the codeword one at a time gives a way to grow a tree. This is a deterministic construction, but gives a nice way to grow uniformly random trees with a given degree sequence (resp. uniformly random trees) using the fact that the entries of a uniform element of $W_{n, \bm d}$ (resp. of $[n]^{n-1}$) have a relatively simple law. Here is an explicit description of this growth procedure: throughout, we keep track of an `active vertex' from which we expand the tree. In the first step, we reveal the root (whose label is $w_1$) and mark it as the active vertex. At step $k \geq 2$, we reveal the $k$th entry $w_k$ in the codeword, and
\begin{enumerate}[(i)]
    \item if this label has not yet appeared in the tree, we add a new vertex labelled $w_k$, connect it to the active vertex, and make $w_k$ the new active vertex;
    \item if the label has previously appeared in the codeword, we add a new vertex, whose label is the lowest index not yet present in the tree, connect it to the active vertex, and then make $w_k$ the new active vertex;
    \item if the label has previously appeared in the tree but not in the codeword, we add a new vertex exactly as in case (ii), then make $w_k$ the new active vertex.
\end{enumerate}

In this sense, we have a discrete analogue of the line-breaking construction for CRTs: if we think of repeated labels as rare events, then most of the time we are growing a branch, occasionally interrupting the process to move to a new position and start growing a new branch.

\subsection{Relation to Bienaym\'e trees} \label{sec: relation}
We now consider randomising the degree sequence $\bm d$, and conditionally given $\bm d$ sampling a uniform element of $W_{n, \bm d}$; this will turn out to give a construction of (a labelled, unordered version of) a Bienaym\'e tree conditioned to have fixed size $n$.

Let $\xi$ be a random variable taking values in $\mathbb N_0$ and such that $\E{\xi} = 1$; then $\xi$ is the offspring distribution of a critical Bienaym\'e tree. For simplicity we shall assume that $\gcd\{k : \Prob{\xi = k} > 0\} = 1$. We consider sampling a random vector $\bm D \in \mathbb N_0^n$ with the distribution of $n$ i.i.d.\ copies of $\xi$ conditioned to have sum $n-1$. Now conditionally given $\bm D$, we sample a uniformly random $\bm w \in W_{n, \bm D}$, and use the bijection $\psi_n$ to map it to a rooted labelled tree $(T^n, \rho)$.

We prefer to work with ordered trees, so we use the following random procedure to convert our labelled tree into an ordered tree. First, we map the root $\rho$ to $\varnothing$. Now for any vertex $i$ of $T^n$ that is not a leaf, we identify its children and give them a uniformly random ordering, independently of everything constructed so far. Finally, we forget our original labelling. We claim that this gives a size-conditioned Bienaym\'e tree with offspring distribution $\xi$.

\begin{proposition_restate}\label{prop: pruferisbienayme}
    The ordered tree corresponding to $(T^n, \rho)$ has the law of a Bienaym\'e tree with offspring distribution $\xi$ conditioned to have $n$ vertices. 
\end{proposition_restate}
\begin{proof}
    Fix a rooted ordered tree $t$ of size $n$; we study its probability of arising from our construction. It will be helpful to consider the rooted labelled ordered tree arising in the process just before forgetting the labelling: there are $n!$ possibilities for such a tree that give rise to the final tree $t$, so we fix one such rooted labelled ordered tree $\widetilde t$ and study the probability that $\widetilde t$ arises instead.

    Consider traversing $\tilde t$ in lexicographical order: let $d_1, d_2, \dots, d_n$ be the out-degrees of the vertices in this order (so, for example, $d_1$ is the degree of the root, regardless of the labelling). Let $\sigma$ be the permutation of $[n]$ such that, for all $i$, the $i$th vertex in lexicographical order has label $\sigma(i)$ in $\tilde t$. We observe that our construction generates $\tilde t$ if and only if all three of the following conditions are satisfied:
    \begin{enumerate}[(1)]
        \item The degree sequence $\bm D$ is chosen so that $D_i = d_{\sigma^{-1}(i)}$ for all $i$;
        \item The rooted labelled tree generated by the codeword matches $\tilde t$ without its ordering;
        \item The randomly chosen ordering on the labelled tree matches the ordering of $\tilde t$.
    \end{enumerate}
    Condition (1) occurs with probability
        \[ \Prob{D_1 = d_{\sigma^{-1}(1)}, \dots, D_n = d_{\sigma^{-1}(n)}} = \frac{\prod_{i=1}^np_{d_{\sigma^{-1}(i)}}}{\Prob{\sum_{i=1}^n\xi_i = n-1}} = \frac{\prod_{i=1}^n p_{d_i}}{\Prob{\sum_{i=1}^n \xi_i = n-1}}.\]
    Given that condition (1) occurs, the codeword is chosen uniformly from a set of size $\binom{n-1}{d_1 \dots d_n}$, and exactly one of these codewords produces the correct labelled tree, so we meet condition (2) with probability $\binom{n-1}{d_1 \dots d_n}^{-1}$. Given that conditions (1) and (2) are met, the ordering matches iff, for each $i$, the permutation of children of the vertex with label $i$ matches the ordering in $\tilde t$. As the vertex with label $i$ has $d_{\sigma^{-1}(i)}$ children, the probability that all sets of children are correctly ordered is
        \[ \frac{1}{\prod_{i=1}^n d_{\sigma^{-1}(i)}!} = \frac{1}{\prod_{i=1}^n d_i!}.\]
    Putting everything together, the probability that we generate $\tilde t$ is
        \[ \frac{\prod_{i=1}^np_{d_i}}{\Prob{\sum_{i=1}^n \xi_i = n-1}} \times \binom{n-1}{d_1 \dots d_n}^{-1} \times \frac{1}{\prod_{i=1}^n d_i!} = \frac{\prod_{i=1}^n p_{d_i}}{(n-1)!\Prob{\sum_{i=1}^n \xi_i = n-1}},\]
    and this quantity is the same for all $\tilde t$ which produce $t$ upon forgetting the ordering. Thus the probability of generating $t$ is $n!$ times this, which is
        \[ \frac{\prod_{i=1}^n p_{d_i}}{\frac{1}{n}\Prob{\sum_{i=1}^n \xi_i = n-1}}.\]
    By the cycle lemma, the denominator is equal to the probability that a random walk with steps $\xi - 1$ reaches $-1$ for the first time at time $n$, which in turn is the probability that a Bienaymé tree with offspring distribution $\xi$ has size $n$. The numerator is the probability of $t$ arising from an unconditioned Bienaymé tree, completing the proof.
\end{proof}

Let us return to our picture of growing trees one vertex at a time: instead of revealing the entire degree sequence in the beginning, we can consider only revealing degrees of vertices the first time we encounter them. To this end, it is useful to reorder $\bm D$ based on first appearances of the corresponding labels in the codeword (with all the zero entries appearing at the end, as there are no leaves in the codeword). Formally, letting $M \leq n - 1$ be the number of distinct entries in the codeword, we define indices $I_1, I_2, \dots, I_M$ inductively by $I_1 = 1$ and
    \[ I_{k+1} = \min\{ i : w_i \not \in \{ w_{I_1}, w_{I_2}, \dots, w_{I_k}\}\}.\]
Then we define an injection $\Sigma : [M] \to [n]$ by $\Sigma(k) = w_{I_k}$ and set
    \[ \hat{\bm D} = (\hat D_1, \hat D_2, \dots, \hat D_n) = (D_{\Sigma(1)}, \dots, D_{\Sigma(M)}, 0, \dots, 0).\]
Note that $\bm D$ has $M$ nonzero entries which are precisely the first $M$ entries of $\hat{\bm D}$. 

We observe the following property of the reordering $\hat{\bm D}$, which we will not prove as it is straightforward:

\begin{proposition_restate}\label{prop: Dhatissizebiased}
    $\hat{\bm D}$ is a size-biased random reordering of $\bm D$: that is, conditionally given $\bm D$, we have
    \[ \Prob{\Sigma(1) = \sigma(1), \dots, \Sigma(k) = \sigma(k) \given\Big \bm D} = \prod_{i = 1}^k \frac{D_{\sigma(i)}}{n-1-\sum_{j=1}^{i-1}D_{\sigma(j)}}\]
for any $k \leq M$.
\end{proposition_restate}

Propositions \ref{prop: pruferisbienayme} and \ref{prop: Dhatissizebiased} show that, provided we have a way to sample $\hat{\bm D}$ directly, we can run the following algorithm to generate an unlabelled, unordered conditioned Bienaymé tree. (We note that this procedure is also used in \cite{snakes} to study Bienaymé trees in the case that $\xi$ has finite variance.) Once again we keep track of an active vertex, which we will call $V_k$. We will also draw half-edges connected to some of our vertices. Let $H_k$ be the set of half-edges after the $k$th step. Most of the interesting behaviour of our algorithm is captured by $V_k$ and $H_k$, but for technical reasons we must also consider a set of `dormant' vertices, denoted $Z_k$. These correspond to vertices created at the ends of branches, whose labels do not (immediately) show up in the codeword. We will arrange that no vertex is simultaneously active and dormant.

We begin at time $0$ with the root as our only vertex, which is marked as active, and with no half-edges or dormant vertices. That is, $V_0 = \rho$, $H_0 = \varnothing$ and $Z_0 = \varnothing$. For $k = 1, 2, \dots$ the $k$th step consists of the following:

\begin{itemize}
    \item Sample a random element $X_k$ from $H_{k-1} \cup Z_{k-1} \cup \{ V_{k-1}\}$, according to the following rules:
    \begin{align*}
        \Prob{X_k = e} &= \frac{1}{n-k} \quad \text{ for all } e \in H_{k-1},\\
        \Prob{X_k = u} &= \frac{n-k-\#H_{k-1}}{(n-k)(n-k+\#Z_{k-1})} \quad \text{ for all } u \in Z_{k-1},\\
        \Prob{X_k = V_{k-1}} &= \frac{n-k-\#H_{k-1}}{n-k+\#Z_{k-1}}.
    \end{align*}
    (It is easy to see that these quantities sum to 1, so for $X_k$ to have a well-defined law we only need to check that each quantity is non-negative. For this it is sufficient to show $\#H_{k-1} \leq n-k$, a fact we will show later.) 
    \item (Branching event) If $X_k = e \in H_{k-1}$ is a half-edge, then we complete this half-edge by creating a new vertex $v$ and connecting it to $e$. The previously active vertex $V_{k-1}$ becomes dormant, and the new vertex becomes the active vertex, i.e. $V_k = v$;
    \item (Growth event) If $X_k = V_{k-1}$, we sample the next entry of $\hat{\bm D}$, say $\hat D_i$. We will show that this quantity is always at least 1. Assuming this for now, we create a new vertex $v$, connect it to $V_{k-1}$ by a (full) edge, then add $\hat D_i - 1$ half-edges to $V_{k-1}$. The dormant set is unchanged from step $k-1$;
    \item (Activation event) If $X_k = u \in Z_{k-1}$ is a dormant vertex, we sample the next entry of $\hat{\bm D}$, say $\hat D_i$. Once again, we assume for now that $\hat D_i \geq 1$. We create a new vertex $v$, connect it to $u$ by a (full) edge, then add $\hat D_i - 1$ half-edges to $u$. The previously active vertex $V_{k-1}$ becomes dormant, $u$ ceases to be dormant and the new vertex $v$ becomes the active vertex.
\end{itemize}

To motivate the choices for the law of $X_k$, observe that the process $\#Z_k$ increases by one in every branching event and stays constant in growth and activation events. Thus $\#Z_{k-1}$ is equal to the number of repeats among the first $k-1$ entries of the codeword, and so $n - k + \#Z_{k-1}$ is the number of labels that have not been used in the codeword yet. Conditionally given that a branching event does not occur (and the probability of this is $\frac{n-k-\#H_{k-1}}{n-k}$), any given dormant vertex becomes active iff the entry in the underlying codeword matches that of the dormant vertex. This codeword entry is conditionally uniform on the set of unused labels in the codeword so far, and so the conditional probability is $\frac{1}{n-k+\#Z_{k-1}}$.

The following calculation shows both that $X_k$ has a valid law and that the sampled $\hat D_i$ in any growth event is at least 1:  among the first $k-1$ steps of the algorithm, $\#Z_{k-1}$ of them are branching events, and the other $k - 1 - \#Z_{k-1}$ are growth or activation events. Each branching event removes a half-edge, while each growth or activation event samples the next entry $\hat D_i$ of $\hat{\bm D}$ and adds $\hat D_i - 1$ half-edges. It follows that $\#H_{k-1} = \sum_{j=1}^{k-1-\#Z_{k-1}}(\hat D_j - 1) - \#Z_{k-1}$. In particular,
    \[ n - k - \#H_{k-1} = n - 1 - \sum_{j=1}^{k-1-\#Z_{k-1}}\hat D_j \geq 0,\]
and since $k-1-\#Z_{k-1}$ is the number of entries of $\hat{\bm D}$ revealed so far, equality holds iff we have already revealed all nonzero entries of $\hat{\bm D}$. Thus, conditional on the first $k-1$ steps, a growth event at step $k$ has positive probability iff there are still positive entries in $\hat{\bm D}$ to reveal, and an activation event at step $k$ has positive probability iff the same condition holds and there exists at least one dormant vertex.

Note that $\#Z_k \leq k$, and so if $k \ll n$ the probability of activation events is negligible. In fact, for $m \ll \sqrt{n}$ we see that with high probability no activation events occur during the first $m$ steps of the process. For this reason, it will turn out that we can usually work on the event that no activation events occur on the time intervals we are interested in. 

\subsection{Relation to size-biased distributions}
In light of the previous section, we seek to understand the law of the size-biased random reordering $\hat{\bm D}$. We will show that this vector is related by a measure change to an i.i.d.\ sequence of random variables $(\xi_i^*)$ with the size-biased distribution: that is, $\xi_1^*, \xi_2^*, \dots, \xi_n^*$ are i.i.d.\ with law
    \[ \Prob{\xi_1^* = k} = \frac{k\Prob{\xi = k}}{\E{\xi}} = k\Prob{\xi = k}.\]
We will mostly be interested in the first $m$ entries of these vectors, where $m = o(n)$ as $n \to \infty$. We can only relate the laws on the event $\{ N^n \geq m\}$, where $N^n$ is the number of non-zero entries of $\hat{\bm D}$, but in our case this event will turn out to have probability tending to 1. In what follows, write $\Xi_n$ to denote a sum of $n$ i.i.d.\ copies of $\xi$ (note that this is a sum of non-size-biased offspring distributions).

\begin{proposition_restate}\label{prop: measurechangediscrete}
    For $1 \leq m < n$ and any bounded test function $f : \mathbb{Z}^m \to \R$, we have
        \[ \E{f(\hat D_1, \dots, \hat D_m)\IndEvent{N^n \geq m}} = \E{f(\xi_1^*, \dots, \xi_m^*)\Theta_m^n(\xi_1^*, \dots, \xi_m^*)}\]
    where $\Theta_m^n$ is the function given by
        \[ \Theta_m^n(k_1, \dots, k_m) = \frac{\Prob{\Xi_{n-m} = n - 1 - \sum_{i=1}^m k_i}}{\Prob{\Xi_n = n-1}}\prod_{i=1}^m\left(\frac{n-i+1}{n-1-\sum_{j=1}^{i-1}k_j}\right).\]
\end{proposition_restate}
\begin{proof}
    By linearity it is enough to consider the case 
    \[
    f(x_1, \dots, x_m) = \IndEvent{x_1 = r_1, \dots, x_m = r_m},
    \]
    where $r_1, \dots, r_m$ are all at least 1 and sum to at most $n-1$. For such a choice of $f$ we have
    \[ \E{f(\hat D_1, \dots, \hat D_m)\IndEvent{N^n \geq m}} = \Prob{\hat D_1 = r_1, \dots, \hat D_m = r_m}\]
    and
    \[ \E{f(\xi_1^*, \dots, \xi_m^*)\Theta_m^n(\xi_1^*, \dots, \xi_m^*)} = \Theta_m^n(r_1, \dots, r_m)\Prob{\xi_1^* = r_1, \dots, \xi_m^* = r_m}.\]
    The event $A = \{ \hat D_1 = r_1, \dots, \hat D_m = r_m \}$ occurs iff there is an $m$-tuple $\bm i = (i_1, \dots, i_m)$ of distinct indices in $[n]$ such that
        \[ D_{i_k} = r_k \text{ for } 1 \leq k \leq m \quad \text{ and } \quad \Sigma(k) = i_k \text{ for } 1 \leq k \leq m.\]
   Let us call this event $A_{\bm i}$. Note that the events $(A_{\bm i})$ are disjoint, and so $\Prob{A} = \sum_{\bm i}\Prob{A_{\bm i}}$. The number of valid index tuples is $\frac{n!}{(n-m)!}$, and by exchangeability the corresponding events all have the same probability. Thus we can consider a fixed index tuple $\bm{i}^* = (1, 2, \dots, m)$, and now $\Prob{A} = \frac{n!}{(n-m)!}\Prob{A_{\bm{i}^*}}$. For brevity, we will write $A^* = A_{\bm{i}^*}$.

   For the event $A^*$ to occur, we first require that $D_1 = r_1, \dots, D_m = r_m$ and the remaining $n-m$ entries sum to $n - 1 - \sum_{j=1}^m r_j$. This occurs with probability
    \[\frac{\Prob{\Xi_{n-m} = n - 1 - \sum_{i=1}^m r_i}}{\Prob{\Xi_n = n-1}} \prod_{i=1}^m p_{r_i}.\]
   Conditionally given this, the probability that $\Sigma(1) = 1, \dots, \Sigma(m) = m$ is $\prod_{i=1}^m \frac{r_i}{n - 1 - \sum_{j=1}^{i-1}r_j}$ by \autoref{prop: Dhatissizebiased}. It follows that

\begin{align*}
    \Prob{A} &= \frac{n!}{(n-m)!} \cdot \frac{\Prob{\Xi_{n-m} = n-1-\sum_{i=1}^m r_i}}{\Prob{\Xi_n = n-1}} \cdot \prod_{i=1}^m p_{r_i} \cdot \prod_{i=1}^m\frac{r_i}{n-1-\sum_{j=1}^{i-1}r_j}\\
    &= \prod_{i=1}^m (n-i+1) \cdot \frac{\Prob{\Xi_{n-m} = n-1-\sum_{i=1}^m r_i}}{\Prob{\Xi_n = n-1}} \cdot \prod_{i=1}^m (r_ip_{r_i}) \prod_{i=1}^m\left(\frac{1}{n-1-\sum_{j=1}^{i-1}r_j}\right)\\
    &= \frac{\Prob{\Xi_{n-m} = n-1-\sum_{i=1}^m r_i}}{\Prob{\Xi_n = n-1}} \prod_{i=1}^m\left(\frac{n-i+1}{n-1-\sum_{j=1}^{i-1}r_j}\right) \prod_{i=1}^m (r_ip_{r_i})\\
    &= \Theta_m^n(r_1, \dots, r_m) \Prob{\xi_1^* = r_1, \dots, \xi_m^* = r_m},\end{align*}
completing the proof.
\end{proof}

This result gives us a general strategy to study expectations and probabilities associated to our discrete line-breaking construction: first, we use the tower law to reduce the problem to studying a conditional expectation given $\hat{\bm D}$. Then we observe that this conditional expectation is some (deterministic) function of $\hat{\bm D}$. Finally, using the measure change, we replace $\hat{\bm D}$ with an i.i.d.\ sequence of size-biased distributions and insert a factor of $\Theta_m^n$ into the expectation.

As an example, consider the length $Y_1^n$ of the first stick in the construction. (We define the length of a stick to mean the number of edges it contains). We note that $Y_1^n = C_1 - 1$, so in particular the first stick has length at least $k$ iff the first $k$ steps in the algorithm are all growth events. Conditionally given $\hat{\bm D}$, we see
\begin{itemize}
    \item The first step is a growth event with probability 1;
    \item Given that the first $j - 1$ steps are all growth events, we have $\# H_{j-1} = \sum_{i=1}^{j-1}(\hat D_i - 1)$, and so the $j$th step is also a growth event with probability $1 - \frac{\sum_{i=1}^{j-1}(\hat D_i - 1)}{n-j}$.
\end{itemize}
It follows that
    \[ \Prob{Y_1^n \geq k \given\Big \hat{\bm D}} = \prod_{j=1}^k \left(1-\frac{\sum_{i=1}^{j-1}(\hat D_i - 1)}{n-j}\right)\]
Taking any $m \geq k$, and setting $f(x_1, x_2, \dots, x_m) = \prod_{j=1}^k\left(1-\frac{\sum_{i=1}^{j-1}(x_i-1)}{n-j}\right)$, the measure change result applied to $f$ implies that
    \[ \Prob{ Y_1^n \geq k, N^n \geq m } = \E{\prod_{j=1}^k\left(1-\frac{\sum_{i=1}^{j-1}(\xi_i^*-1)}{n-j}\right)\Theta_m^n(\xi_1^*, \dots, \xi_m^*)}.\]

\subsection{Heuristics for scaling limits in the $\alpha$-stable case}

%We aim to study scaling limits of our discrete construction. We will consider the particular case where the offspring distribution has tail asymptotic $\Prob{\xi = k} \sim Ck^{-\alpha-1}$, for some $\alpha \in (1, 2)$. Note that then $\xi^*$ has tail $\Prob{\xi^* = k} \sim Ck^{-\alpha}$. It is already known that the $\alpha$-stable tree arises as a scaling limit \cite{duquesnealphalimit}; we aim to give an alternative characterisation of this limit.

In the next section, we will give a full proof of \autoref{thm:scalinglimit}. We will sketch a proof here in the more restrictive setting where $\Prob{\xi = k} \sim Ck^{-\alpha-1}$ for some constant $C > 0$. Note that then $\xi^*$ is such that $\Prob{\xi^* = k} \sim Ck^{-\alpha}$. Following the strategy outlined in the previous section, we seek limits in distribution for various functions involving $\xi^*$. Let us begin with the measure change $\Theta_m^n = \Theta_m^n(\xi_1^*, \dots, \xi_m^*)$. First we should determine how $m$ should grow with $n$. We can study the probability 
    \[\Prob{\Xi_{n-m} = n - 1 - \sum_{i=1}^m \xi_i^* \given\Big \xi^*}\]
to get an idea: assuming $m \ll n$ we know that $\frac{\Xi_{n-m} - n + m}{n^{1/\alpha}}$ converges in distribution to a spectrally positive $\alpha$-stable random variable $L_1$, say with density $p$. (Note that this may not be quite the same $L_1$ defined in \autoref{sec: intro} --- there may be a constant scaling involved. We will deal with this more precisely later). We may write
    \[ \Prob{\Xi_{n-m} = n - 1 - \sum_{i=1}^m \xi_i^* \,\Bigg|\, \xi^*} = \Prob{\frac{\Xi_{n-m}-n+m}{n^{1/\alpha}} = -\frac{1}{n^{1/\alpha}}\sum_{i=1}^m (\xi_i^* - 1) - \frac{1}{n^{1/\alpha}} \,\Bigg|\, \xi^*} \]
and now, by the local limit theorem for $\Xi$, if $\frac{1}{n^{1/\alpha}}\sum_{i=1}^{m}(\xi_i^* - 1)$ converges in distribution to some random variable $Z$, then the above probability is asymptotic to $n^{-1/\alpha}p(-Z)$. Meanwhile the probability $\Prob{\Xi_n = n-1}$ is asymptotic to $n^{-1/\alpha}p(0)$, and so the ratio converges to $p(-Z)/p(0)$.

We know that $k^{-\frac{1}{\alpha-1}}\sum_{i=1}^{\lfloor tk \rfloor} (\xi_i^* - 1)$ converges in distribution as $k \to \infty$ to $\sigma_t$, an $(\alpha-1)$-stable subordinator (again possibly scaled by a constant). Thus we should choose $m = \lfloor tk \rfloor$, where $k = k(n)$ is such that $k^{\frac{1}{\alpha-1}} = n^{1/\alpha}$. That is, $k = n^{1-\frac 1\alpha}$, and $m = \lfloor tn^{1-\frac 1\alpha}\rfloor$. Varying $t$ will allow us to consider growing trees over time.

With this choice of $m$, let us estimate the product in $\Theta_m^n$ as
\begin{align*}
    \prod_{i=1}^{\lfloor tn^{1-\frac1\alpha}\rfloor}\left(\frac{n-i+1}{n-1-\sum_{j=1}^{i-1}\xi_j^*}\right) 
    %&= \prod_{i=1}^{\lfloor tn^{1-\frac1\alpha}\rfloor}\left(\frac{n-1-\sum_{j=1}^{i-1}\xi_j^*}{n-i+1}\right)^{-1}\\
    &= \prod_{i=1}^{\lfloor tn^{1-\frac1\alpha}\rfloor}\left(1-\frac{\sum_{j=1}^{i-1}(\xi_j^*-1) + 1}{n-i+1}\right)^{-1}\\
    &= \exp\left(-\sum_{i=1}^{\lfloor tn^{1-\frac1\alpha}\rfloor}\log\left(1-\frac{\sum_{j=1}^{i-1}(\xi_j^*-1)+1}{n-i+1}\right)\right)\\
    %&\approx \exp\left(\sum_{i=1}^{\lfloor tn^{1-\frac1\alpha}\rfloor}\frac{\sum_{j=1}^{i-1}(\xi_j^*-1)+1}{n-i+1}\right)\\
    &\approx \exp\left(\frac 1n\sum_{i=1}^{\lfloor tn^{1-\frac1\alpha}\rfloor}\sum_{j=1}^{i-1}(\xi_j^*-1)\right)\\
    &\approx \exp\left(\frac{1}{n^{1-\frac1\alpha}}\sum_{i=1}^{\lfloor tn^{1-\frac1\alpha}\rfloor}\left(\frac{1}{n^{\frac1\alpha}}\sum_{j=1}^{i-1}(\xi_j^*-1)\right)\right)\\
    %&\approx \exp\left(\frac{1}{n^{1-\frac1\alpha}}\sum_{i=1}^{\lfloor tn^{1-\frac1\alpha}\rfloor}\sigma_{i/n^{1-\frac1\alpha}}\right)\\
    &\approx \exp\left(\int_0^t \sigma_s\,ds\right).
\end{align*}
Putting everything together, we expect that
    \[ \Theta_{\lfloor tn^{1-\frac1\alpha}\rfloor}^n \xrightarrow{d} \exp\left(\int_0^t \sigma_s\,ds\right)\frac{p(-\sigma_t)}{p(0)}.\]
We will prove a slightly stronger statement in the next section. As the process
    \[ t \mapsto \frac{1}{n^{1/\alpha}}\sum_{j=1}^{\lfloor tn^{1-\frac 1\alpha}\rfloor}(\xi_i^* - 1)\]
converges in distribution to $\sigma_t$, we then expect the process given by
    \[ t \mapsto \frac{1}{n^{1/\alpha}}\sum_{j=1}^{\lfloor tn^{1-\frac1\alpha}\rfloor}(\hat{D_i}-1)\]
to converge to the measure-changed version $\widetilde \sigma$ of $\sigma$. 
%such that for any suitable test function $F$ we have $\E{F((\widetilde \sigma_s)_{s \leq t})} = \E{F((\sigma_s)_{s \leq t})\exp(\int_0^t \sigma_s\,ds)\frac{p(-\sigma_t)}{p(0)}}$. 

Assuming that branching events are rare in the discrete process, the partial sums of $\hat{\bm D}$ approximately give the number of free half-edges attached to the subtree present at each step of the construction. Near the beginning of the process, the probability of a branching event is roughly $1/n$ times this number of free half-edges, and so $\hat{\bm D}$ is playing an analogous role to the intensity function in a line-breaking construction. Thus, we expect $\widetilde \sigma$ to act as the intensity for a line-breaking construction that yields the $\alpha$-stable tree. Indeed, looking back to our first stick length example, if we take $k = \lfloor sn^{1-\frac1\alpha}\rfloor$ for some $s \leq t$, then we see
    \begin{multline*} \Prob{Y_1^n \geq \lfloor sn^{1-\frac1\alpha}\rfloor, N^n \geq \lfloor tn^{1-\frac1\alpha}\rfloor} \\
    = \E{\prod_{j=1}^{\lfloor sn^{1-\frac1\alpha}\rfloor}\left(1-\frac{\sum_{i=1}^{j-1}(\xi_i^*-1)}{n-j}\right)\Theta_{\lfloor tn^{1-\frac1\alpha}\rfloor}^n\left(\xi_1^*, \dots, \xi_{\lfloor tn^{1-\frac1\alpha}\rfloor}^*\right)}\end{multline*}
and by similar calculations to the ones used for $\Theta_m^n$, the product term should converge to $\exp(-\int_0^s \sigma_r\,dr)$. As the $\Theta$ term converges to the measure change, the limit of this expectation should be $\E{\exp(-\int_0^s \widetilde \sigma_r\,dr)}$, which is exactly the probability that the first stick has length $\geq s$ in the line-breaking construction driven by $\widetilde \sigma$.

\section{Convergence of discrete trees: a new proof of \autoref{thm:scalinglimit}} \label{sec: cvgdisc}

Our goal in this section is to show that the discrete trees introduced in \autoref{sec: discrete} converge on rescaling to the continuum tree $(\Tcal,d,\mu)$ introduced in \autoref{sec: intro}, thus giving a proof of \autoref{thm:scalinglimit}.

Recall that we assume the offspring distribution $\xi$ is in the domain of attraction of a stable law of index $\alpha$. As in the previous section, we let $\Xi_n = \sum_{i=1}^n \xi_i$, where $\xi_i$ are i.i.d.\ copies of $\xi$. By Theorem 1 of Kortchemski~\cite{igorduquesneproof}, there exists an increasing sequence $a_n > 0$, regularly varying of index $1/\alpha$, such that $\frac{\Xi_n - n}{a_n}$ converges in distribution to $L_1$, where $L$ is the spectrally positive $\alpha$-stable Lévy process, and moreover we have the local limit theorem
    \[ \lim_{n \to \infty} \sup_{k \in \mathbb Z}\left|a_n\Prob{\Xi_n - n = k} - p\left(\frac{k}{a_n}\right)\right| = 0, \]
where $p$ is the density of $L_1$. We can also obtain a scaling limit for the random walk with jumps following the size-biased distribution $\xi^*$. In what follows, let $m_n = n/a_n$.

\begin{proposition_restate}\label{prop: sizebiasedscalinglimit}
    For any $t > 0$, we have the convergence in distribution
        \[ \frac{1}{a_n}\sum_{i=1}^{\fl{tm_n}}\xi_i^* \overset d\to \sigma_t,\]
    where $\sigma$ is an $(\alpha-1)$-stable subordinator with Lévy measure $C_\alpha x^{-\alpha}\,dx$, $x > 0$. (Note that $C_\alpha = \frac{\alpha(\alpha-1)}{\Gamma(2-\alpha)}$ is the same constant appearing in the Lévy measure of $L$. In particular, we have $\E{e^{-\lambda \sigma_t}} = \exp(-t\alpha \lambda^{\alpha - 1}).$)
\end{proposition_restate}
\begin{proof}
    As $\xi$ is in the domain of attraction of an $\alpha$-stable law, there exists a slowly varying function $l(t)$ such that $\Prob{\xi \geq t} = l(t)t^{-\alpha}$ for all $t > 0$. By the proof used in Theorem 3.8.2 in \cite{durrettprobth}, we also know $n\Prob{\xi \geq xa_n} \to \int_x^\infty C_\alpha t^{-\alpha-1}\,dt = \frac{C_\alpha}{\alpha}x^{-\alpha}$ for any fixed $x$. Combining these results, we deduce that
        \[ l(xa_n) \sim \frac{C_\alpha a_n^\alpha}{\alpha n}\]
    as $n \to \infty$ with $x$ fixed. Now we may calculate
        \[ \Prob{\xi^* \geq t} = t\Prob{\xi \geq t} + \int_t^\infty \Prob{\xi \geq x}\,dx \]
    and so by Karamata's integral theorem (Proposition 1.5.10 of \cite{RegularVariation}) we have $\Prob{\xi^* \geq t} \sim \frac{\alpha}{\alpha - 1}l(t)t^{-\alpha + 1}$. In particular, for fixed $x$ we have
        \[ \Prob{\xi^* \geq xa_n} \sim \frac{\alpha}{\alpha - 1}l(xa_n)(xa_n)^{-\alpha+1} \sim \frac{1}{m_n} \frac{C_\alpha}{\alpha-1}x^{-\alpha+1} = \frac{1}{m_n} \int_x^\infty C_\alpha t^{-\alpha}\,dt\]
    and so, following the proof in Theorem 3.8.2 in \cite{durrettprobth}, we deduce the desired convergence in distribution.
\end{proof}

We can extend this convergence result to a statement about convergence of c\`adl\`ag processes: we may define, for each $n$, a c\`adl\`ag process $\sigma^n(t)$ such that
    \[ \sigma^n(t) \overset{d}{=} \frac 1{a_n}\sum_{i=1}^{\lfloor tm_n\rfloor}(\xi_i^*-1), \quad 0 \leq t \leq T\]
and now $\sigma^n \overset d\to \sigma$ as elements of the Skorokhod space $D([0, T])$. (We could equally have defined $\sigma^n$ as a sum of $\xi_i^*$, without subtracting 1 from each, and obtained the same limit, but it is convenient to include the subtraction as it simplifies calculations in the rest of this section.)  By the Skorokhod representation theorem we may construct $\sigma^n$ (for all $n$) and $\sigma$ on a common probability space such that $\sigma^n \to \sigma$ a.s.\ (with convergence holding in the Skorokhod J1-metric). For the rest of this section we will work in this setting. It will be useful to name the relevant copies $\xi_i^{n, *}$ of the size-biased distribution, so that
    \[ \sigma^n(t) = \frac{1}{a_n}\sum_{i=1}^{\fl{tm_n}}(\xi_i^{n,*}-1).\]

The following lemma on $N^n$, the number of vertices of out-degree at least 1, will be helpful when applying the measure change later:

\begin{lemma_restate}\label{lem: badeventrare}
    For any $T > 0$ we have $\Prob{N^n < \lfloor Tm_n\rfloor} \to 0$.
\end{lemma_restate}
\begin{proof}
    If $\xi_1, \dots, \xi_n$ are i.i.d.\ copies of the (non-size-biased) offspring distribution, then $N_n$ has the law of $\#\{i \leq n : \xi_i \geq 1\}$ conditional on $\sum_{i=1}^n \xi_i = n-1$. Thus we can crudely estimate
    \[ \Prob{N^n < \lfloor Tm_n\rfloor} \leq \frac{\Prob{\#\{i \leq n : \xi_i \geq 1\} < \lfloor Tm_n\rfloor}}{\Prob{\sum_{i=1}^n \xi_i = n - 1}}.\]
    By construction $m_n \ll n$ so the numerator decays exponentially in $n$ by e.g.\ Chernoff bounds for binomial distributions, while the denominator only decays polynomially. Thus we have a bound tending to 0.
\end{proof}

\subsection{Convergence of the measure change}
The key result in this section follows.

\begin{theorem}\label{thm: measurechangeconv}
    Fix a finite time horizon $T > 0$. Set $M^n_t = \Theta_{\fl{tm_n}}^n(\xi_1^*, \dots, \xi_{\fl{tm_n}}^*)$ for $0 \leq t \leq T$, and let $M_t = \exp(\int_0^t \sigma_s\,ds)p(-\sigma_t)/p(0)$. Then $M^n \overset d\to M$ in the Skorokhod space $D([0, T])$.
\end{theorem}

In what follows, and in all subsequent proofs, we will abuse notation and write $\Prob{\Xi_n = Z}$ as a shorthand for the random variable $\Psi_n(Z)$, where $\Psi_n$ is the probability mass function of $\Xi_n$. We can also view this as the conditional probability $\Prob{\Xi_n = Z \given\Big Z}$ with $\Xi_n$ taken to be independent of $Z$.

\begin{proof}
    As discussed earlier, we work in the setting that $\sigma^n$ and $\sigma$ are defined on a common space with $\sigma^n \to \sigma$ a.s. Define 
        \[ M_t^n = \Theta_{\fl{tm_n}}^n\left(\xi_1^{n, *}, \dots, \xi_{\fl{tm_n}}^{n, *}\right);\]
    we will show that $M^n \to M$ a.s.\ in the Skorokhod metric.

    We begin by handling the product term: we may write
    \begin{align*}
        \prod_{i=1}^{\fl{tm_n}}\left(\frac{n-i+1}{n-1-\sum_{j=1}^{i-1}\xi_j^{n,*}}\right) &= \prod_{i=1}^{\fl{tm_n}}\left(1-\frac{\sum_{j=1}^{i-1}(\xi_j^{n,*}-1)+1}{n-i+1}\right)^{-1}\\
        &= \exp\left(-\sum_{i=1}^{\fl{tm_n}}\log\left(1-\frac{\sum_{j=1}^{i-1}(\xi_j^{n,*}-1)+1}{n-i+1}\right)\right)\\
        &= \exp\left(-\int_0^{\frac{\fl{tm_n}}{m_n}}m_n\log\left(1-\frac{\sum_{j=1}^{\fl{sm_n}}(\xi_j^{n,*}-1)+1}{n-\fl{sm_n}}\right)\,ds\right)\\
        &= \exp\left(-\int_0^{\frac{\fl{tm_n}}{m_n}}m_n\log\left(1-\frac{a_n\sigma^n(s)+1}{n-\fl{sm_n}}\right)\,ds\right).
    \end{align*}
    We would like this to converge to $\exp(\int_0^t \sigma_s\,ds)$ and so we will prove that the integrand converges to $-\sigma_s$ a.s.\ (again, in the Skorokhod metric). As $\sigma^n \to \sigma$ a.s.\ in the Skorokhod metric, it is enough to show that
        \[ \left(m_n\log\left(1-\frac{a_n\sigma^n(s)+1}{n-\fl{sm_n}}\right) + \sigma^n(s)\right)_{s \in [0, T]} \to 0 \quad \text{ a.s.\ in the Skorokhod metric,}\]
    which amounts to showing uniform convergence to 0. This is easily shown using a Taylor expansion and the observation that, for any $s$, $\sigma^n(s) \to \sigma_s$ a.s.\ and this quantity is at most $\sigma_T$. Now, as integration is a continuous operation on $D([0, T])$, we obtain
        \[ -\int_0^t m_n\log\left(1-\frac{a_n\sigma^n(s)+1}{n-\fl{sm_n}}\right)\,ds \to \int_0^t \sigma_s\,ds \]
    almost surely, again with respect to the metric on $D([0, T])$. The left-hand side is not quite of the correct form as the upper limit of integration is slightly wrong (we have $t$ instead of the correct $\fl{tm_n}/m_n$), but this is easily remedied by showing that the difference tends to 0 uniformly a.s. --- that is,
        \[ \sup_{t \in [0, T]}\left|\int_{\frac{\fl{tm_n}}{m_n}}^t m_n\log\left(1-\frac{a_n\sigma^n(s)+1}{n-\fl{sm_n}}\right)\,ds\right| \to 0 \quad \text{ a.s.}\]
    This is again a straightforward calculation, so we omit it. Now we finish by exponentiating (again a continuous operation on $D([0, T])$).

    Now we can handle the ratio of probabilities: writing $k = k(t) = \fl{tm_n}$ for brevity, we observe that
        \[ \frac{\Prob{\Xi_{n - k} = n - 1 - \sum_{i=1}^{k}\xi_i^*}}{\Prob{\Xi_n = n-1}} = \frac{a_n\Prob{\Xi_{n - k} - (n - k) = -a_n\sigma^n(t) - 1}}{a_n\Prob{\Xi_n - n = -1}}. \]
    By a straightforward application of the local limit theorem the denominator (which is non-random) converges to $p(0)$.

    We handle the numerator with another application of the local limit theorem: setting $S^n = \sup_{\ell \in \mathbb Z}\left|a_n\Prob{\Xi_n - n = \ell} - p(\ell/a_n)\right| \to 0$, we have
    \begin{multline*}
        \left|a_n\Prob{\Xi_{n -k} - (n - k) = -a_n\sigma^n(t) - 1} - \frac{a_n}{a_{n-k}}p\left(-\frac{a_n}{a_{n-k}}\sigma^n(t)-\frac{1}{a_{n-k}}\right)\right|\\
        \leq \frac{a_n}{a_{n-k}}S^{n-k} \quad \leq \sup_{t \in [0, T]} \left\{\frac{a_n}{a_{n-k(t)}}S^{n-k(t)}\right\},
    \end{multline*}
    and similarly
    \[
        \left|\frac{a_n}{a_{n-k}}p\left(-\frac{a_n}{a_{n-k}}\sigma^n(t)-\frac{1}{a_{n-k}}\right) - p\left(-\frac{a_n}{a_{n-k}}\sigma^n(t)-\frac{1}{a_{n-k}}\right)\right| \leq \sup_{t \in [0, T]}\left|\frac{a_n}{a_{n-k(t)}}-1\right| \|p\|_\infty.
    \]
    Hence the processes
        \[ \Bigg(a_n\Prob{\Xi_{n -k(t)} - (n - k(t)) = -a_n\sigma^n(t) - 1}\Bigg)_{t \in [0, T]} \]
    and
        \[ \left(p\left(-\frac{a_n}{a_{n-k(t)}}\sigma^n(t)-\frac{1}{a_{n-k(t)}}\right)\right)_{t \in [0, T]}\]
    are a.s.\ uniformly close to each other. Thus it suffices to show the latter converges a.s.\ to $p(-\sigma_t)$, which is clear by continuity of $p$. 
\end{proof}

Note that $\E{|M_t^n|} = \E{M_t^n} = \Prob{N^n \geq \lfloor Tn^{1-\frac1\alpha}\rfloor} \to 1 = \E{|M_{t}|}$, and so by Scheffé's lemma we have $\E{|M_t^n - M_t|} \to 0$ in the setting that everything is defined on the same space.

With this convergence established, we can study limits of conditional probabilities given $\hat{\bm D}$ by replacing the entries $\hat D_i$ with i.i.d.\ entries $\xi_i^*$. In an abuse of notation, we will write $\Prob{\cdot \given\Big \hat{\bm D} = \xi^*}$ to denote this alternative setting.

The processes $t \mapsto \frac{1}{a_n}\sum_{i=1}^{\fl{m_nt}}\hat D_i$ are obtained as measure changes of $\sigma^n$ by $M^n$. As we can construct a common space on which both $\sigma^n \to \sigma$ and $M^n \to M$ almost surely, we deduce that $\left(\frac{1}{a_n} \sum_{i=1}^{\fl{m_nt}}\hat D_i, t \leq T\right) \convdist \left(\widetilde \sigma_t, t \leq T\right)$. As the time horizon $T$ was arbitrary, we can extend this to a convergence result on $D([0, \infty))$:

\begin{corollary_restate} \label{cor: cumuldeg}
As $n \to \infty$,
\[
\left( \frac{1}{a_n} \sum_{i=1}^{\fl{m_n t}} \hat{D}_i, t \ge 0\right) \convdist (\widetilde{\sigma}_t, t \ge 0)
\]
for the space of c\`adl\`ag functions $D([0,\infty))$ endowed with the Skorokhod topology on $[0, \infty)$.
\end{corollary_restate}

\subsection{Convergence of random finite-dimensional distributions}
Recall that $C_1, C_2, \dots$ are the positions of repeated entries, i.e. times of branching events (but not activation events), in the codeword for the discrete construction. We will write $C_1^n, C_2^n, \dots$ for clarity. We also introduce the attachment points $J_1^n, J_2^n, \dots$ defined as follows: if the branching event at time $C_i^n$ selects a half-edge with vertex $v$, then $J_i^n$ is the first time that the label of $v$ appears in the codeword. We aim to show that, for any $k \geq 1$,
    \[ \frac{1}{a_n}\left(C_1^n, C_2^n, \dots, C_k^n, J_1^n, \dots, J_k^n\right) \overset{d}{\to} (Y_1, Y_2, \dots, Y_k, Z_1, \dots, Z_k)\]
where $Y_i$ are the jump times of a Poisson process of intensity $\widetilde \sigma$, and conditionally given $\widetilde \sigma$ and $(Y_1, \dots, Y_k)$, the random variables $(Z_1,\ldots,Z_k)$ are independent and $Z_i$ has c.d.f.\ $\frac{\widetilde \sigma_t}{\widetilde \sigma_{Y_i-}}, t \in [0, Y_i)$. Our strategy of proof is similar to that of Section 4 of \cite{snakes}. 

As alluded to before, it will be easier to work in the non-measure changed setting --- we will show the following analogous result first:

\begin{proposition_restate}\label{prop: fdmarginaldensity}
    Let $k \geq 0$ and let $s_1, \dots, s_k, t_1, \dots, t_k$ be fixed with $s_i \leq t_i \leq T$ for all $i$. Define processes $Q^{n,k}_t, Q^k_t$ for $t \in [0, T]$ by
        \[ Q_t^{n,k} = m_n\Prob{\frac{J_1^n}{m_n} \leq s_1, \dots, \frac{J_k^n}{m_n} \leq s_k, \frac{C_1^n}{m_n} \leq t_1, \dots, \frac{C_k^n}{m_n}\leq t_k, C_{k+1}^n = \fl{t m_n} \given\Big \hat{\bm D} = \xi^*},\]
    and
        \[ Q_t^k = \sigma_{t-} \exp\left(-\int_0^t \sigma_u\,du\right) \! \int_{\prod_{i=1}^k[0,t_i]} \! \! \!\IndEvent{u_1 < u_2 < \dots < u_k < t}\prod_{i=1}^k\left(\sigma_{u_i-}\wedge\sigma_{s_i}\right)du_1\dots du_k,\]
    with the convention that the product integral in $Q_t^0$ is equal to 1. Then $Q^{n,k} \overset d\to Q^k$ in $D([0, T])$.
\end{proposition_restate}
Before we prove this lemma, we give some motivation for choosing this statement to prove. If $\eta_1 < \eta_2 < \dots$ are the jump times of a Poisson process of intensity $\sigma$, then we know
    \[ \Prob{\eta_1 \in du_1, \dots, \eta_k \in du_k \,\big|\, \sigma} = \IndEvent{u_1 < \dots < u_k}\sigma_{u_1}\dots\sigma_{u_k}\exp\left(-\int_0^{u_k}\sigma_r\,dr\right)\,du_1\dots du_k.\]
Now adding in attachment points $\zeta_1, \zeta_2, \dots$, again such that conditionally given $\sigma$ and $(\eta_j)_{j \geq 1}$, the random variables $(\zeta_1, \zeta_2, \ldots)$ are independent and $\zeta_i$ has c.d.f.\ $\sigma_t /\sigma_{\eta_i-}, t \in [0, \eta_i)$, we obtain
    \begin{multline*}
        \Prob{\eta_1 \in du_1, \dots, \eta_k \in du_k, \zeta_1 \leq s_1, \dots, \zeta_{k-1} \leq s_{k-1} \,\big|\, \sigma} \\= \IndEvent{u_1 < \dots < u_k}\sigma_{u_k} \prod_{i=1}^{k-1}\left(\left(\sigma_{s_i} \wedge \sigma_{u_i-}\right)\frac{\sigma_{u_i}}{\sigma_{u_i-}}\right)\exp\left(-\int_0^{u_k}\sigma_r\,dr\right)\,du_1\dots du_k.
    \end{multline*}
Replacing $k$ by $k+1$, renaming $u_{k+1}$ to $t$ and integrating over $0 \leq u_i \leq t_i, 1 \leq i \leq k$ yields
    \begin{multline*}
        \Prob{\eta_1 \leq t_1, \dots, \eta_k \leq t_k, \eta_{k+1} \in dt, \zeta_1 \leq s_1, \dots, \zeta_k \leq s_k}\\= dt\,\sigma_t \exp\left(-\int_0^t\sigma_u\,du\right) \! \! \int_{\prod_{i=1}^k[0, t_i]}\! \! \IndEvent{u_1 < \dots < u_k < t}\prod_{i=1}^k\left(\left(\sigma_{s_i} \wedge \sigma_{u_i-}\right)\frac{\sigma_{u_i}}{\sigma_{u_i-}}\right) \!du_1\dots du_k
    \end{multline*}
which is almost in the form of $Q_t^k\,dt$, except for the factors of $\sigma_{u_i} / \sigma_{u_i-}$ in the integral. However, $\sigma$ has countably many discontinuities, and so $\prod_{i=1}^k \frac{\sigma_{u_i}}{\sigma_{u_i-}} = 1$ Lebesgue-almost everywhere on $\prod_{i=1}^k[0, t_i]$. Thus we have exactly
    \[ \Prob{\eta_1 \leq t_1, \dots, \eta_k \leq t_k, \eta_{k+1} \in [t, t+dt), \zeta_1 \leq s_1, \dots, \zeta_k \leq s_k} = Q_t^k\,dt.\]
We will need the following analytic lemma in our proof.
\begin{lemma_restate}\label{lem: j1products}
    Let $f_n, \tilde f_n, g_n, \tilde g_n$ be (deterministic) c\`adl\`ag functions $[0, T] \to \R$. Suppose that
    \begin{enumerate}[(a)]
        \item $\sup_{t \in [0, T]}|f_n(t) - \tilde f_n(t)| \to 0$ as $n \to \infty$;
        \item $\sup_{t \in [0, T]}|g_n(t) - \tilde g_n(t)| \to 0$ as $n \to \infty$;
        \item There exists a constant $C > 0$ such that $\sup_{t \in [0, T]}|\tilde f_n(t)| \leq C$ and $\sup_{t \in [0, T]}|\tilde g_n(t)| \leq C$ for all $n$ sufficiently large.
    \end{enumerate}
    Then
    \begin{enumerate}
        \item $\sup_{t \in [0, T]}\big|f_n(t)g_n(t) - \tilde f_n(t)\tilde g_n(t)\big| \to 0$ as $n \to \infty$;
        \item There exists a constant $C' > 0$ such that all of $\sup_{t \in [0, T]}|f_n(t)|$, $\sup_{t \in [0, T]}|g_n(t)|$ and $\sup_{t \in [0, T]}|\tilde f_n(t)\tilde g_n(t)|$ are at most $C'$.
    \end{enumerate}
\end{lemma_restate}

We do not prove \autoref{lem: j1products} as it is straightforward, but we will make use of it extensively to prove \autoref{prop: fdmarginaldensity}.

\begin{proof}[Proof of \autoref{prop: fdmarginaldensity}]
    We will use the Skorokhod representation theorem again and prove almost sure convergence in the Skorokhod metric. We introduce auxiliary processes $R_t^{n, k}$, defined analogously to $Q_t^k$ but with $\sigma$ replaced by $\sigma^n$. With this we know $R^{n, k} \to Q^k$ in the Skorokhod metric (using the same time change that gives $\sigma^n \to \sigma$), and so it is enough to show that as $n \to \infty$, the difference between $Q_t^{n, k}$ and $R_t^{n, k}$ tends to 0 uniformly in $t$. We seek to prove this stronger convergence result by induction on $k$.

    Dealing with activation events precisely requires handling a great number of cases, but we can exploit the fact that these events are rare in order to effectively ignore them. More precisely, define an auxiliary process $\widetilde Q_t^{n, k}$ by 
    \[ \widetilde Q_t^{n,k} = m_n\Prob{\!\frac{J_1^n}{m_n} \leq s_1, \dots, \frac{J_k^n}{m_n} \leq s_k, \frac{C_1^n}{m_n} \leq t_1, \dots, \frac{C_k^n}{m_n}\leq t_k, C_{k+1}^n = \fl{t m_n}, \mathcal E^{n, k} \! \given\Big \! \hat{\bm D} = \xi^*\!} \]
    where we define the event
    \[ \mathcal E^{n, k} = \{ \text{there are no activation events among the first } \min\{\fl{Tm_n}, C_{k+1}^n\} \text{ steps} \}.\]
    Note that $1 - \Prob{\mathcal E^{n, k}} \leq \frac{(k+1)Tm_n}{n - \fl{Tm_n}}$ for all $n$ and $k$, so for all $n$ and all $k$ we almost surely have the estimate
    \[ \sup_{t \in [0, T]} \left|\widetilde Q_t^{n, k} - Q_t^{n, k}\right| \leq O\left(\frac{m_n^2}{n}\right) = o(1).\]
    We deduce that, for any fixed $k$,
    \[
        \sup_{t \in [0, T]}\left|Q_t^{n, k} - R_t^{n, k}\right| \to 0 \text{ a.s.\ } \iff \sup_{t \in [0, T]}\left|\widetilde Q_t^{n, k} - R_t^{n, k}\right| \to 0 \text{ a.s.\ }
    \]
    First consider the case $k = 0$, so that
        \[ \widetilde Q_t^{n,0} = m_n\Prob{C_1^n = \fl{tm_n}, \mathcal E^{n, 0} \given\Big \hat{\bm D^n} = \xi^{n,*}}.\]
    The event $\{C_1^n = l\} \cap \mathcal E^{n, 0}$ occurs if and only if the first $l-1$ steps are all growth events and step $l$ is a branching event. We know that conditionally given that among the first $j-1$ steps there have been $b \geq 0$ branching events and no activation events, at step $j$ we:
    \begin{itemize}
        \item branch with probability $\frac{\sum_{i=1}^{j-1-b}(\xi_i^{n,*}-1) - b}{n - j}$;
        \item grow with probability $1 - \frac{\sum_{i=1}^{j-1-b}(\xi_i^{n,*}-1)}{n - j + b}.$
    \end{itemize}
    It follows that
        \[ \Prob{C_1^n = \fl{tm_n}, \mathcal E^{n, 0} \given\Big \hat{\bm D^n} = \xi^{n,*}} = \prod_{i=1}^{\fl{tm_n}-1}\left(1-\frac{\sum_{j=1}^{i-1}(\xi_j^{n,*}-1)}{n-i}\right) \times \frac{\sum_{j=1}^{\fl{tm_n}-1}(\xi_j^{n, *}-1)}{n-\fl{tm_n}}.\]
    By an argument very similar to the one in \autoref{thm: measurechangeconv}, the product is uniformly close to $\exp(-\int_0^t\sigma^n(s)\,ds)$. In addition, as $m_n/n = 1/a_n$ we have
        \[ m_n \times \frac{\sum_{j=1}^{\fl{tm_n}-1}(\xi_j^{n,*}-1)}{n-\fl{tm_n}} - \sigma^n(t-) \to 0\]
    uniformly. Both $\exp(-\int_0^t\sigma^n(s)\,ds)$ and $\sigma^n(t-)$ are uniformly bounded in $n$ (as they have c\`adl\`ag limits in the Skorokhod space), so we may apply \autoref{lem: j1products} to get uniform closeness of the products. This establishes almost sure uniform closeness between $\widetilde Q^{n, 0}$ and $R^0$, hence also uniform closeness between $Q^{n, 0}$ and $R^0$.

    Now assume $k \geq 1$ and the uniform closeness holds for the $(k-1)$-dimensional marginals. We observe that conditionally given $\hat{\bm D} = \xi^*$, all of $(C_1^n, \dots, C_k^n, J_1^n, \dots, J_{k-1}^n)$ and that $\mathcal E^{n, k-1}$ holds, the random variables $J_k^n$ and $C_{k+1}^n$ are independent. Letting $\mathcal F_k^n$ be the $\sigma$-algebra associated to this conditioning, the conditional c.d.f.\ of $J_k^n$ is
    \[ \Prob{J_k^n \leq r \given\Big \mathcal F_k^n} = \frac{\sum_{j=1}^{r-B(r)}(\xi^{n,*}_i-1)-B(r)}{\sum_{j=1}^{C^n_k-k}(\xi_i^{n, *}-1)-(k-1)}, \quad 0 \leq r \leq C^n_k-1\]
    where $B(r) = B_{k-1}^n(r) = \#\{1 \leq i \leq k-1 : J_i^n \leq r\}$ (note in particular that $B(C^n_k-1)=k-1$). We would like to say that this probability only depends on $\mathcal F_k^n$ through $C_k^n$, but this is not quite true due to the appearance of $B(r)$, a random function defined in terms of $J_1^n, \dots, J_{k-1}^n$. However, as $0 \leq B(r) \leq k-1$ for all $r$ in the relevant range, we can essentially ignore it: more precisely, we have estimates
    \[ \frac{\sum_{j=1}^{r-k+1}(\xi^{n,*}_j-1)-(k-1)}{\sum_{j=1}^{C^n_k-k}(\xi_j^{n, *}-1)-(k-1)} \leq \Prob{J_k^n \leq r \given\Big \mathcal F_k^n} \leq \frac{\sum_{j=1}^{r}(\xi^{n,*}_j-1)}{\sum_{j=1}^{C_k-k}(\xi_j^{n, *}-1)-(k-1)}\]
    which depend on $\mathcal F_k^n$ only through $C_k^n$. Thus for any nontrivial $\mathcal F_k^n$-measurable event $\mathcal A$ with, say, $\mathcal A \subset \{C^n_k = l\}$, we have the almost sure bounds
        \[ \frac{\sum_{j=1}^{r-k+1}(\xi^{n,*}_j-1)-(k-1)}{\sum_{j=1}^{l-k}(\xi_j^{n, *}-1)-(k-1)} \leq \Prob{J_k^n \leq r \given\Big \mathcal A} \leq \frac{\sum_{j=1}^{r}(\xi^{n,*}_j-1)}{\sum_{j=1}^{l-k}(\xi_j^{n, *}-1)-(k-1)}. \]
    Importantly, this is true only for $0 \leq r \leq l-1$; if $l \leq r$ then the probability is trivially equal to 1.
    
    Similarly, we can handle $C_{k+1}^n$ conditionally given $\mathcal F_k^n$. Conditionally given that $\mathcal E^{n, k-1}$ holds, we know that 
        \begin{multline*} C_{k+1}^n = r \text{ and } \mathcal E^{n, k} \text{ holds} \iff\\ \text{steps } C_k^n + 1, \dots, r-1 \text{ are all growth events and step } r \text{ is a branching event.}\end{multline*}
    This occurs with probability
        \[ \Prob{C_{k+1}^n = r, \mathcal E^{n, k} \given\Big \mathcal F_k^n} = \prod_{i=C_k^n + 1}^{r-1}\left(1-\frac{\sum_{j=1}^{i-1-k}(\xi_j^{n,*}-1)}{n-i+k}\right) \times \frac{\sum_{j=1}^{r-1-k}(\xi_j^{n,*}-1)-k}{n-r},\]
    which depends on $\mathcal F_k^n$ only through $C_k^n$. Thus for any $\mathcal F_k^n$-measurable event $\mathcal A \subset \{C_k^n = l\}$ with $l \leq r-1$ we have
    \[ \Prob{C_{k+1}^n = r, \mathcal E^{n, k} \given\Big \mathcal A} = \prod_{i=l+1}^{r-1}\left(1-\frac{\sum_{j=1}^{i-1-k}(\xi_j^{n,*}-1)}{n-i+k}\right) \times \frac{\sum_{j=1}^{r-1-k}(\xi_j^{n,*}-1)-k}{n-r}. \]
    We are now ready to prove the convergence of $\widetilde Q^{n,k}$: we will give almost sure upper and lower bounds for $\widetilde Q^{n, k}$ which are both uniformly close to $R^{n, k}$. In what follows, it will be helpful to abbreviate
        \[ \mathcal A_{k-1} = \left\{\frac{J_1^n}{m_n} \leq s_1, \dots, \frac{J_{k-1}^n}{m_n} \leq s_{k-1}, \frac{C_1^n}{m_n} \leq t_1, \dots, \frac{C_{k-1}^n}{m_n} \leq t_{k-1}\right\},\]
    so that for example $\widetilde Q_t^{n, k-1} = m_n\Prob{\mathcal A_{k-1}, C^n_{k} = \fl{tm_n}, \mathcal E^{n, k-1} \given{} \hat{\bm D}=\xi^*}$. To prove the upper bound, note
    \begin{align*}
        \widetilde Q_t^{n,k} &= m_n\Prob{\mathcal A_{k-1}, J_k^n \leq \fl{s_k m_n}, C^n_k \leq \fl{t_k m_n}, C^n_{k+1} = \fl{tm_n}, \mathcal E^{n, k} \given\Big \hat{\bm D} = \xi^*}\\
            &= m_n\sum_{l=0}^{\fl{t_km_n}}\Prob{\mathcal A_{k-1}, C^n_k = l, J^n_k \leq \fl{s_k m_n}, C^n_{k+1} = \fl{tm_n}, \mathcal E^{n, k} \given\Big \hat{\bm D} = \xi^*}\\
            &= m_n\sum_{l=0}^{\fl{t_k m_n}}\Prob{\mathcal A_{k-1}, C^n_k = l, \mathcal E^{n, k-1} \given\Big \hat{\bm D} = \xi^*}\\
            &\quad\quad\quad\quad\quad \times \Prob{J^n_k \leq \fl{s_k m_n}, C^n_{k+1} = \fl{tm_n}, \mathcal E^{n, k} \given\Big \mathcal A_{k-1}, C^n_k = l, \mathcal E^{n, k-1}, \hat{\bm D} = \xi^*}\\
            &= \sum_{l=0}^{\fl{t_k m_n}}\widetilde Q_{l/m_n}^{n,k-1}\Prob{J^n_k \leq \fl{s_k m_n}, C^n_{k+1} = \fl{tm_n}, \mathcal E^{n, k} \given\Big \mathcal A_{k-1}, C^n_k = l, \mathcal E^{n, k-1}, \hat{\bm D} = \xi^*}\\
            &\leq \sum_{l=0}^{\fl{t_k m_n}}\frac{\widetilde Q_{l/m_n}^{n, k-1}}{m_n}\left(\frac{\sum_{j=1}^{\fl{s_k m_n}}(\xi_j^{n,*}-1)}{\sum_{j=1}^{l-k}(\xi_j^{n,*}-1)-(k-1)} \wedge 1\right)\\
            &\quad\quad\quad\quad \times \prod_{i=l+1}^{\fl{tm_n}-1}\left(1-\frac{\sum_{j=1}^{i-1}(\xi_j^{n,*}-1)}{n-i+k}\right)\times \frac{m_n\sum_{j=1}^{\fl{tm_n}-1-k}(\xi_j^{n,*}-1)-m_nk}{n-\fl{tm_n}}. 
    \end{align*}
    At this point it is helpful to rewrite these in terms of the discrete processes $\sigma^n(t)$: we have
    \begin{align*}
        & \widetilde Q_t^{n,k} \\
        &\leq \sum_{l=0}^{\fl{t_k m_n}}\frac{\widetilde Q_{l/m_n}^{n,k-1}}{m_n}\! \left(\frac{\sigma^n(s_k)}{\sigma^n(\tfrac{l - k}{m_n})-\frac{m_n(k-1)}{n}}\wedge 1\right) \! \! \prod_{i=l+1}^{\fl{tm_n}-1} \! \! \left(\! 1-\frac{n\sigma^n\left( \tfrac{i-1}{m_n}\right)}{(n-i+k)m_n} \! \right) \!\! \frac{\sigma^n(t-\tfrac{k+1}{m_n}) - \frac{km_n}{n}}{1-\frac{\fl{tm_n}}{n}}\\
        &= \int_0^{t_k} \! \! \widetilde Q_u^{n,k-1} \! \! \left(\frac{\sigma^n(s_k)}{\sigma^n(u - \tfrac{k}{m_n})-\frac{m_n(k-1)}{n}}\wedge 1\right)\!\prod_{i=\fl{um_n}+1}^{\fl{tm_n}-1} \! \left(1-\frac{n\sigma^n\left(\tfrac{i-1}{m_n}\right)}{(n-i+k)m_n}\right)\sigma^n(t - \tfrac{1}{m_n})\,du\\
        &\hspace{12.5cm} + o(1).
    \end{align*}
    where the $o(1)$ term is a.s.\,uniform in $t$. We now aim to estimate the integrand for each $u \in [0, t_k]$. We know that, almost surely,
    \begin{enumerate}[(a)]
        \item $\widetilde Q_u^{n, k-1}$ is uniformly close to $R_u^{n, k-1}$ (by hypothesis);
        \item $\frac{\sigma^n(s_k)}{\sigma^n(u - k/m_n)-\frac{m_n(k-1)}{n}}\wedge 1$ is uniformly close to $\frac{\sigma^n(u-) \wedge \sigma^n(s_k)}{\sigma^n(u-)}$;
        \item $\prod_{i=\fl{um_n}+1}^{\fl{tm_n}-1} \! \! \left(1-\frac{n\sigma^n\left(\tfrac{i-1}{m_n}\right)}{(n-i+k)m_n}\right)$ is uniformly close to $\exp\left(-\int_{u}^t\sigma^n(r)\,dr\right)$, by imitating the argument of \autoref{thm: measurechangeconv};
        \item $\sigma^n(t-(k+1)/m_n)$ is uniformly close to $\sigma^n(t-)$.
    \end{enumerate}
    In addition, all of these estimates are uniform in $u \in [0, t_k]$. The right-hand side of each estimate is uniformly bounded in $n$ (as they converge in the Skorokhod metric). Thus by iteratively applying \autoref{lem: j1products} we see that the integrand of our estimate on $\widetilde Q_t^{n, k}$ is uniformly close to
        \[ u \mapsto R_u^{n, k-1} \frac{\sigma^n(u-) \wedge \sigma^n(s_k)}{\sigma^n(u-)}\exp\left(-\int_u^t\sigma^n(r)\,dr\right)\sigma^n(t-),\]
    and moreover this is uniform in $u \in [0, t_k]$. Using the uniformity in $u$ we may integrate to see that our integral bound on $\widetilde Q_t^{n, k}$ is uniformly (in $t$) close to
        \[ t \mapsto \int_0^{t_k}R_u^{n, k-1} \frac{\sigma^n(u-) \wedge \sigma^n(s_k)}{\sigma^n(u-)}\exp\left(-\int_u^t\sigma^n(r)\,dr\right)\sigma^n(t-)\,du = R_t^{n, k}.\]
    Thus for all $\varepsilon > 0$ there is almost surely an $N$ such that for all $t \in [0, T]$ we have $\widetilde Q_t^{n, k} \leq R_t^{n, k} + \varepsilon$ whenever $n \geq N$. The lower bound is completely analogous, and so $\widetilde Q_t^{n, k}$ is uniformly close to $R_t^{n, k}$, completing the proof.
\end{proof}

By integrating the processes in \autoref{prop: fdmarginaldensity} from $0$ to $t_{k+1}$ and shifting $k$ we deduce that
    \begin{multline*} \Prob{\frac{C_1^n}{m_n} \leq t_1, \dots, \frac{C_k^n}{m_n} \leq t_k, \frac{J_1^n}{m_n} \leq s_1, \dots, \frac{J_{k-1}^n}{m_n} \leq s_{k-1} \given\Big \hat{\bm D} = \xi^*}\\
    \overset d\to \Prob{\eta_1 \leq t_1, \dots, \eta_k \leq t_k, \zeta_1 \leq s_1, \dots, \zeta_{k-1} \leq s_{k-1} \given\Big \sigma},
    \end{multline*}
where $\eta_i$ are the atoms of a Poisson process with intensity $\sigma$ and $\zeta_i$ are the attachment points defined conditionally given $\sigma$ and $\eta$ as before. Now we can introduce the measure change to recover the law of the true rescaled branch times: note that
    \begin{align*}
        &\Prob{\frac{C_1^n}{m_n} \leq t_1, \dots, \frac{C_k^n}{m_n} \leq t_k, \frac{J_1^n}{m_n} \leq s_1, \dots, \frac{J_{k-1}^n}{m_n} \leq s_{k-1}, N_n \geq \fl{Tm_n}}\\
        &= \E{\Prob{\frac{C_1^n}{m_n} \leq t_1, \dots, \frac{C_k^n}{m_n} \leq t_k, \frac{J_1^n}{m_n} \leq s_1, \dots, \frac{J_{k-1}^n}{m_n} \leq s_{k-1} \,\Big|\, \hat{\bm D} = \xi^*}\Theta^n_{\fl{Tm_n}}(\xi^*)}\\
        &\to \E{\Prob{\eta_1 \leq t_1, \dots, \eta_k \leq t_k, \zeta_1 \leq s_1, \dots, \zeta_{k-1} \leq s_{k-1} \given\Big \sigma}\exp\left(\int_0^t\sigma_s\,ds\right)\frac{p(-\sigma_t)}{p(0)}}\\
        &= \E{\Prob{\eta_1 \leq t_1, \dots, \eta_k \leq t_k, \zeta_1 \leq s_1, \dots, \zeta_{k-1} \leq s_{k-1} \given\Big \sigma = \widetilde \sigma}}\\
        &= \Prob{Y_1 \leq t_1, \dots, Y_k \leq t_k, Z_1 \leq s_1, \dots, Z_{k-1} \leq s_{k-1}},
    \end{align*}
where in the third line we have used the Skorokhod representation and the fact that the measure change converges in $L^1$ in this setting. Thus we have the following convergence in distribution:

\begin{theorem}\label{thm: fdmarginals}
    Let $C_1^n, C_2^n, \dots$ be the times of repeats in the codeword used to generate a Bienaymé tree conditioned to have $n$ vertices, and $J_1^n, J_2^n, \dots$ the first occurrence times of these repeats. Let $Y_1, Y_2, \dots$ be the jump times of a Poisson process of intensity $\widetilde \sigma$, where $\widetilde \sigma$ is the process introduced in \autoref{sec: lbalphastab}, and let $Z_1, Z_2, \dots$ be the corresponding attachment points. Then
        \[ \frac{1}{m_n}(C_1^n, \dots, C_k^n, J_1^n, \dots, J_{k-1}^n) \overset d\to (Y_1, \dots, Y_k, Z_1, \dots, Z_{k-1}). \]
\end{theorem}
\begin{proof}
    We have already shown that, for $t_1, \dots, t_k \leq T$ and $s_1, \dots, s_{k-1}$ with $s_i \leq t_i$, 
        \begin{multline*}
        \Prob{\frac{C_1^n}{m_n} \leq t_1, \dots, \frac{C_k^n}{m_n} \leq t_k, \frac{J_1^n}{m_n} \leq s_1, \dots, \frac{J_{k-1}^n}{m_n} \leq s_{k-1}, N^n \geq \fl{Tm_n}}\\
        \to \Prob{Y_1 \leq t_1, \dots, Y_k \leq t_k, Z_1 \leq s_1, \dots, Z_{k-1} \leq s_{k-1}} \end{multline*}
    and so it is enough to remove the $N^n \geq \fl{Tm_n}$ from the left-hand side (as then we have pointwise convergence of distribution functions). But we can do this since $\Prob{N^n < \fl{Tm_n}} \to 0$ by \autoref{lem: badeventrare}.
\end{proof}

\subsection{Gromov--Hausdorff--Prokhorov convergence of the trees}

We begin by converting the convergence in \autoref{thm: fdmarginals} into a statement about convergence of subtrees. We can do this as follows: let $T^n$ be the discrete rooted labelled tree of size $n$ as constructed from the algorithm in \autoref{sec: discrete}, and let $T^n(k)$ be the subtree spanned by the root and vertices labelled $1, 2, \dots, k$. On the event that labels $1, 2, \dots, k$ do not appear in the codeword before time $C^n_k$ (which we can show occurs with probability tending to 1), $T^n(k)$ precisely matches the tree constructed after time $C^n_k - 1$. Recall that $d^n$ denotes the graph distance on $V(T^n)$. The $\mathbb R$-tree corresponding to the rescaled discrete tree $(V(T^n(k)),\frac{1}{m_n}d^n|_{T^n(k)})$, is a.s.\ isometric to the $\R$-tree obtained from a line-breaking construction with cut times
    \[ y^n_i = \frac{1}{m_n}(C^n_i - 1), \quad 1 \leq i \leq k\]
and attachment times
    \[ z^n_i = \frac{1}{m_n}(J^n_i - 1), \quad 1 \leq i \leq k-1 \]
taken at time $y_k^n$. We endow $(V(T^n(k)),\frac{1}{m_n}d^n|_{V(T^n(k))})$ with the empirical measure $\mu^n_k = \frac{1}{k}\sum_{i=1}^k\delta_i$ on the vertices labelled $1, 2, \ldots, k$.

Analogously, we can consider the tree $\mathcal T$ constructed from the procedure in \autoref{sec: lbalphastab}, and recall that $\mathcal T(k) = \mathcal T_{Y_k}$ is the subtree obtained by terminating the process at time $Y_k$. Recall also that $\mu_k = \frac{1}{k} \sum_{i=1}^k \delta_{Y_i}$ and that, from (\ref{eq:asGHPconv}), we have
\[
\left(\mathcal T(k), d|_{\mathcal{T}(k)},\mu_k\right) \to (\mathcal{T},d,\mu)
\]
almost surely as $k \to \infty$.  

For each fixed $k$, we have that
    \[ (y_1^n, \dots, y_k^n, z_1^n, \dots, z_{k-1}^n) \overset d\to (Y_1, \dots, Y_k, Z_1, \dots, Z_{k-1}) \]
and it follows straightforwardly that
\begin{equation} \label{eq: fdds}
\left(V(T^n(k)),\frac{1}{m_n}d^n|_{V(T^n(k))}, \mu^{n}_k \right) \convdist \left(\mathcal T(k), d|_{\mathcal{T}(k)},\mu_k\right)
\end{equation}
for the Gromov--Hausdorff--Prokhorov distance. In other words, the random finite-dimensional distributions converge.

In order to complete the proof of \autoref{thm:scalinglimit}, we need a tightness argument to deduce that $(V(T^n),d^n/m_n,\mu^n) \overset d\to (\mathcal T,d,\mu)$ in the Gromov--Hausdorff--Prokhorov sense. We follow closely the strategy deployed by Haas and Miermont in Sections 4.3 and 4.4 of \cite{haasmiermbt}. That paper proves a collection of scaling limit theorems for Markov branching trees, a family of models which includes conditioned Galton--Watson trees. Indeed, Theorem 8 (Case 2) of \cite{haasmiermbt} is a version of \autoref{thm:scalinglimit} restricted to the special case of offspring distributions satisfying $\Prob{\xi = k} \sim c k^{-\alpha-1}$ as $k \to \infty$ for some $c > 0$. Haas and Miermont's method for showing convergence of the random finite-dimensional distributions makes use of the theory of fragmentation processes, and is completely different to ours. Their tightness proof, however, adapts well to our situation. Since we wish this paper to be self-contained, we reprove various estimates used in \cite{haasmiermbt} via our line-breaking approach rather than appealing to the fragmentation perspective.

As argued in Section 4.4 of \cite{haasmiermbt}, by Proposition 2.4 of \cite{EvansWinter}, the convergence in distribution $(V(T^n), d^n/m_n) \convdist (\Tcal,d)$ in the Gromov--Hausdorff sense entails that $(V(T^n), d^n/m_n,\mu_n)$ is a relatively compact sequence in the Gromov--Hausdorff--Prokhorov sense. Therefore, it suffices to identify any subsequential limit as $(\Tcal,d,\mu)$. But this then follows from the convergence of the random finite-dimensional distributions (\ref{eq: fdds}). Therefore, in what follows we may ignore the measures and just prove tightness for the Gromov--Hausdorff distance.

In this subsection, let us abuse notation and write $T^n$ in place of $(V(T^n), d^n)$ and $T^n(k)$ in place of $(V(T^n(k)),d^n|_{V(T^n(k))})$. By the principle of accompanying laws (see e.g.\ Theorem 3.2 of \cite{Billingsley}), it is enough to show
    \[ \lim_{k \to \infty} \limsup_{n \to \infty} \Prob{d_{\mathrm{GH}}(T^n(k),T^n) > \varepsilon m_n} = 0\]
for any $\varepsilon > 0$. Note that $T^n \setminus T^n(k)$ is a forest: indeed, conditionally on the number of vertices in $T^n(k)$ and the number of edges from $T^n(k)$ to its complement, it is a forest of i.i.d.\ Bienaym\'e trees conditioned on its total size and number of trees. Moreover, $d_{\mathrm{GH}}\left(T^n(k), T^n\right)$ is at most the maximum height of any individual tree in this forest. Write $F_1^n(k), F_2^n(k), \ldots$ for the sizes of the trees in the forest, and write $H(m)$ for the height of a Bienaym\'e tree with the same offspring distribution conditioned to have size $m$. Then
    \begin{equation} \label{eq:Markovtailbound}
    \Prob{d_{\mathrm{GH}}\left(T^n(k), T^n\right) > \varepsilon m_n} \le \E{\sum_{i \ge 1} \Prob{H(F_i^n(k)) > \varepsilon m_n | F_i^n(k)}}.
    \end{equation}

We adapt Lemma 33 of \cite{haasmiermbt} to prove the following.

\begin{lemma_restate} \label{lem: HMtailbound}
For each $q > 0$, there exists a universal constant $M_{q} > 0$ such that
\[
\Prob{H(n) \ge x m_n} \le \frac{M_{q}}{x^{q}}
\]
for all $x > 0$ and all $n \ge 1$.
\end{lemma_restate}

\begin{proof}
Lemma 33 of \cite{haasmiermbt} makes use of a hypothesis $(H')$, which Haas and Miermont only prove holds in the restricted case $\Prob{\xi = k} \sim c k^{-\alpha-1}$ as $k \to \infty$ for some $c > 0$. However, all that is really needed in the proof is that for fixed $\gamma > 1$,
\begin{equation} \label{eq: replaceshypothesisH'}
 \liminf_{n \to \infty} m_n \E{1 - \sum_{i=1}^{\hat{D}_1} \left(\frac{A_i}{n}\right)^{\gamma}} > 0, 
\end{equation}
where $A_1, A_2, \ldots$ are the sizes of the subtrees from the children of the root. We will give a crude bounding argument to show that this condition also holds in our case. Writing $\beta = \gamma - 1 > 0$, we observe that
\[ m_n \E{1 - \sum_{i=1}^{\hat{D}_1} \left(\frac{A_i}{n}\right)^{\gamma}} \geq m_n\E{1 - \left(\frac{A_*}{n}\right)^\beta},\]
where $A_*$ is a size-biased pick from $(A_i)$. Letting $F_1, F_2, \dots$ be i.i.d.\ random variables having the same distribution as the size of an unconditioned Bienaymé tree, and setting $S_n = \Xi_n - n$, where $(\Xi_n)_{n \geq 1}$ is the random walk with increments $\xi$ as before, we see that the right-hand side of the above is
\begin{align*}
    &m_n \sum_{d = 1}^{n - 1}\sum_{j = 1}^{n - 1} \left(1 - \left(\frac jn\right)^\beta\right)\Prob{\xi = d} \cdot d \cdot \frac{j}{n-1} \frac{\Prob{F_1 = j}\Prob{F_1 + \dots + F_{d-1} = n - 1 - j}}{\Prob{F_1 + \dots + F_d = n - 1}}\\
    &= m_n\sum_{j=1}^{n-1}\sum_{d = 1}^{n-1} \left(1 - \left(\frac jn\right)^\beta\right) \frac{(d-1)\Prob{\xi = d}}{n-1-j}\frac{\Prob{S_j = -1}\Prob{S_{n-1-j} = -(d-1)}}{\Prob{S_{n-1} = -d}},
\end{align*}
using the cycle lemma. For the sake of obtaining a lower bound, we may restrict our attention to $j \in [\frac n3, \frac{2n}{3}]$ and $d \in [a_n, 2a_n]$. Here, the term $1 - \left(\frac jn\right)^\beta$ is uniformly bounded away from 0, and $n - 1 - j = \Theta(n)$, so for sufficiently large $n$ we have that (\ref{eq: replaceshypothesisH'}) is bounded below by a constant multiple of
\begin{align*}
     &\quad\,\, m_n\sum_{j = n/3}^{2n/3} \sum_{d = a_n}^{2a_n} \frac 1n (d-1)\Prob{\xi = d} \frac{\Prob{S_j = -1}\Prob{S_{n-1-j} = -(d-1)}}{\Prob{S_{n-1} = -d}}\\ 
    &= m_n\sum_{j}\frac 1n\E{(\xi - 1)\frac{\Prob{S_j = -1}\Prob{S_{n-1-j} = -(\xi-1) \mid \xi}}{\Prob{S_{n-1} = -\xi \mid \xi}}\IndEvent{a_n \leq \xi \leq 2a_n}}\\
    &= \sum_{j} \frac{m_n a_{n-1}}{na_ja_{n-1-j}}\E{(\xi - 1)\frac{a_j\Prob{S_j = -1} a_{n-1-j}\Prob{S_{n-1-j} = -(\xi-1) \mid \xi}}{a_{n-1}\Prob{S_{n-1} = -\xi \mid \xi}}\IndEvent{a_n \leq \xi \leq 2a_n}}\\
    &= \Omega(1)\sum_{j} \frac{m_n}{na_n}\E{(\xi - 1)\frac{a_j\Prob{S_j = -1} a_{n-1-j}\Prob{S_{n-1-j} = -(\xi-1) \mid \xi}}{a_{n-1}\Prob{S_{n-1} = -\xi \mid \xi}}\IndEvent{a_n \leq \xi \leq 2a_n}}.
\end{align*}
For any fixed $\delta > 0$, three applications of the local limit theorem yield that for sufficently large $n$, (\ref{eq: replaceshypothesisH'}) is bounded below by
\[ \Omega(1)\sum_{j}\frac{m_n}{na_n} \E{(\xi - 1)\frac{\left(p(0) - \delta\right)\left(p\left(\frac{-(\xi - 1)}{a_{n-1-j}}\right) - \delta\right)}{p\left(\frac{-\xi}{a_{n-1}}\right) + \delta}\IndEvent{a_n \leq \xi \leq 2a_n}}.\]
As $(a_n)$ is regularly varying, we can find a constant $R$ such that $a_{n-1-j} \geq \frac{a_n}{R}$ whenever $n$ is sufficiently large and $j \in [\frac n3, \frac{2n}{3}]$. For such $n$, we observe that on the event $a_n \leq \xi \leq 2a_n$ we have $\frac{\xi - 1}{a_{n-1-j}} \in [0, 2R]$ and $\frac{\xi}{a_{n-1}} \in [0, 2R]$. Thus taking $\delta < \inf_{x \in [0, 2R]}p(-x)$, we obtain the existence of a constant $\eta > 0$ such that
    \[ \frac{\left(p(0) - \delta\right)\left(p\left(\frac{-(\xi - 1)}{a_{n-1-j}}\right) - \delta\right)}{p\left(\frac{-\xi}{a_{n-1}}\right) + \delta} \geq \eta \]
whenever $a_n \leq \xi \leq 2a_n$. Thus we obtain that (\ref{eq: replaceshypothesisH'}) is bounded below by
\begin{align*}
    \Omega(1) \sum_{j=n/3}^{2n/3} \frac{m_n}{na_n}\E{(\xi - 1)\eta \IndEvent{a_n \leq \xi \leq 2a_n}} &= \Omega(1) \sum_{j=n/3}^{2n/3} \frac{m_n}{n}\Prob{a_n \leq \xi \leq 2a_n}\\
    &= \Omega(1) \times m_n\Prob{\frac{\xi}{a_n} \in [1, 2]}.
\end{align*}
By standard theory of Lévy processes we know $m_n\Prob{\frac{\xi}{a_n} \in [1, 2]} \to \nu_{\alpha}([1, 2])$, where $\nu_\alpha$ is the Lévy measure of the $\alpha$-stable Lévy process $L$. In particular, our lower bound has a strictly positive limit.

The proof now follows as in Lemma 33 of \cite{haasmiermbt}. We provide a sketch of the argument here: first, without loss of generality we can focus on large values of $q$, specifically $q > \frac{\alpha}{\alpha - 1}$. Pick some $\varepsilon > 0$ such that $\gamma := q(\frac{\alpha - 1}{\alpha} - \varepsilon) > 1$. Now we can find a new sequence $\widetilde m_n \sim m_n$ such that
    \[ \widetilde m_n \E{1 - \sum_{i = 1}^{\widehat{D_1}} \left(\frac{A_i}{n}\right)^\gamma} \geq \tilde M \text{ for all } n \quad \text{ and } \quad \frac{\widetilde m_k}{\widetilde m_n} \leq \left(\frac kn\right)^{\frac{\alpha - 1}{\alpha} - \varepsilon} \text{ for all } 1 \leq k \leq n, n \geq N \]
for some fixed $\tilde M$ and $N$. Now pick $M_q$ large enough that the following hold:
\begin{itemize}
    \item $(1-u)^{-q} \leq 1 + 2qu$ for all $u \in [0, M_q^{-q}]$;
    \item $M_q \geq (2q/\tilde M)^q$;
    \item $M_q \geq \E{H(n)^q}/{\widetilde m_n}^q$ for all $n \leq N$.
\end{itemize}
We prove the result via a double induction: we prove the main result by induction on $n$, and to prove the result for a given $n$, we prove by induction on $i \geq 1$ that
    \[ \Prob{H(n) < x\widetilde m_n} \geq 1 - \frac{M_q}{x^q} \text{ for all } x \in \left(0, \frac{i}{\widetilde m_n}\right). \]
Markov's inequality handles the case $n \leq N$ (for all $i$), and the right-hand side is negative for small enough values of $x$ (and $i$), so we only need to consider $n > N, x > M_q^{1/q}$. Since $H(n)$ is 1 plus the maximum height across all subtrees rooted at children of the root,
\begin{align*}
    \Prob{H(n) < x\widetilde m_n} &= \E{\prod_{i = 1}^{\widehat D_1} \Prob{H(A_i) < x\widetilde m_n - 1 \given\big A_i}}\\
        &\geq \E{\prod_{i = 1}^{\widehat D_1}\left(1 - \frac{M_q \widetilde m_{A_i}^q}{(x\widetilde m_n - 1)^q}\right)^+}\\
        &\geq \E{1 - \sum_{i = 1}^{\widehat D_1}\frac{M_q \widetilde m_{A_i}^q}{(x\widetilde m_n - 1)^q}}.
\end{align*}
Note that the first inequality is valid by our inductive hypotheses: on the event that $(A_i)$ is a non-trivial partition, we can use the inductive hypothesis for $n$, and on the event that it is the trivial partition, we can use that $x\widetilde m_n - 1 < i - 1$, so the inductive hypothesis for $i$ is valid. Now we can use the bounds $\frac{1}{(x\widetilde m_n -1)^q} \leq \frac{1}{(x\widetilde m_n)^q}\left(1 + \frac{2q}{x \widetilde m_n}\right)$ (valid as $x > M_q^{1/q}$) and $\left(\frac{\widetilde m_{A_i}}{\widetilde m_n}\right)^q \leq \left(\frac{A_i}{n}\right)^{\gamma}$ (valid as $n \geq N$). Some further algebra implies the required bound.
\end{proof}

We continue to proceed exactly as in \cite{haasmiermbt}. Note that since $(m_n)_{n \ge 1}$ is regularly varying of index $1-1/\alpha$, we have that $(m_n^{2\alpha/(\alpha-1)}n^{-1})_{n \ge 1}$ is regularly varying of index 1 and so we can find a constant $C$ such that 
\[
\frac{m_k^{2\alpha/(\alpha-1)}k^{-1}}{m_n^{2\alpha/(\alpha-1)}n^{-1}} \le C\sqrt{\frac{k}{n}}
\]
for all $1 \le k \le n$. Combining (\ref{eq:Markovtailbound}) and \autoref{lem: HMtailbound}, and taking $q=2\alpha/(\alpha-1)$, we obtain that for any $n$ and $k$,
\begin{align*}
    \Prob{d_{\mathrm{GH}}\left(T^n(k), T^n\right) > \varepsilon m_n} & \le \frac{M_q}{\varepsilon^{q}} \E{\sum_{i \ge 1} \left(\frac{m_{F_i^n(k)}}{m_n}\right)^{2\alpha/(\alpha-1)} } \\
    &\le  \frac{M_q}{\varepsilon^{q}} \E{\sum_{i\ge 1} \frac{F_i^n(k)}{n-|T^n(k)|} \frac{m_{F_i^n(k)}^{2\alpha/(\alpha-1)} /F_i^n(k)}{m_n^{2\alpha/(\alpha-1)}/n}} \\
    & \le \frac{C M_q}{\varepsilon^{q}} \E{\sum_{i\ge 1} \frac{F_i^n(k)}{n-|T^n(k)|} \sqrt{\frac{F_i^n(k)}{n}}} \\
    & \le R(\varepsilon) \E{\sqrt{\frac{F_*^n(k)}{n}}},
\end{align*}
where $R(\varepsilon)$ is a constant depending only on $\varepsilon$ and $F_*^n(k)$ is a size-biased pick from the sizes $(F_1^n(k), F_2^n(k), \ldots)$ of the trees in $T^n \setminus T^n(k)$. The key result to prove is the following:

\begin{proposition_restate}\label{prop: tightnessprop}
    For each fixed $k$, we have $F_*^n(k)/n \convdist F_*(k)$ as $n \to \infty$, where $F_*(k)$ is a random variable stochastically dominated by the $\mathrm{Beta}(1-1/\alpha,k)$ distribution and so, in particular, $F_*(k) \convdist 0$ as $k \to \infty$.
\end{proposition_restate}

Before we prove this proposition, we note that it is sufficient to deduce the desired tightness, as we have
\begin{align*}
    \lim_{k \to \infty} \limsup_{n \to \infty} \Prob{d_{GH}\left(T^n(k), T^n\right) > \varepsilon m_n} &\leq R(\varepsilon) \lim_{k \to \infty} \limsup_{n \to \infty} \E{\sqrt{F_*^n(k)/n}\,}\\
    &= R(\varepsilon) \lim_{k \to \infty} \E{\sqrt{F_*(k)}} = 0.
\end{align*}

\begin{proof}[Proof of \autoref{prop: tightnessprop}]
In what follows, let $L = |T^n(k)|$ be the number of vertices in $T^n(k)$, and let $N$ be the number of edges from $T^n(k)$ to $T^n \setminus T^n(k)$. One can show that conditionally given $L$ and $N$, the forest $T^n \setminus T^n(k)$ has the law of $N$ i.i.d.\ Bienaymé trees (with offspring distribution $\xi$) conditioned to have total size $n - L$.

We study the conditional probabilities $\Prob{F_*^n(k) \geq \fl{xn} \given\big L, N}$ for fixed $x > 0$. This probability is zero whenever $x \geq 1$, so we can consider $x \in (0, 1)$. Let $F_1, F_2, \dots$ be i.i.d.\ random variables, also independent of $L$ and $N$, distributed as the size of an unconditioned Bienaymé tree with offspring distribution $\xi$. Then 
    \begin{align*} 
    & \Prob{F_*^n(k) \geq \fl{xn} \given\big L, N} \\
    & \qquad = \sum_{j=\fl{xn}}^n\! \frac{\frac{j}{n-L}N\Prob{F_1 = j \given\big L, N}\Prob{F_2 + \dots + F_N = n - L - j \given\big\! L, N}}{\Prob{F_1 + \dots + F_N = n - L \given\big L, N}}.
    \end{align*}
Taking $S_n = \Xi_n - n$ as before, independently of $L$ and $N$, we know by the cycle lemma that, for any fixed $a, b \geq 1$,
    \[ \Prob{F_1 + \dots + F_a = b} = \frac{a}{b}\Prob{S_b = -a}.\]
Using this, we have
    \[ \Prob{F_*^n(k) \geq \fl{xn} \given\big L, N} = \sum_{j=\fl{xn}}^n \frac{\Prob{S_j = -1}\Prob{F_1 + \dots + F_{N-1} = n - L - j \given\big L, N}}{\Prob{S_{n-L} = -N \given\big L, N}}.\]
We now wish to apply local limit theorems to this expression. We already have a local limit for $(S_n)$, namely
    \[ a_n \Prob{S_n = r} = p\left(\frac{r}{a_n}\right) + o(1) \]
as $n \to \infty$, uniformly in $r$, but by Theorem 4.2.1 of \cite{locallimits} we also have a local limit for sums of $F_i$: specifically, we can find a sequence $b_n$ such that $\frac{F_1 + \dots + F_n}{b_n}$ converges in distribution to a $(1/\alpha)$-stable subordinator at time $1$. Letting $q$ be the density of this limit law, we have
    \[ b_n \Prob{F_1 + \dots + F_n = r} = q\left(\frac{r}{b_n}\right) + o(1)\]
as $n \to \infty$, again uniformly in $r$. We note that $b_n$ is regularly varying of index $\alpha$, and we may further choose the sequence so that $b_{\fl{a_n}} \sim n$ as $n \to \infty$.

Guided by our limits, we write the conditional probability in the more suggestive format
    \[  \frac{1}{n}\sum_{j = \fl{xn}}^n \frac{a_{n-L}}{a_j}\cdot \frac{n}{b_{N-1}} \cdot\frac{a_j \Prob{S_j = -1} \cdot b_{N-1}\Prob{F_1 + \dots + F_{N-1} = n - L - j \given\big L, N}}{a_{n-L}\Prob{S_{n-L} = -N \given\big L, N}}\]
which, via the substitution $j = \fl{un}$, can in turn be expressed as the integral $\int_x^1 I_k^n(u)\,du$, where
    \[ I_k^n(u) = \frac{a_{n-L}}{a_{\fl{un}}}\cdot \frac{n}{b_{N-1}} \cdot \frac{a_{\fl{un}}\Prob{S_{\fl{un}}=-1} \cdot b_{N-1}\Prob{F_1 + \dots F_{N-1} = n-\fl{un}-L \given\big L, N}}{a_{n-L}\Prob{S_{n-L} = -N \given\big L, N}}.\]
We seek an almost sure pointwise limit for $I_k^n$. Let $Y_k$ be the time of the $k$th cut in the continuous line-breaking construction. Then we know that
    \[ \frac{N}{a_n} \overset d\to \widetilde \sigma_{Y_k} \quad \text{ and } \quad \frac{L}{m_n} \overset d\to Y_k \]
as $n \to \infty$ (with $k$ fixed), and these convergences hold jointly. Let us work on a probability space where these convergences hold almost surely. Then, almost surely, for all $u \in (x, 1)$ the following hold:
\begin{itemize}
    \item $L \ll n$ as $n \to \infty$, so $n - L \sim n$. Thus $\frac{n-L}{\fl{un}} \to u^{-1}$;
    \item Since the numerator and denominator of $\frac{n-L}{\fl{un}}$ both tend to $+\infty$ and $(a_n)$ is regularly varying of index $1/\alpha$ it follows that $\frac{a_{n-L}}{a_{\fl{un}}} \to u^{-1/\alpha}$;
    \item $\frac{a_n}{N-1} \to (\widetilde \sigma_{Y_k})^{-1}$. As $(b_n)$ is regularly varying of index $\alpha$ it follows that $\frac{b_{\fl{a_n}}}{b_{N-1}} \to (\widetilde \sigma_{Y_k})^{-\alpha}$, and so $\frac{n}{b_{N-1}}$ also converges to $(\widetilde \sigma_{Y_k})^{-\alpha}$.
    \item By the local limit results for $S$ and $F$, the fraction converges to 
        \[ \frac{p(0) \cdot q\left((1-u)/(\widetilde \sigma_{Y_k})^\alpha\right)}{p\left(-\widetilde \sigma_{Y_k}\right)}. \]
\end{itemize}
Combining everything, we obtain the pointwise limit
    \[ \lim_{n \to \infty} I_k^n(u) = u^{-\frac1\alpha}(\widetilde \sigma_{Y_k})^{-\alpha}\frac{p(0)}{p(-\widetilde \sigma_{Y_k})}q\left(\frac{1-u}{\widetilde \sigma_{Y_k}^\alpha}\right) =: \mathcal I_k(u).\]
By studying the above estimates more closely, we see that there exist constants $R = R(x) \leq x^{-\alpha} + 100$ and (random) $n_0 \in \mathbb N$ such that for all $n \geq n_0$ and all $u \in (x, 1)$ we have $I_k^n(u) \leq R \cdot \mathcal I_k(u)$. In particular, if we can show that $\int_x^1 \mathcal I_k(u) du < \infty$ almost surely, then by dominated convergence it will follow that
    \begin{equation}
    \Prob{F_*^n(k) \geq \fl{xn} \given\big L, N} \to \int_x^1 \mathcal I_k(u)\,du. \label{eqn:condtailprob}
    \end{equation}
To estimate $\mathcal I_k(u)$ we can use the duality relation (14.41) in \cite{satobook} between $p$ and $q$: we see that
    \[ \mathcal I_k(u) = u^{-\frac1\alpha}(1-u)^{-1-\frac1\alpha} \widetilde \sigma_{Y_k} \frac{p(0)\cdot p\left(-\frac{\widetilde \sigma_{Y_k}}{(1-u)^{\frac 1\alpha}}\right)}{p(-\widetilde \sigma_{Y_k})}, \quad u \in [x, 1)\]
and as $u \to 1$ we have $\mathcal I_k(u) \to 0$ almost surely (to see this use the stretched exponential decay estimate (\ref{eqn:satoasymptotic}) for $p$). Thus $\mathcal I_k$ is almost surely the restriction to $[x, 1)$ of a continuous function on $[x, 1]$ and so is bounded. We deduce that (\ref{eqn:condtailprob}) holds. Now we can use the bounded convergence theorem to eliminate the conditioning, and obtain
    \[ \Prob{F_*^n(k) \geq \fl{xn}} \to \E{\int_x^1 \mathcal I_k(u)\,du} = \int_x^1 \E{\mathcal I_k(u)}\,du, \quad x \in (0, 1).\]
By monotone convergence we also have $\int_0^1 \E{\mathcal I_k(u)}\,du \leq 1$, and so we may define a right-continuous non-increasing function $G_k : \R \to [0, 1]$ by
    \[ G_k(t) = \begin{cases}1 & t < 0\\ \int_t^1 \E{\mathcal I_k(u)}\,du & 0 \leq t < 1\\0 & t \geq 1.\end{cases}\]
It follows that we can construct a random variable $F_*(k)$ with c.d.f.\ $1 - G_k$. We have
    \begin{equation} 
    \lim_{n \to \infty} \Prob{F_*^n(k) \geq \fl{xn}} = G_k(x) \label{eqn:tailprob}
    \end{equation}
for all $x \neq 0$. If $G_k(0) = 1$ then the convergence also holds for $x = 0$; if not, then $G_k$ is discontinuous at 0. It follows that (\ref{eqn:tailprob}) holds at all continuity points of $G_k$, and so $F_*^n(k)/n \overset d\to F_*(k)$.

It remains to identify the law of $F_*(k)$. We will show that it is a mixture of a point mass at $0$ (with weight $1 - G_k(0)$) and a $\text{Beta}(1-\frac 1\alpha, k)$ distribution (with weight $G_k(0)$). If this is the case then 
    \[
    \E{F_*(k)} = G_k(0) \cdot \frac{1-1/\alpha}{k+1-1/\alpha} \leq \frac{1-1/\alpha}{k+1-\frac1\alpha} \to 0
    \]
as $k \to \infty$, and it follows immediately that $F_*(k) \convdist 0$ as $k \to \infty$.

We study the conditional law of $F_*(k)$ given that $F_*(k) > 0$. This has density function $\phi_k(x) = c\E{\mathcal{I}_k(x)}$, where $c = 1/G_k(t)$. Applying the measure change at the stopping time $Y_k$ allows us to write this density as
    \[ \phi_k(x) = cx^{-\frac1\alpha}(1-x)^{-1-\frac1\alpha}\E{\sigma_{Y_k}\exp\left(\int_0^{Y_k}\sigma_s\,ds\right)p\left(-\frac{\sigma_{Y_k}}{(1-x)^{\frac1\alpha}}\right)}. \]
Conditionally given $\sigma$, the stopping time $Y_k$ has density
    \[ \psi_k(t \given\big \sigma) = \frac{\sigma_t}{(k-1)!}\left(\int_0^t\sigma_s\,ds\right)^{k-1}\exp\left(-\int_0^t\sigma_s\,ds\right), \quad t \ge 0,\]
and so by the tower law (and Fubini's theorem) we obtain
    \[ \phi_k(x) = \frac{cx^{-\frac1\alpha}(1-x)^{-1-\frac1\alpha}}{(k-1)!}\int_0^\infty\E{\sigma_t^2 \left(\int_0^t \sigma_s\,ds\right)^{k-1}p\left(-\frac{\sigma_t}{(1-x)^{\frac1\alpha}}\right)}\,dt. \]
By the scaling property of the subordinator $\sigma$, we have the following equality in distribution (as processes):
    \[ \left(\sigma_t, t \geq 0\right) \overset d= \left((1-x)^{\frac1\alpha}\sigma_{t(1-x)^{-1+\frac1\alpha}}, t \geq 0\right).\]
Thus we may express the density $\phi_k(x)$ as
\begin{align*}
    &\frac{cx^{-\frac1\alpha}(1-x)^{-1+\frac1\alpha}}{(k-1)!} \! \! \int_0^\infty \! \! \E{\sigma_{t(1-x)^{-1+\frac1\alpha}}^2\left((1-x)^\frac1\alpha\int_0^t \! \sigma_{s(1-x)^{-1+\frac1\alpha}}\,ds\right)^{k-1} \!\!\! p\left(-\sigma_{t(1-x)^{-1+\frac1\alpha}}\right)}\,dt\\
    &= \frac{cx^{-\frac1\alpha}(1-x)^{-1+\frac1\alpha}}{(k-1)!} \! \! \int_0^\infty \! \! \E{\sigma_{t(1-x)^{-1+\frac1\alpha}}^2\left((1-x)\int_0^{t(1-x)^{-1+\frac1\alpha}} \! \! \! \! \! \sigma_{r}\,dr\right)^{k-1} \!\!\!\!\! p\left(-\sigma_{t(1-x)^{-1+\frac1\alpha}}\right)}\,dt\\
    &= \frac{cx^{-\frac1\alpha}(1-x)^{k-2+\frac1\alpha}}{(k-1)!} \! \! \int_0^\infty \! \! \E{\sigma_{t(1-x)^{-1+\frac1\alpha}}^2\left(\int_0^{t(1-x)^{-1+\frac1\alpha}}\sigma_{r}\,dr\right)^{k-1} \!\!\! p\left(-\sigma_{t(1-x)^{-1+\frac1\alpha}}\right)}\,dt\\
    &= \frac{x^{-\frac1\alpha}(1-x)^{k-1}}{(k-1)!} \cdot c\int_0^\infty \E{\sigma_u^2\left(\int_0^u\sigma_{r}\,dr\right)^{k-1} \!\!\! p\left(-\sigma_{u}\right)}\,du,
\end{align*}
after various changes of variable. The integral term in the last expression does not depend on $x$ and thus acts only as a normalising constant for the density. In particular, we can identify this density as that of the $\text{Beta}(1-\frac1\alpha, k)$ distribution.% and it also follows that
%    \[
%        \int_0^{\infty} \E{\sigma_u^2 \left(\int_0^u \sigma_r dr\right)^{k-1} p(-\sigma_u)} du = \frac{\Gamma(k+1-1/\alpha)}{\Gamma(1-1/\alpha)}. \qedhere
%    \]
\end{proof}

\begin{remark}
We expect that, in fact, $F_*(k)$ has no mass at zero and has precisely a $\text{Beta}(1-\frac 1\alpha, k)$ distribution, and should be interpreted as a size-biased pick from among the component sizes of the forest $\Tcal \setminus \Tcal(k)$. Indeed such a size-biased pick has distribution $\text{Beta}(1-\frac 1\alpha, k)$. To the best of our knowledge, this has not previously been explicitly identified in the literature for a general $k$, although it is easy to deduce it from pre-existing results, as we now show. In the case $k = 1$, the ranked sequence of component sizes of $\mathcal T \setminus \mathcal T(1)$ forms the \emph{fine spinal mass partition} defined in \cite{spinalpartition}. By Corollary 10 of the same paper, the sequence of sizes follows a Poisson--Dirichlet distribution with parameters $(\frac1\alpha, 1 - \frac1\alpha)$. (See Chapter 3 of Pitman~\cite{PitmanCSP} for more information about the Poisson--Dirichlet distributions.)

Now let $P^k = (P_1^k, P_2^k, \dots)$ be the ranked sequence of component sizes of $\mathcal T \setminus \mathcal T(k)$. Conditionally given $P^{k-1}$, we can sample $P^k$ as follows: we take a size-biased pick from the components of $\mathcal T \setminus \mathcal T(k-1)$, then reveal the fine spinal mass partition of this component. This has the effect of replacing some $P_I^{k-1}$ with the collection $\{ P_I^{k-1}\tilde P_i : i \in \mathbb N \}$, where $I$ is a size-biased index, $\tilde P$ is a copy of $P^1$ independent of everything else, and finally we reorder the sequence to preserve the decreasing order. By Theorem 3.1 in \cite{repeatedspinalpartitionlaw}, we see that if $P^{k-1} \sim \mathrm{PD}(\frac1\alpha, \theta)$ for some $\theta > 0$, then $P^k \sim \mathrm{PD}(\frac1\alpha, \theta + 1)$. It follows by induction that $P^k \sim \mathrm{PD}(\frac1\alpha, k - \frac1\alpha)$ for every $k \ge 1$. Now $F_*(k)$ is just a size-biased pick from $P^k$, and thus it has law $\text{Beta}(1-\frac1\alpha, k)$. 
\end{remark}

\section{Relation to Wang's construction}
In this section, we show that our construction provides a new perspective on Wang's construction of the stable tree from \cite{minmin}. The key result is \autoref{prop: sigmatildeidentified} below, which connects $\widetilde{\sigma}$ and the jumps of a normalised $\alpha$-stable excursion.

Suppose $\mathbf{x} = (x_1, x_2, \ldots)$ is a positive real sequence such that, for any $\varepsilon > 0$, the set $\{ i \in \mathbb N : x_i \geq \varepsilon \}$ is finite. Then we can rearrange the terms into decreasing order: write $(\mathbf{x})^{\downarrow}$ for this rearrangement. This notion clearly generalises to families $(x_\gamma : \gamma \in \Gamma)$ of positive real numbers indexed by arbitrary countable sets. Recall that $\Delta_1 \ge \Delta_2 \ge \dots \ge 0$ are the ordered jumps of $\widetilde{\sigma}$. Let $\Delta_1^n \ge \Delta_2^n \ge \dots \ge 0$ denote the ordered (out-)degrees in the finite-$n$ model, so that $(\Delta_1^n, \Delta_2^n, \ldots, \Delta_n^n) = (D_1, D_2, \ldots, D_n)^{\downarrow}$. For all $m \ge n+1$, let $\Delta^n_m = 0$ so that we append an infinite sequence of 0's to the end of this vector.

Let $\ell_2^{\downarrow} = \{(x_1, x_2, \ldots) \in \ell_2 : x_1 \ge x_2 \ldots \ge 0\}$ which we endow with the usual $\ell_2$ norm. For sequences indexed by an arbitrary countable set $\Gamma$ we instead write $\ell_2(\Gamma)$ for the corresponding set of square summable sequences, and $\ell_2^+(\Gamma)$ when all of the elements of the sequences are strictly positive. 

    \begin{proposition_restate} \label{prop: ell2conv}
        As $n \to \infty$,
        \[
        a_n^{-1} (\Delta_1^n, \Delta_2^n, \ldots) \convdist (\Delta_1, \Delta_2, \ldots)
        \]
        in $\ell_2^{\downarrow}$.
    \end{proposition_restate}

The proof of this result will make use of the theory of size-biased point processes developed by Aldous in Section 3.3 of \cite{aldouscritrg}. We need a little set-up.

Suppose that $\Gamma$ is a countable index-set and we have $\mathbf{Y} = \{Y_{\gamma}: \gamma \in \Gamma\}$ such that $\mathbf{Y} \in \ell_2^+(\Gamma)$. Conditionally on $\{Y_{\gamma}: \gamma \in \Gamma\}$, let $E_{\gamma}$ be exponentially distributed with parameter $Y_{\gamma}$, independently for different $\gamma \in \Gamma$. If the sequence $\mathbf{Y}$ is summable then it is straightforward to check that putting the values $Y_{\gamma}$ in increasing order of their associated exponential random variables has precisely the effect of putting them in size-biased random order. But the construction also makes sense when the sequence is not itself summable, as long as it is square-summable. For $t \in \R^+$, define
\[
S(t) = \sum_{\gamma \in \Gamma} Y_{\gamma} \IndEvent{E_{\gamma} \le t}.
\]
Then $S(t) < \infty$ a.s.\ for all $t$. Let $S_{\gamma} = S(E_{\gamma}-)$. The set $\Xi := \{(S_{\gamma},Y_{\gamma}): \gamma \in \Gamma\}$ is called  the \emph{size-biased point process} associated with $\mathbf{Y}$. This is an element of the set $\mathcal{M}$ of collections of points in $[0,\infty) \times (0,\infty)$ such that there are only finitely many points in compact rectangles of the form $[0,a] \times [\delta,1/\delta]$, $a > 0, \delta > 0$. We endow $\mathcal{M}$ with the topology of vague convergence of counting measures on $[0,\infty) \times (0,\infty)$.

Observe we may also construct the process $S(t)$ in terms of an i.i.d.\ sequence of $\Expo{1}$ random variables, say $\{E_\gamma' : \gamma \in \Gamma\}$, this time independent of $\mathbf Y$: then we may set
    \[ S(t) = \sum_{\gamma \in \Gamma} Y_\gamma \IndEvent{E_\gamma' \le Y_\gamma t}. \]
This is useful for coupling processes $S(t)$ across various choices of $\mathbf Y$.

 A convenient formulation of the result we shall use is given in Proposition 17 of Aldous and Limi\'c~\cite{aldouslimic}, which we reproduce below.

\begin{proposition_restate} \label{prop: aldouslimic}
Suppose that $\mathbf{Y}^{(n)} \in \ell_2^+(\Gamma^{(n)})$ and that $\Xi^{(n)}$ is the associated size-biased point process, for each $1 \le n < \infty$. Suppose further that
\[
\Xi^{(n)} \convdist \Xi^{(\infty)},
\]
where
\begin{enumerate}
\item $\sup\{s: (s,y) \in \Xi^{(\infty)} \text{ for some $y$}\} = \infty$ a.s.
\item if $(s,y) \in \Xi^{(\infty)}$ then $\sum_{\substack{(s',y') \in \Xi^{(\infty)}\\s' < s}} y' = s$ a.s.
\item $\max\{y: (s,y) \in \Xi^{(\infty)} \text{ for some } s > a\} \convprob 0$ as $a \to \infty$.
\end{enumerate}
Write $\mathbf{Y}^{\infty}$ for the projection of $\Xi^{(\infty)}$ onto its second co-ordinate. Then 
\[
(\mathbf{Y}^{(n)})^{\downarrow} \convdist (\mathbf{Y}^{(\infty)})^{\downarrow}
\]
in $\ell_2^{\downarrow}$.
\end{proposition_restate}

\begin{proof}[Proof of \autoref{prop: ell2conv}]
Let $\mathbf{Y}^{(n)} = \{a_n^{-1} D_i : 1 \le i \le n, D_i > 0\}$. By construction, the process $(S^{(n)}(t), t \ge 0)$ is a random time-change of $(a_n^{-1} \sum_{i=1}^{\fl{m_n t}} \widehat{D}_i, t \ge 0)$ and so, in particular, we have
\[
\Xi^{(n)} = \left\{ \left(\frac{1}{a_n} \sum_{i=1}^{j-1} \widehat{D}_i, \frac{1}{a_n} \widehat{D}_j \right): j \ge 1 \right\}.
\]
By \autoref{cor: cumuldeg}, we have 
\[
\left(\frac{1}{a_n} \sum_{i=1}^{\fl{m_n t}} \widehat{D}_i, t \ge 0\right) \convdist \widetilde{\sigma};
\]
we will use this to show that
\[
\Xi^{(n)} \convdist \Xi^{(\infty)} = \{(\widetilde{\sigma}_{t-}, \Delta\widetilde{\sigma}_t): \Delta \widetilde{\sigma}_t > 0, t \ge 0\}.
\]
To prove this, by a strengthening of Theorem 16.16 of \cite{kallenbergfmp} it is enough to prove convergence in distribution of $\Xi^{(n)}(U)$ to $\Xi^{(\infty)}(U)$ for any open, relatively compact set $U$ such that $\Xi^{(\infty)}(\partial U) = 0$ a.s. Proving this statement is routine but rather long, so we give only a brief sketch here: fix a set $U$ and let us work on a common probability space such that the convergence $(\frac{1}{a_n}\sum_{i=1}^{\fl{tm_n}}\widehat D_i, t \geq 0) \to \widetilde \sigma$ in $D([0, \infty))$ is almost sure. We now work on the (probability 1) event that this convergence does occur and that the limiting point measure puts no mass on the boundary.

Now suppose that $\Xi^{(\infty)}$ has precisely $K$ atoms $(s_1, y_1) \dots, (s_K, y_K) \in U$. That is, the process $\widetilde \sigma$ has jumps whose values go from $s_i$ to $s_i + y_i$ for each $i$, and that these are the only jumps for which $(\widetilde \sigma_{t-}, \Delta\widetilde \sigma_t) \in U$. Using the convergence of c\`adl\`ag processes we note that, for each $n$, we can identify $K$ jumps of the $n$th discrete process, say from $s_i^{(n)}$ to $s_i^{(n)} + y_i^{(n)}$, such that $s_i^{(n)} \to s_i$ and $y_i^{(n)} \to y_i$ for all $i$. (Note these jumps need not be distinct, but for sufficiently large $n$ they will be.) As $U$ is an open set, we conclude that for sufficiently large $n$, all points $(s_i^{(n)}, y_i^{(n)})$ lie in $U$ and thus $\Xi^{(n)}$ has at least $K$ atoms in $U$.

It remains to show that for all but finitely many $n$ these are the only atoms of $\Xi^{(n)}$ in $U$. Suppose not, then we can find a subsequence in $n$ along which there is an extra atom $(u^{(n)}, v^{(n)}) \in U$. Using relative compactness of $U$ it is straightforward to derive a contradiction: either we find another atom in the limiting measure in $U$ or on the boundary, or we find that $\Xi^{(n)}$ has a pair of atoms that both converge to a common point (which violates either right continuity or existence of left limits in $\widetilde \sigma$).

We now just need to check Conditions 1--3 of \autoref{prop: aldouslimic}. Since $\widetilde{\sigma}_t \to \infty$ as $t \to \infty$, 1 is clear. Condition 2 is a consequence of the fact that $\widetilde{\sigma}$ is a right-continuous increasing process which evolves only by jumps. Condition 3 follows immediately from the fact that the jumps of $\widetilde{\sigma}$ are a.s.\ square-summable, which was proved in \autoref{prop: quadraticvarfinite}.
\end{proof}

We will need one more proposition formalising an intuitive result about convergence of the processes $S^{(n)}$. The proof is deferred to the Appendix.

\begin{restatable}{proposition_restate}{sbppconvergence}\label{prop: sbppconvergence}
Let $\mathbf{X}^{(n)}, \mathbf{X}$ be random elements of $\ell_2^\downarrow = \ell_2^\downarrow(\mathbb N)$, not necessarily defined on the same probability space. Define, for $t \geq 0$,
    \[ S^{(n)}(t) = \sum_{i=1}^\infty X_i^{(n)}\IndEvent{E_i \leq X_i^{(n)}t} \quad \text{ and } \quad S(t) = \sum_{i=1}^\infty X_i \IndEvent{E_i \leq X_i t}, \]
where $E_i \sim \Expo{1}$ are i.i.d. and independent of $\mathbf{X}^{(n)}$ and $\mathbf{X}$. If $\mathbf{X}^{(n)} \convdist \mathbf{X}$, then we have $(S^{(n)}(t), t \ge 0) \convdist (S(t), t \ge 0)$ in the Skorokhod space $D([0, \infty))$. 
\end{restatable}

\begin{proposition_restate} \label{prop: sigmatildeidentified}
    Let $\Theta = (\Theta_1, \Theta_2, \dots)$ be the ordered jump sizes of a normalised $\alpha$-stable excursion. Then
    \[
    (\widetilde{\sigma}_t)_{t \ge 0} \eqdist \left( \sum_{i \ge 1} \Theta_i \IndEvent{E_i \le \Theta_i t} \right)_{t \ge 0},
    \]
where $E_1, E_2, \ldots$ are i.i.d.\ $\Expo{1}$ and independent of $\Theta$.
\end{proposition_restate}

\begin{proof}
    Define a new process $\rho$ as follows: first, generate a sequence $(\Delta_i)_{i \geq 1} \in \ell_2^\downarrow$ distributed as the law of the jumps of $\widetilde \sigma$ arranged into decreasing order. Independently of these, let $E_i \sim \Expo{1}$ be i.i.d.\ and set
        \[ \rho_t = \sum_{i=1}^\infty \Delta_i \IndEvent{E_i \leq \Delta_i t}. \]
    The main step is to show that $\rho$ and $\widetilde \sigma$ have the same law.
    
    In what follows, write $\widetilde \sigma^n(t)=\frac{1}{a_n}\sum_{i=1}^{\fl{tm_n}}\widehat D_i$, and note that $\widetilde \sigma^n \convdist \widetilde \sigma$. (This is not exactly the analogue of $\sigma^n$ defined in \autoref{prop: sizebiasedscalinglimit}, as we have not subtracted 1 from each term in the sum, but as $m_n \ll a_n$ this difference vanishes in the limit and the convergence in distribution still holds.) We will show that $\widetilde \sigma^n$ also converges in distribution to $\rho$ in the Skorokhod metric $D([0, T])$, where $T$ is an arbitrary time horizon. %(Note that we will not construct an instance of $\widetilde \sigma^n$ in this proof.) 

    By \autoref{prop: ell2conv} and Skorokhod's representation theorem we can construct $(\Delta_i^n)$ on the same space as $\rho$ such that $a_n^{-1}(\Delta_1^n, \Delta_2^n, \dots) \to (\Delta_1, \Delta_2, \dots)$ almost surely in $\ell^2$. By \autoref{prop: sbppconvergence} the processes
        \[ \rho^n(t) = \frac{1}{a_n}\sum_{i=1}^n \Delta_i^n \IndEvent{E_i \leq \frac{\Delta_i^n}{a_n}t}\]
    converge to $\rho$ in probability on the space $D([0, T])$. Moreover, the jump chain of $\rho^n$ has the same distribution as that of $\widetilde \sigma^n$. Let $0 = \tau_0^n < \tau_1^n < \tau_2^n < \dots < \tau_n^n$ be the jump times of $\rho^n$, and let $(\widehat \Delta^n)$ be the reordering of $(\Delta^n)$ induced by the jumps of $\rho^n$. We see that $\widetilde \sigma^n(t) \overset d= \rho^n(\lambda^n(t))$, where $\lambda^n$ is the minimal piecewise linear function such that $\lambda^n(\frac{i}{m_n}) = \tau_i^n$ and (say) $\lambda^n$ is constant beyond time $n/m_n$. As such, we can explicitly construct $\widetilde \sigma^n(t) := \rho^n(\lambda^n(t))$ on our common probability space. We will show that, with this construction, $\widetilde \sigma^n(t) \convprob \rho$ with respect to the Skorokhod metric.

    We aim to bound the Skorokhod distance $d_D = d_{D([0, T])}$ between $\widetilde \sigma^n$ and $\rho^n$: let $\widetilde \lambda^n$ be the minimal piecewise linear time change such that
        \[ \widetilde \lambda^n\left(\frac{i}{m_n}\right) = \tau_i^n \wedge T \quad \text{ for } i \leq \fl{Tm_n} \]
    and $\widetilde \lambda^n(T) = T$. We will show that
        \[ \sup_{t \in [0, T]}|\widetilde \lambda^n(t) - t| \convprob 0 \quad \text{ and } \quad \sup_{t \in [0, T]}|\widetilde \sigma^n - \rho^n \circ \widetilde \lambda^n| \convprob 0, \]
    from which we can deduce that $d_D(\widetilde \sigma^n, \rho^n) \convprob 0$. For the first limit, we note that we can bound the supremum by looking only at the endpoints of each piece. Moreover, for all $i \leq \fl{Tm_n}$ note that $\frac{i}{m_n}$ is closer to $\tau_i^n \wedge T$ than to $\tau_i^n$. Hence we can bound
        \[ \sup_{t \in [0, T]}|\widetilde \lambda^n(t) - t| \leq \sup_{i \leq \fl{Tm_n}} \left|\tau_i^n - \frac{i}{m_n}\right|. \]
    By identifying the conditional law of $\tau_i^n - \tau_{i-1}^n$, we see that the process defined by
        \[ Y_i = \tau_i^n - \frac{1}{m_n}\sum_{j=1}^{i}\left(1 - \frac 1n - \frac 1n \sum_{k=1}^{j-1} \widehat \Delta_k^n\right)^{-1}\]
    is a mean 0 martingale, both in its natural filtration $(\mathcal F_i)$ and the larger filtration $\widetilde{\mathcal F_i} = \sigma(\mathcal F_i, (\widehat \Delta_k^n)_{k \geq 1})$. We may bound
    \begin{multline} \label{eq: martingalebound}
    \Prob{\sup_i \left|\tau_i^n - \frac{i}{m_n}\right| > \varepsilon}\\
    \leq \Prob{\sup_{i}|Y_i| > \frac12\varepsilon} + \Prob{\sup_{i} \frac {1}{m_n}\left|\sum_{j=1}^i\left(1 - \frac 1n - \frac 1n \sum_{k=1}^{j-1} \widehat \Delta_k^n\right)^{-1} - i\right| > \frac12\varepsilon}.
    \end{multline}
    We bound the first term on the right-hand side by Doob's inequality and a technical condition as follows. Let $A_n$ be the event that $\sum_{k=1}^{\fl{Tm_n}}\widehat \Delta_k^n \leq n/2$. We can calculate
        \[ \Prob{\left\{\sup_{i} |Y_i| > \frac 12 \varepsilon\right\} \cap A_n} \leq \frac{16}{\varepsilon^2}\E{Y_{\fl{Tm_n}}^2\IndEvent{A_n}} \leq \frac{16}{\varepsilon^2}\sum_{i=1}^{\fl{Tm_n}}\E{(Y_i - Y_{i-1})^2\IndEvent{A_n}}, \]
    using that $Y$ is an $(\widetilde {\mathcal F_i})$-martingale to establish the first inequality. Now for each $i$ we can estimate
        \[ \E{(Y_i - Y_{i-1})^2\IndEvent{A_n}} \leq \E{(Y_i - Y_{i-1})^2\IndEvent{\sum_{k=1}^{i-1}\widehat \Delta_k^n \leq n/2}}. \]
    Conditionally given $(\widehat \Delta_k^n)_{k\leq i - 1}$, the increment $Y_i - Y_{i-1}$ has the law of a centered exponential random variable of rate $\beta_i = m_n\left(1 - \frac1n - \frac1n \sum_{k=1}^{i-1}\widehat \Delta_k^n\right)$. The conditional second moment of this is $\beta_i^{-2}$, and so by the tower law
    \begin{align*}
        \E{(Y_i - Y_{i-1})^2\IndEvent{\sum_{k=1}^{i-1}\widehat \Delta_k^n \leq n/2}} &= m_n^{-2}\E{\left(1 - \frac1n - \frac1n \sum_{k=1}^{i-1}\widehat \Delta_k^n\right)^{-2}\!\IndEvent{\sum_{k=1}^{i-1}\widehat \Delta_k^n \leq n/2}}\\
        &\leq m_n^{-2}\E{\left(1 - \frac 1n - \frac 1n \times \frac n2\right)^{-2}\IndEvent{\sum_{k=1}^{i-1}\widehat \Delta_k^n \leq n/2}}\\
        &\leq m_n^{-2}\left(\frac 12 - \frac 1n\right)^{-2} \leq 5m_n^{-2} \text{ for sufficiently large $n$. }
    \end{align*}
    Putting this into our initial estimate, we obtain the bound
        \[ \Prob{\left\{\sup_{i} |Y_i| > \frac 12 \varepsilon\right\} \cap A_n} \leq \frac 1{m_n} \times \frac{80T}{\varepsilon^2} \to 0. \]
    Finally we note that $\Prob{A_n} \to 1$, and thus $\Prob{\sup_i |Y_i| > \frac 12 \varepsilon} \to 0$ too.
        
    For the second term on the right-hand side of (\ref{eq: martingalebound}), we note that the term in the absolute value function is always non-negative, and the supremum is always attained at $i = \fl{Tm_n}$. Thus the probability is equal to
        \begin{equation} \label{eq: errorterm}
        \Prob{\sum_{j=1}^{\fl{Tm_n}} \left(\left(1 - \frac 1n - \frac 1n\sum_{k=1}^{j-1}\widehat \Delta_k^n\right)^{-1}-1\right) > \frac{m_n \varepsilon}{2}}. 
        \end{equation}
    On the event $A_n$ (introduced above) the estimate
        \[ \left(1 - \frac 1n - \frac 1n\sum_{k=1}^{j-1}\widehat \Delta_k^n\right)^{-1} \leq 1 + C\left(\frac 1n + \frac 1n \sum_{k=1}^{j-1}\widehat \Delta_k^n\right) \]
    holds for every $j$, where $C$ is some absolute constant. Working on this event yields that (\ref{eq: errorterm}) is bounded above by 
    \begin{multline*}
    \Prob{C\sum_{j=1}^{\fl{Tm_n}}\left(\frac 1n + \frac 1n \sum_{k=1}^{j-1}\widehat \Delta_k^n\right) > \frac{m_n \varepsilon}{2}} + \Prob{A_n^c}\\
        \leq \Prob{\frac1n \sum_{j=1}^{\fl{Tm_n}}\sum_{k=1}^{j-1}\widehat \Delta_k^n \geq m_n\left(\frac{\varepsilon}{2C} - \frac{T}{n}\right)} + o(1).
    \end{multline*}
    Since $\frac{1}{n}\sum_{j=1}^{\fl{Tm_n}}\sum_{k=1}^{j-1}\widehat \Delta_k^n$ converges in distribution to $\int_0^T \widetilde \sigma_s\,ds$, while $m_n\left(\frac{\varepsilon}{2C} - \frac{T}{n}\right) \to \infty$, this probability tends to zero. This establishes $\sup_{t \in [0, T]}|\widetilde \lambda^n(t) - t| \convprob 0$.

    It remains to bound the distance between $\rho^n(\widetilde \lambda^n(t))$ and $\widetilde \sigma^n(t) = \rho^n(\lambda^n(t))$. Observe that this distance is maximised at $t = T$ (we prove this by splitting into the cases $\tau_{\fl{Tm_n}}^n \leq T$ and $\tau_{\fl{Tm_n}}^n > T$), where we have
        \[ \widetilde \sigma^n(T) = \frac{1}{m_n}\sum_{i=1}^{\fl{Tm_n}} \widehat \Delta_i^n \quad \text{ and } \rho^n(\widetilde \lambda^n(T)) = \frac{1}{m_n}\sum_{i=1}^{I} \widehat \Delta_i^n, \]
    where $I = \sup\{i : \tau_i^n \leq T\}$. For any $\delta > 0$ the event $B_\delta^n = \{ I \in ((T-\delta)m_n, (T+\delta)m_n) \}$ has probability tending to 1, and on this event we have
        \[ |\rho^n(\widetilde \lambda^n(T)) - \widetilde \sigma^n(T)| \leq \frac{1}{m_n}\sum_{i = \fl{(T-\delta)m_n}}^{\fl{(T+\delta)m_n}} \widehat \Delta_i^n. \]
    As $n \to \infty$ the right-hand side converges in distribution to $\widetilde \sigma_{T + \delta} - \widetilde \sigma_{T - \delta}$. Thus, by picking $\delta$ suitably for each $\varepsilon$, we see that $\Prob{|\rho^n(\widetilde \lambda^n(T)) - \widetilde \sigma^n(T)| > \varepsilon} \to 0$. It follows that $d_D(\widetilde \sigma^n, \rho^n) \convprob 0$. We also know $d_D(\rho^n, \rho) \to 0$ almost surely, hence in probability and so $\widetilde \sigma^n \convprob \rho$ in the Skorokhod metric.

    We also know $\widetilde \sigma^n \convdist \widetilde \sigma$, so we identify $\widetilde \sigma_t \overset d= \sum_{i=1}^\infty \Delta_i \IndEvent{E_i \leq \Delta_i t}$. It remains only to identify $(\Delta_i)$. By \autoref{prop: ell2conv} we know that for any $k$ we have $a_n^{-1}(\Delta_1^n, \dots, \Delta_k^n) \convdist (\Delta_1, \dots, \Delta_k)$, so we will show the same sequence converges to $(\Theta_1, \dots, \Theta_k)$. Let $L^n : \{0, 1, \dots, n\} \to \mathbb Z$ be the \L{ukasiewicz} path of the discrete tree $T^n$ of size $n$, and construct a c\`adl\`ag function $e^n \in D([0, 1])$ by taking
        \[ e^n(t) = a_n^{-1}L^n(\fl{nt}). \]
    Then we have $e^n \convdist \mathbbm{e}$, where $\mathbbm{e}$ is a normalised $\alpha$-stable excursion (this result is implicit in \cite{duquesnealphalimit} or \cite{igorduquesneproof}). Now consider a map $\mathfrak J : D([0, 1]) \to \R^k$ given as follows: for any $f \in D([0, 1])$, we identify the $k$ largest upward jumps of $f$ and arrange them into decreasing order (appending zeros if there are fewer than $k$ such jumps). This gives a well-defined continuous map, and applying it to the above convergence gives
        \[ \mathfrak J(e^n) \convdist \mathfrak J(\mathbbm{e}). \]
    We identify $\mathfrak J(e^n) \overset d= a_n^{-1}((\Delta_1^n - 1)^+, \dots, (\Delta_k^n - 1)^+)$ and $\mathfrak J(\mathbbm{e}) \overset d= (\Theta_1, \dots, \Theta_k)$. It follows that $a_n^{-1}(\Delta_1^n, \dots, \Delta_k^n) \convdist (\Theta_1, \dots, \Theta_k)$ and so $(\Delta_1, \dots, \Delta_k) \overset d= (\Theta_1, \dots, \Theta_k)$ as distributions on $\R^k$. As $k$ was arbitrary, the result follows. 
\end{proof}

\appendix
\section{Appendix: omitted proofs}
\firstcuttimelaw*
\begin{proof}
Recall that if $M_1 \sim \mathrm{ML}(1-1/\alpha,1-1/\alpha)$ then 
    \[
    \E{M_1^k} = \frac{\Gamma(k+1) \Gamma(1-\tfrac{1}{\alpha})}{\Gamma((k+1)(1-\tfrac{1}{\alpha}))}
    \]
for $k \ge 1$ and (by Carleman's condition) the distribution is determined by its moments. So it is sufficient to show that $\E{Y_1^k} = \alpha^{-k} \E{M_1^k}$.

    In what follows, denote by $\Psi_\alpha$ the characteristic exponent of $L$ (the spectrally positive $\alpha$-stable Lévy process such that $L_1$ has density $p$) and by $\Psi_{\alpha-1}$ the characteristic exponent of $\sigma$. We will prove that the result holds for all $k \in (0, \infty)$ (including non-integral $k$). We begin by restricting to the case that $(k+1)\left(1 - \frac1\alpha\right) \not \in \mathbb Z$.

    By the measure change formula we have the identity
        \[ \Prob{Y_1 \geq t} = \E{e^{-\int_0^t \widetilde \sigma_s\,ds}} = \E{\frac{p(-\sigma_t)}{p(0)}}. \]
    For any $k > 0$ we can write
        \[ \E{Y_1^k} = \int_0^\infty kt^{k-1}\Prob{Y_1 \geq t}\,dt = \lim_{\varepsilon \downarrow 0} \int_0^\infty kt^{k-1}e^{-\varepsilon t} \Prob{Y_1 \geq t}\,dt, \]
    using monotone convergence to prove the limit on the right. For each fixed $\varepsilon > 0$, we have
\begin{align*}
    \int_0^\infty kt^{k-1}e^{-\varepsilon t}\Prob{Y_1 \geq t}\,dt &= \frac{1}{p(0)}\int_0^\infty kt^{k-1}e^{-\varepsilon t}\E{p(-\sigma_t)}\,dt\\
    &= \frac{1}{2\pi p(0)}\int_0^\infty kt^{k-1}e^{-\varepsilon t}\E{\int_\R e^{i\lambda \sigma_t}e^{-\Psi_\alpha(\lambda)}\,d\lambda}\,dt\\
    &= \frac{1}{2\pi p(0)}\int_0^\infty kt^{k-1}e^{-\varepsilon t}\left(\int_\R \E{e^{i\lambda \sigma_t}}e^{-\Psi_\alpha(\lambda)}\,d\lambda\right)\,dt\\
    &= \frac{1}{2\pi p(0)}\int_0^\infty kt^{k-1}e^{-\varepsilon t}\left(\int_\R e^{-t\Psi_{\alpha-1}(\lambda)}e^{-\Psi_\alpha(\lambda)}\,d\lambda\right)\,dt\\
    &= \frac{1}{2\pi p(0)}\int_\R e^{-\Psi_\alpha(\lambda)}\left(\int_0^\infty kt^{k-1}e^{-(\varepsilon + \Psi_{\alpha-1}(\lambda))t}\,dt\right)\,d\lambda\\
    &= \frac{1}{2\pi p(0)}\int_\R e^{-\Psi_\alpha(\lambda)} \cdot \left(\varepsilon + \Psi_{\alpha-1}(\lambda)\right)^{-k}k\Gamma(k)\,d\lambda\\
    &= \frac{\Gamma(k+1)}{2\pi p(0)}\int_\R \left(\varepsilon + \Psi_{\alpha-1}(\lambda)\right)^{-k}e^{-\Psi_\alpha(\lambda)}\,d\lambda.
\end{align*}
    Now if $j \geq 0$ is an integer such that $\alpha j - (\alpha - 1)k < -1$, we can show that (still for fixed $\varepsilon > 0$)
        \[ \int_\R (\varepsilon + \Psi_{\alpha-1}(\lambda))^k \Psi_\alpha(\lambda)^j\,d\lambda = 0. \]
    Indeed, by contour integration, for every $R > 0$ we have 
        \[ \int_\R (\varepsilon + \Psi_{\alpha - 1}(\lambda))^k \Psi_\alpha(\lambda)^j\,d\lambda = \int_\R (\varepsilon + \Psi_{\alpha-1}(\lambda + iR))^k \Psi_\alpha(\lambda + iR)^j\,d\lambda \]
    but the modulus of the right-hand side is at most a constant multiple of $\int_\R (R^2 + x^2)^{\frac{\alpha j - (\alpha - 1)k}{2}}\,dx$, which tends to 0 as $R \to \infty$.

    This condition on $j$ is met for $j = 0, 1, \dots, l-1$, where $l = \fl{(k+1)(1-\frac1\alpha)}$. It follows that
       \[ \E{Y_1^k} = \frac{\Gamma(k+1)}{2\pi p(0)}\lim_{\varepsilon \downarrow 0} \int_\R (\varepsilon + \Psi_{\alpha-1}(\lambda))^{-k}\left(e^{-\Psi_\alpha(\lambda)} - 1 + \Psi_\alpha(\lambda) - \dots - \frac{(-\Psi_\alpha(\lambda))^{l-1}}{(l-1)!}\right)\,d\lambda. \]
    We may now apply dominated convergence to obtain
        \[ \E{Y_1^k} = \frac{\Gamma(k+1)}{2\pi p(0)}\int_\R \Psi_{\alpha-1}(\lambda)^{-k}\left(e^{-\Psi_\alpha(\lambda)} - 1 + \Psi_\alpha(\lambda) - \dots - \frac{(-\Psi_\alpha(\lambda))^{l-1}}{(l-1)!}\right)\,d\lambda.\]
    Now we have the relations
        \[ \Psi_{\alpha - 1}(\lambda) = \alpha(-\Psi_\alpha(\lambda))^{1 - \frac1\alpha} \quad \text{ and } \quad -i\Psi_{\alpha - 1}(\lambda) = \frac{d}{d\lambda}\left(-\Psi_\alpha(\lambda)\right), \]
    which are valid for all $\lambda \in \R$. (Some care is required taking fractional powers of complex numbers. For example, given $a, b \in \R$ and $z \in \mathbb C$, we can only say $(z^a)^b = z^{ab}$ if $a \arg(z) \in (-\frac \pi2, \frac \pi2)$ or $b$ is an integer. Fortunately, the results above turn out to be true whenever $\lambda$ is real). Using these, we get
    \begin{align*}
    \E{Y_1^k} &= \frac{i\Gamma(k+1)}{2\pi p(0) \alpha^{k+1}}\int_\R (-\Psi_\alpha)^{-(k+1)(1-\frac 1\alpha)}\left(e^{-\Psi_\alpha} - 1 + \dots - \frac{(-\Psi_\alpha)^{l-1}}{(l-1)!}\right)\frac{d}{d\lambda}(-\Psi_\alpha)\,d\lambda\\
        &= \frac{i\Gamma(k+1)}{2\pi p(0)\alpha^{k+1}}\int_\gamma z^{-(k+1)(1-\frac1\alpha)}\left(e^z - 1 - z - \dots - \frac{z^{l-1}}{(l-1)!}\right)\,dz,
    \end{align*}
    where $\gamma : \R \to \mathbb C$ is the curve given by $\gamma(\lambda) = -\Psi_\alpha(\lambda) = (-i\lambda)^\alpha$.

    Now let $G(z) = z^{-(k+1)(1-\frac1\alpha)}\left(e^z - 1 - \dots - \frac{z^{l-1}}{(l-1)!}\right)$. We evaluate $\int_\gamma G(z)\,dz$ via contour integration as follows. Note that for each $R > 0$, $\gamma(-R)$ has negative real part and positive imaginary part, while $\gamma(R) = \overline{\gamma(-R)}$. We can replace $\gamma|_{[-R, R]}$ by the union of the following four pieces:
    \begin{itemize}
        \item The vertical line from $\gamma(-R)$ down to $\Re(\gamma(-R))$;
        \item The horizontal line from $\Re(\gamma(-R))$ to 0 (where we take $\arg(z) = \pi$ to deal with the branch cut for $z^{-(k+1)(1-\frac1\alpha)})$;
        \item The horizontal line from 0 to $\Re(\gamma(-R)) = \Re(\gamma(R))$ (where we take $\arg(z) = -\pi$ instead);
        \item The vertical line from $\Re(\gamma(R))$ down to $\gamma(R)$.
    \end{itemize}
    As $R \to 0$, the integrals along the vertical lines go to zero. Meanwhile, by standard integral expressions for $\Gamma(z)$ when $\Re(z) < 0$, we see that the two horizontal integrals each converge to
        \[ \pm e^{\mp i\pi (k+1)(1-\frac1\alpha)}\Gamma\left(1 - (k+1)\left(1 - \frac1\alpha\right)\right). \]
    In total, we see that
\begin{align*}
    \int_\gamma G(z)\,dz &= \left(e^{-i\pi(k+1)(1-\frac1\alpha)}-e^{i\pi(k+1)(1-\frac1\alpha)}\right)\Gamma\left(1-(k+1)\left(1-\frac1\alpha\right)\right)\\
    &= -2i\sin\left(\pi(k+1)\left(1-\frac1\alpha\right)\right)\Gamma\left(1-(k+1)\left(1-\frac1\alpha\right)\right)\\
    &= \frac{-2i\pi}{\Gamma\left((k+1)\left(1 - \frac1\alpha\right)\right)},
\end{align*}   
    using the reflection formula for the Gamma function to obtain the last line. Finally we have
\begin{equation}
    \E{Y_1^k} = \frac{i\Gamma(k+1)}{2\pi p(0)\alpha^{k+1}} \times \frac{-2i\pi}{\Gamma\left((k+1)\left(1-\frac1\alpha\right)\right)}= \alpha^{-k}\frac{\Gamma(k+1)}{\Gamma\left((k+1)\left(1-\frac1\alpha\right)\right)} \times \frac{1}{\alpha p(0)}. \label{eq: gammas}
\end{equation}
Recall that this holds whenever $(k+1)(1-\frac1\alpha)$ is not an integer --- the set of such $k$ is dense in $[0, \infty)$. The right-hand side of (\ref{eq: gammas}) makes sense for all $k \geq 0$ and is continuous in $k$. By dominated convergence the map $k \mapsto \E{Y_1^k}$ is continuous for $k \in [0, \infty)$. It follows that the left- and right-hand sides of (\ref{eq: gammas}) are equal for all $k \geq 0$. Using that $\E{Y_1^0} = 1$, we see that $\frac{1}{\alpha p(0)} = \Gamma(1 - \frac1\alpha)$, completing the proof.
\end{proof}

\recipemg*

\begin{proof}
    By condition (ii) we know $e^{-\Psi_L(\lambda)}$ is integrable, and so $L_1$ has a density by Fourier inversion, namely
        \[ p(x) = \frac{1}{2\pi}\int_\R e^{-i\lambda x}e^{-\Psi_L(\lambda)}\,d\lambda. \]
    As discussed in \autoref{sec: lbalphastab}, we will prove the equivalent condition that
        \[ \E{\exp\left(\int_0^t (c+\sigma_s)\,ds\right)p(-c-\sigma_t)\IndEvent{\sigma_t < \infty}} = p(-c).\]
    First observe that $\sigma$ is killed at rate $a$, independently of its other jumps, and so $\Prob{\sigma_t < \infty} = e^{-at}$. We can pull this factor out of the expectation and henceforth forget about the killing in $\sigma$: we aim to show that
        \begin{equation} e^{-at}\E{\exp\left(\int_0^t (c+\sigma_s)\,ds\right)p(-c-\sigma_t)} = p(-c). \label{eqn:dagger} 
        \end{equation}
    Using the Fourier representation of the density, the left-hand side of (\ref{eqn:dagger}) expands to
    \[ \frac{e^{-at}}{2\pi}\E{\int_\R\exp\left(i(c+\sigma_t)\lambda + \int_0^t(c+\sigma_s)\,ds\right)e^{-\Psi_L(\lambda)}\,d\lambda }.\]
    We would like to exchange the integral with the expectation --- this is valid provided the modulus of the integrand has finite double integral. To this end, we compute
    \begin{multline*}
        \E{\int_\R \left|\exp\left(i(c+\sigma_t)\lambda + \int_0^t (c+\sigma_s)\,ds\right)e^{-\Psi_L(\lambda)}\right|\,d\lambda}\\
        = \int_\R \left|e^{-\Psi_L(\lambda)}\right|\,d\lambda \times \E{\exp\left(\int_0^t (c+\sigma_s)\,ds\right)}.
    \end{multline*}
    We already know that the integral term is finite, so we just need to prove that the expectation is finite too. We can represent $\sigma_s = b^2s + J_s$, where $J$ is a pure jump subordinator. Then we have
    \[ \E{\exp\left(\int_0^t (c+\sigma_s)\,ds\right)} = \exp\left(ct + \frac 12b^2t^2\right)\E{\exp\left(\int_0^t J_s\,ds\right)}.\]
    We can apply Campbell's formula to the rightmost expectation: let $\Pi \subset \R^+ \times \R^+$ be the (atoms of) the Poisson process of jumps of $J$, so that
        \[ J_t = \sum_{\substack{(s, x) \in \Pi\\s \leq t}}x.\]
    Then $\int_0^t J_s\,ds = \sum_{s \leq t} (t-s)x$. It follows that
    \begin{align*}
        \E{\exp\left(\int_0^t J_s\,ds\right)} &= \exp\left(\int_0^\infty \int_0^t \left(e^{(t-s)x}-1\right)\,ds\,x\nu(dx)\right)\\
        &= \exp\left(\int_0^\infty \left(e^{tx}-1-tx\right)\,\nu(dx)\right).
    \end{align*}
    This quantity is finite by condition (i). Thus the exchange is valid, and so the left-hand side of (\ref{eqn:dagger}) is equal to
        \[ \frac{e^{-at}}{2\pi}\int_\R \E{\exp\left(i\lambda \sigma_t + \int_0^t \sigma_s\,ds\right)}\exp\left(c(t + i\lambda) - \Psi_L(\lambda)\right)\,d\lambda. \]
    Once again, we can decompose $\sigma_t = b^2t + J_t$, so that
    \[i\lambda \sigma_t + \int_0^t \sigma_s\,ds = \left(b^2i\lambda t + \frac 12b^2t^2\right) + \left(i\lambda J_t + \int_0^t J_s\,ds\right).\]
    Thus the LHS of (\ref{eqn:dagger}) further expands to
        \[ \frac{e^{-at}}{2\pi}\int_\R \E{\exp\left(i\lambda J_t + \int_0^t J_s\,ds\right)}\exp\left(b^2i\lambda t + \frac 12b^2t^2 + c(t+i\lambda) - \Psi_L(\lambda)\right)\,d\lambda.\]
    By another application of Campbell's formula, the inner expectation is
    \begin{align*}
        \E{\exp\left(i\lambda J_t + \int_0^t J_s\,ds\right)} &= \exp\left(\int_0^\infty \int_0^t \left(e^{i\lambda x + (t-s)x}-1\right)\,ds\,x\nu(dx)\right)\\
        &= \exp\left(\int_0^\infty \left(e^{i\lambda x}\left(\frac{e^{tx}-1}{x}\right)-t\right)\,x\nu(dx)\right)\\
        &= \exp\left(\int_0^\infty \left(e^{(t+i\lambda)x}-e^{i\lambda x} - tx\right)\,\nu(dx)\right),
    \end{align*}
    noting that all of these integrals are finite by condition (i).
    Thus, the left-hand side of (\ref{eqn:dagger}) can be written as
        \begin{equation*} \frac{1}{2\pi}\!\int_\R\! \exp\left(-at+\int_0^\infty \!\left(e^{(t+i\lambda)x}-e^{i\lambda x} - tx\right)\!\nu(dx) + b^2i\lambda t + \frac12b^2t^2 + c(t+i\lambda) - \Psi_L(\lambda)\right) d\lambda.\end{equation*}
    For brevity, let us denote by $H(t, \lambda)$ the expression inside the exponential. It is now useful to separate the drift and Brownian term in $\Psi_L$ from the (compensated) jumps: write 
        \[ \Psi_L(\lambda) = ai\lambda + \frac 12b^2\lambda^2 + \Psi_0(\lambda),\]
    where
        \[ \Psi_0(\lambda) = -\int_0^\infty (e^{i\lambda x} - 1 - i\lambda x)\,\nu(dx).\]
    Condition (i) implies that $\Psi_0$ has an analytic continuation to all of $\mathbb C$, given by the same formula. We now observe that
        \[ \int_0^\infty \left(e^{(t+i\lambda)x} - e^{i\lambda x} - tx\right)\,\nu(dx) = \Psi_0(\lambda) - \Psi_0(\lambda - it),\]
    and so we have
    \begin{align*}
        H(t, \lambda) &= -at + \Psi_0(\lambda) - \Psi_0(\lambda - it) + b^2i\lambda t + \frac12b^2t^2 + c(t+i\lambda) - \Psi_L(\lambda)\\
        &= -at - \Psi_0(\lambda - it) + b^2i\lambda t + \frac12b^2t^2 + c(t+i\lambda) - ai\lambda - \frac12b^2\lambda^2\\
        &= -ai(\lambda - it) - \frac 12b^2 (\lambda - it)^2 + ic(t - i\lambda) - \Psi_0(\lambda - it)\\
        &= ic(\lambda - it) - \Psi_L(\lambda - it).
    \end{align*}

    In summary, we have shown that
        \[ e^{-at}\E{\exp\left(\int_0^t (c+\sigma_s)\,ds\right)p(-c-\sigma_t)} = \frac{1}{2\pi}\int_\R \exp\left(ic(\lambda - it) - \Psi_L(\lambda - it)\right)\,d\lambda,\]
    and we would like to show that this common quantity is equal to $p(-c)$. By Fourier inversion we know
        \[ p(-c) = \frac 1{2\pi} \int_\R \exp(ic\lambda - \Psi_L(\lambda))\,d\lambda \]
    and so it remains only to prove that
        \begin{equation} 
        \int_\R \exp\left(ic(\lambda - it) - \Psi_L(\lambda - it)\right)\,d\lambda = \int_\R \exp\left(ic\lambda - \Psi_L(\lambda)\right)\,d\lambda. \label{eqn:star}
        \end{equation}
    To this end, set $G(z) = icz - \Psi_L(z)$, $z \in \mathbb C$. We can show that, for any closed polygonal contour $\gamma$, we have $\oint_\gamma G(z)\,dz = 0$ --- this is by Fubini's theorem, which is valid due to the condition that $\int_0^\infty (x \wedge x^2) \,\nu(dx) < \infty$. Thus by Morera's theorem $G$ is holomorphic on $\mathbb C$, so $\exp(G(z))$ is holomorphic too.

    Thus for any closed contour $\gamma$ in the complex plane, we have $\oint_\gamma \exp(G(z))\,dz = 0$. We consider the rectangular contour with vertices at $\pm R$, $\pm R - it$ and then take $R \to \infty$. Observe that
        \begin{multline*} \oint_\gamma \exp(G(z))\,dz = \int_{-R}^R \exp(ic(\lambda - it) - \Psi_L(\lambda - it))\,d\lambda - \int_{-R}^R \exp(ic\lambda - \Psi_L(\lambda))\,d\lambda\\
        + i\int_0^t \exp(G(R - is))\,ds - i\int_0^t \exp(G(-R-is))\,ds, 
        \end{multline*}
    and so we will obtain (\ref{eqn:star}) if we can show that the final two integrals tend to 0 as $R \to \infty$.

    To this end, we will show that the integrands of these two integrals tend to 0 uniformly: that is, for $z = \lambda - is$, $s \in [0, t]$ we have $\exp(G(z)) \to 0$ uniformly in $s$ as $|\lambda| \to \infty$. As $|\exp(G(z))| = \exp(\Realpart{G(z)})$ it suffices to show $\Realpart{G(z)} \to -\infty$ uniformly, and since $s$ is bounded we can further reduce this to showing that $\Realpart{\Psi_L(z)} \to +\infty$ uniformly. It is straightforward to calculate
        \[ \Realpart{\Psi_L(\lambda - is)} = as - \frac 12b^2s^2 + \frac 12\lambda^2 b^2 - \int_0^\infty\left(e^{sx}\cos(|\lambda|x) - 1 - sx\right)\,\nu(dx).\]
    Note that the term $as - \frac 12b^2s^2$ is bounded. We can decompose the integral as
    \begin{multline*}
        -\int_0^\infty (e^{sx}\cos(|\lambda|x) - 1 - sx)\,\nu(dx) \\= \int_0^\infty e^{sx}(1-\cos(|\lambda|x))\nu(dx)
        - \int_0^\infty(e^{sx}-1-sx)\,\nu(dx).
    \end{multline*}
    The second term is bounded by condition (i), while the first term is lower bounded by removing the $e^{sx}$ term. Hence we have
        \[ \Realpart{\Psi_L(\lambda - is)} \geq \frac 12\lambda^2b^2 + \int_0^\infty (1-\cos(|\lambda|x))\,\nu(dx) + O(1) \]
    and now we recognise that
        \[ \frac 12\lambda^2 b^2 + \int_0^\infty (1-\cos(|\lambda|x))\,\nu(dx) = \Realpart{\Psi_L(\lambda)} \to \infty,
        \]
    by condition (ii). This completes the proof.
\end{proof}

\polyaurncvg*
\begin{proof}
    Write $X_n = \frac{A_n}{M_n}$ and $Y_n = \frac{\#\{i \leq n : A_i \neq A_{i-1}\}}{n}$ for brevity, and define a filtration
        \[ \mathcal F_n = \sigma((M_j)_{j \geq 0}, (A_i)_{0 \leq i \leq n}). \]
    A straightforward calculation reveals that $X_n$ is an $(\mathcal F_n)$--martingale. We also have $X_n \in [0, 1]$ for all $n$, so the martingale is uniformly integrable and thus converges almost surely to some $X_\infty$. Now let
        \[ S_n = \sum_{i=1}^n \left(\IndEvent{A_i \neq A_{i-1}} - \frac{A_{i-1}}{M_{i-1}}\right). \]
    Another calculation reveals that $(S_n)_{n \geq 0}$ is also an $(\mathcal F_n)$--martingale. Since
        \[ Y_n = \frac{S_n}{n} + \frac{1}{n}\sum_{i=1}^n \frac{A_{i-1}}{M_{i-1}}, \]
    and we already know $\frac{1}{n}\sum_{i=1}^n \frac{A_{i-1}}{M_{i-1}} \to X_\infty$ almost surely, it suffices to show $\frac{S_n}{n} \to 0$.

    By the Azuma--Hoeffding inequality (noting that the increments of $S_n$ all have modulus at most 1) we obtain, for each $\varepsilon > 0$,
        \[ \Prob{|S_n| \geq \varepsilon n} \leq 2\exp\left(- \frac12 \varepsilon^2 n \right), \]
    which is summable in $n$. Hence, almost surely, we have $|S_n/n| < \varepsilon$ for all but finitely many $n$. Repeating for, say, each $\varepsilon = 1/k$, we see $S_n/n \to 0$ almost surely.
\end{proof}

\sbppconvergence*
\begin{proof}
    Note that it is sufficient to prove the convergence in $D([0,T])$ for arbitrary $T > 0$. So fix $T > 0$. We begin by handling the case that $\mathbf{X}^{(n)}$ and $\mathbf X$ are deterministic, say $\mathbf{X}^{(n)} = x^{(n)}$ and $\mathbf X = x$. We can construct $S^{(n)}$ and $S$ on the same probability space by using the same $E_i$ for all of them. In this setting we will prove that $S^{(n)} \to S$ in probability with respect to the Skorokhod metric on $[0, T]$. We introduce an intermediate sequence of processes $R^{(n)}$ defined by
    \[ R^{(n)}(t) = \sum_{i=1}^\infty x_i \IndEvent{E_i \leq x_i^{(n)}t}. \]
    Let $d_D$ be the Skorokhod distance on $[0, T]$. We will show that $d_D(S^{(n)}, R^{(n)}) \to 0$ in probability and $d_D(R^{(n)}, S) \to 0$ in probability.

    For the first of these statements, we will show convergence in probability with respect to the uniform norm (which implies the desired result). For any $t \in [0, T]$ we have
\begin{align*}
    |S^{(n)}(t) - R^{(n)}(t)| &\leq \sum_{i=1}^{\infty} |x_i^{(n)} - x_i|\IndEvent{E_i \leq x_i^{(n)}t}\\
    &\leq \sum_{i=1}^{\infty} |x_i^{(n)} - x_i|\IndEvent{E_i \leq x_i^{(n)}T}.
\end{align*}
This bound does not depend on $t$ and so 
\begin{align*}
    \E{\sup_{t \in [0, T]}|S^{(n)}(t) - R^{(n)}(t)|} &\leq \sum_{i=1}^\infty |x_i^{(n)} - x_i|\left(1 - e^{-x_i^{(n)}T}\right)\\
    &\leq T\sum_{i=1}^\infty |x_i^{(n)} - x_i| \cdot x_i^{(n)}\\
    &\leq T ||x^{(n)} - x||_2 \cdot ||x^{(n)}||_2.
\end{align*}
Now $x^{(n)}$ is bounded in $\ell_2$ and $x^{(n)} - x \to 0$ in $\ell_2$, so this expectation tends to zero. In particular, by Markov's inequality we have, for any $\varepsilon > 0$,
    \[ \Prob{\sup_{t \in [0, T]}|S^{(n)}(t) - R^{(n)}(t)| > \varepsilon} \to 0, \]
establishing the first convergence in probability.

For the second convergence, some more care is needed. Let $\varepsilon > 0$ and $\delta > 0$ be given. We will show that for sufficiently large $n$ we have $\Prob{d_D(R^{(n)}, S) > \varepsilon} < \delta$. We introduce a new notation $||\cdot||_{2, N}$ defined by
    \[ ||u||_{2, N} = \sqrt{\sum_{i = N+1}^\infty u_i^2}. \]
We can interpret this as the $\ell_2$-norm of the projection of $u$ onto the subspace of codimension $N$ spanned by the $i$th unit vectors for $i > N$. Note that for any $u \in \ell_2$ we have $||u||_{2, N} \to 0$ as $N \to \infty$. For any $N$ and any $n$ we can estimate
\begin{align*}
    &\E{\left|\sum_{i=N+1}^\infty x_i \left(\IndEvent{E_i \leq x_i^{(n)}T} + \IndEvent{E_i \leq x_i T}\right)\right|} \\
    &\qquad\qquad\leq \sum_{i=N+1}^\infty x_i \left(\Prob{E_i \leq x_i^{(n)} T} + \Prob{E_i \leq x_i T}\right)\\
    &\qquad\qquad\leq T\sum_{i=N+1}^\infty x_i(x_i^{(n)} + x_i)\\
    &\qquad\qquad\leq T\left(||x||_{2, N}^2 + ||x||_{2, N} \cdot ||x^{(n)}||_2\right)\\
    &\qquad\qquad\leq T||x||_{2, N}\left(||x||_2 + ||x^{(n)}||_2\right)\\
    &\qquad\qquad\leq C||x||_{2, N},
\end{align*}
where $C$ is an absolute constant (here we are using the fact that $x^{(n)}$ is bounded in $\ell_2$). This bound does not depend on $n$, and goes to zero as $N \to \infty$. By Markov's inequality we can choose $N$ large enough such that
    \[ \Prob{\left|\sum_{i=N+1}^\infty x_i \left(\IndEvent{E_i \leq x_i^{(n)}T} + \IndEvent{E_i \leq x_i T}\right)\right| \leq \varepsilon} \geq 1 - \frac 12\delta \]
for every $n \geq 1$. Let $A_1^{(n)}$ denote this event.

Having chosen $N$, we define truncated processes
    \[ \widetilde R^{(n)}(t) = \sum_{i = 1}^N x_i\IndEvent{E_i \leq x_i^{(n)}t} \quad \text{ and } \quad \widetilde S(t) = \sum_{i = 1}^N x_i \IndEvent{E_i \leq x_i t}.\]
Let $B$ be the event that $\widetilde S$ does not jump at time $T$, and note that this has probability 1. Let $K$ be the (random) number of times $\widetilde S$ jumps before time $T$, and let $\pi : [K] \to [N]$ be the unique injection such that these jumps occur at times
    \[ \frac{E_{\pi(1)}}{x_{\pi(1)}} < \frac{E_{\pi(2)}}{x_{\pi(2)}} < \dots < \frac{E_{\pi(K)}}{x_{\pi(K)}}. \]
Let $A_2^{(n)}$ be the event that $\widetilde R^{(n)}$ also has $K$ jumps before time $T$ and that these are given by
    \[ \frac{E_{\pi(1)}}{x_{\pi(1)}^{(n)}} < \frac{E_{\pi(2)}}{x_{\pi(2)}^{(n)}} < \dots < \frac{E_{\pi(K)}}{x_{\pi(K)}^{(n)}}, \]
using the same permutation $\pi$. We observe that on the full measure event $B$, the events $A_2^{(n)}$ necessarily occur for all sufficiently large $n$, and so $\Prob{A_2^{(n)}} \to 1.$ On the event $A_2^{(n)}$ we may construct a time change $\lambda^{(n)} : [0, T] \to [0, T]$ as follows: we take the minimal piecewise linear function such that $\lambda^{(n)}(0) = 0$, $\lambda^{(n)}(T) = T$ and
    \[ \lambda^{(n)}\left(\frac{E_{\pi(i)}}{x_{\pi(i)}}\right) = \frac{E_{\pi(i)}}{x_{\pi(i)}^{(n)}} \]
for $i = 1, 2, \dots, K$. Observe that
    \[ \sup_{t \in [0, T]}\left|\lambda^{(n)}(t) - t\right| \leq \max_{i \leq K}\left|\frac{E_{\pi(i)}}{x_{\pi(i)}} - \frac{E_{\pi(i)}}{x_{\pi(i)}^{(n)}}\right| \]
which converges to 0 almost surely (on the event that $\lambda^{(n)}$ is defined). Letting $A_3^{(n)}$ be the event that $\sup_{t \in [0, T]}\left| \lambda^{(n)}(t) - t\right| \leq \varepsilon$, we see that $\Prob{A_3^{(n)}} \to 1$. Note that $\widetilde R^{(n)}(\lambda^{(n)}(t)) = \widetilde S(t)$ for all $t \in [0, T]$. Now we can compute 
\begin{align*}
    &\sup_{t \in [0, T]} |R^{(n)}(\lambda^{(n)}(t)) - S(t)| \\
    &\qquad\qquad= \sup_{t \in [0, T]}\left|\sum_{i = 1}^\infty x_i\left(\IndEvent{E_i \leq x_i^{(n)}\lambda^{(n)}(t)} - \IndEvent{E_i \leq x_i t}\right)\right|\\
    &\qquad\qquad= \sup_{t \in [0, T]}\left|\widetilde R^{(n)}(\lambda^{(n)}(t)) - \widetilde S(t) + \sum_{i = N+1}^\infty x_i\left(\IndEvent{E_i \leq x_i^{(n)}\lambda^{(n)}(t)} - \IndEvent{E_i \leq x_i t}\right)\right|\\
    &\qquad\qquad= \sup_{t \in [0, T]} \left| \sum_{i = N+1}^\infty x_i\left(\IndEvent{E_i \leq x_i^{(n)}\lambda^{(n)}(t)} - \IndEvent{E_i \leq x_i t}\right) \right|\\
    &\qquad\qquad\leq \left| \sum_{i = N+1}^\infty x_i \left(\IndEvent{E_i \leq x_i^{(n)}T} + \IndEvent{E_i \leq x_i T}\right)\right|.
\end{align*}
On the event $A_1^{(n)}$ this quantity is at most $\varepsilon$. Thus on the event $A^{(n)} := A_1^{(n)} \cap A_2^{(n)} \cap A_3^{(n)}$ we have $d_D(R^{(n)}, S) \leq \varepsilon$. Note that $\liminf_{n \to \infty} \Prob{A^{(n)}} \geq 1 - \frac12 \delta$, and so for sufficiently large $n$ we have $\Prob{A^{(n)}} \geq 1 - \delta$. Thus $R^{(n)} \to S$ in probability in the Skorokhod metric.

This completes the proof in the case that $\mathbf{X}^{(n)}$ and $\mathbf X$ are deterministic. For the general case, we use Skorokhod's representation theorem again, this time to construct $\mathbf{X}^{(n)}$ and $\mathbf X$ on the same space such that $\mathbf{X}^{(n)} \to \mathbf X$ almost surely. Constructing our exponential random variables also on this space, independent of everything else, and using them to define all processes $S^{(n)}$ and $S$, we deduce from the previous case that for any fixed $\varepsilon > 0$ we have
    \[ \Prob{d_D(S^{(n)}, S) \geq \varepsilon \given\Big \mathbf{X}^{(n)}, \mathbf X} \to 0 \text{ almost surely.}\]
By the bounded convergence theorem we may take expectations of the above to see that
    \[ \Prob{d_D(S^{(n)}, S) \geq \varepsilon} \to 0.\]
As $\varepsilon$ was arbitrary we see that $S^{(n)} \to S$ in probability for this coupling. In particular, the laws of $S^{(n)}$ converge weakly to that of $S$.
\end{proof}

\section*{Acknowledgements}

We are grateful to Andreas Kyprianou for a very helpful suggestion in the early stages of the project, and to Minmin Wang for a discussion of \cite{minmin}. LH's work is supported by the EPSRC Centre for
Doctoral Training in the Mathematics of Random Systems: Analysis, Modelling, and Simulation (EP/S023925/1). Some of the work was carried out while CG
was in residence at the Simons Laufer Mathematical Sciences Institute in Berkeley, California, during the Spring 2025 semester, which LH also visited for a week with funding from the Institute. They gratefully acknowledge the financial support of the National Science Foundation under Grant No.\ DMS-1928930.

%now enable appendix numbering format and include any appendices

%next line adds the Bibliography to the contents page
\addcontentsline{toc}{chapter}{Bibliography}
%uncomment next line to change bibliography name to references
%\renewcommand{\bibname}{References}
\bibliography{refs}        %use a bibtex bibliography file refs.bib
\bibliographystyle{plain}  %use the plain bibliography style

\end{document}